\documentclass{cmslatex}
\usepackage{latexsym, amssymb, enumerate, amsmath}
\usepackage{graphicx}
    \graphicspath{{./pix/}}
\usepackage{color}
\usepackage{subfigure}
\usepackage[all]{xy}
\usepackage{cases}

\renewcommand{\theequation}{\arabic{section}.\arabic{equation}}

\newtheorem{Theorem}{Theorem}[section]
\newtheorem{Lemma}{Lemma}[section]
\newtheorem{Corollary}{Corollary}[section]

\begin{document}

\title{
Derivation and Analysis of Simplified Filters for Complex Dynamical Systems
}

\author{
Wonjung Lee
\thanks {
Mathematics Institute
and Centre for Predictive Modelling, School of Engineering, 
University of Warwick, 
U.K.
(W.Lee@warwick.ac.uk).
}
\and Andrew Stuart
\thanks {
Mathematics Institute, University of Warwick, 
U.K.
(A.M.Stuart@warwick.ac.uk).
}
}

\pagestyle{myheadings} 
\markboth{Derivation and Analysis of Simplified Filters for Complex Dynamical Systems}{Wonjung Lee and Andrew Stuart}
\maketitle

\begin{abstract}
Filtering is concerned with the sequential estimation of the state,
and uncertainties, of a Markovian system, given noisy observations. 
It is particularly difficult to achieve accurate filtering in
complex dynamical systems, such as those arising in turbulence,
in which effective low-dimensional representation of the desired
probability distribution is challenging. 
Nonetheless recent advances have shown considerable success in filtering 
based on certain carefully chosen simplifications 
of the underlying system, which 
allow closed form filters. This leads to filtering
algorithms with significant, but judiciously chosen, model error.
The purpose of this article is to analyze the effectiveness of these 
simplified filters, and to suggest modifications of them which 
lead to improved filtering in certain time-scale regimes. 
We employ a Markov switching process for the true signal underlying 
the data, rather than working with a fully resolved DNS PDE model. 
Such Markov switching models haven been demonstrated to
provide an excellent surrogate test-bed for the turbulent bursting 
phenomena which make filtering of complex physical models, such as
those arising in atmospheric sciences, so challenging.
\end{abstract}

\begin{keywords}
\smallskip
Sequential filtering, Bayesian statistics, Complex dynamical systems, 
Model error.

{\bf AMS subject classifications.}
60G35, 93E11, 94A12
\end{keywords}

\section{Introduction}
\label{sec:intro}

\subsection{Overview}
Filtering is concerned with the sequential updating of Markovian systems,
given noisy, partial observations of the system state 
\cite{law2014data,majda2012filtering,reich2015probabilistic}.
Due to the increasing prevalence of data in all areas of science and 
engineering, and due to 
the inherent complexity of physical models developed for the description of
many phenomena 
arising in science and engineering, the need for accurate and speedy 
filters is paramount. However in its full form filtering requires
the description of a time-evolving probability distribution on the 
system state, conditioned on data, which for many systems can be hard 
to represent in a computationally tractable way. This is a particular
challenge for the complex physical models arising in areas such
as atmospheric sciences 
\cite{kalnay2003atmospheric},
oceanography 
\cite{bennett1992inverse}
and
oil reservoir simulation 
\cite{oliver2008inverse}.
However a recent body of work by Majda and coworkers 
\cite{harlim2008filtering, majda2010mathematical, majda2012filtering,
castronovo2008mathematical,
keating2011new, branicki2013dynamic,
branicki2012filtering, sapsis2013statistically, harlim2013test}
has demonstrated the possibility of using drastic simplifications of 
the models for complex turbulent phenomena in order to construct effective
filters which are computationally tractable in real-time. The
underlying philosophy of this work is to replace the true underlying
Markovian model (often deterministic, but chaotic) with a simplified
stochastic model which captures the key physical phenomena 
at the statistical level yet
is amenable to closed form expressions for the purpose of filtering.
It is possible to interpret this work as providing an important step
towards the adoption of {\em physically informed machine learning},
going beyond traditional machine learning methodologies which
often attempt to build models from the data alone 
\cite{bishop2006pattern, murphy2012machine}.
The purpose of our work is to shed further light on this body of work,
through analysis, through the derivation of new methods in the same
spirit, and through careful numerical experiments.

In order to carry out this program we do not work with a full complex
model of turbulence for our true signal, but rather work with a  simple
switching stochastic model (SSM), a stochastic differential equation 
driven by a sign-alternating two-state Markov 
process \cite{mao2006stochastic,walter2006conditional}.  The system 
is either forced or dissipated depending on the sign of the driving 
signal, and as a consequence admits intermittent bursting phenomena,
similar to what is seen in real turbulent signals
\cite{zakharov1992kolmogorov, frisch1996turbulence, bohr2005dynamical}.
The use of this model as a simplified model for turbulent bursting, and 
demonstration of its effectiveness in this context, may be seen from 
the papers \cite{gershgorin2010test, gershgorin2010improving}.
This SSM, then, is viewed as the ``true'' Markov model
whose signals generate the data. Our objective is to find simplified
models, amenable to filtering, which capture the essential features of
the SSM. 
We now define the filtering problem and
outline the simplified models that we consider.

\subsection{The True Model and Model Error}
Consider an $\mathbb{R}^d$-valued Markov process 
$x(t)$ where $t \geq 0$. The process is hidden
and we only have access to $y_n$, $n \in \mathbb{N}$, 
which is a (partial) noisy observation of
$x_n \equiv x(nT)$ for some $T > 0$.
For $Y_{n}:=\{y_1,\cdots,y_{n}\}$ the key objective in probabilistic 
filtering is the sequential updating of $\mathbb{P}(x_n|Y_n)$
\cite{kushner1967approximations, 
jazwinski1970stochastic, 
anderson1979optimal, 
doucet2001sequential,
majda2012filtering}.

To perform filtering,
the standard approach adopted in large
scale geophysical applications is to alternate
the uncertainty propagation
$\mathbb{P}\left( x_{n-1} \vert Y_{n-1} \right)  \mapsto  \mathbb{P}\left( x_n \vert Y_{n-1}\right)$,
and the data acquisition
$\mathbb{P}\left( x_{n} \vert Y_{n-1} \right)  \mapsto \mathbb{P}\left( x_n \vert Y_{n}\right)$
in a sequential manner.
The former step
corresponds to probabilistic solution of
the governing equation for $x(t)$,
while the latter step is accomplished by 
Bayes' rule
$\mathbb{P}( x_n \vert Y_n )  \propto \mathbb{P}(x_n \vert Y_{n-1}) \mathbb{P}(y_n \vert x_n)$,
which asserts that
the posterior distribution is proportional to the product of the prior distribution and 
the likelihood (viewed as function of $x_n$).
Examples of Bayesian filters include 
the Kalman filter
\cite{Kalman60,
kalman1961new},
the extended Kalman filter 
\cite{gelb1974applied},
the ensemble Kalman filter
\cite{evensen2009data},
the particle filter
\cite{gordon1993novel}
and the Gaussian mixture filter
\cite{sorenson1971recursive,
chen2000mixture, 
stordal2011bridging}.

In this paper the true model $x_n$ underlying the data will be 
found from discrete time sampling of the
following switching stochastic model, or SSM for short:  
\begin{equation}
\begin{split}
\label{eq:fm}
\textbf{(SSM)}\qquad
\begin{cases}
du &= -\gamma u dt + \sigma_u dB_u \\
\gamma &\in  \,\,\,\{\gamma_{+} ,\gamma_{-}  \} 
\end{cases}
\end{split}
\end{equation}
where $\gamma(t)$
is a Markov process,
alternately
taking constant values of
$\{\gamma_{+}>0,\gamma_{-}<0\}$.
The distribution functions of the random variables
\begin{equation*}
\begin{split}
\tau^{\gamma_{+}} &= \inf \{ t: \gamma(t)=\gamma_{-} \vert \gamma(0)=\gamma_{+}\}\\
\tau^{\gamma_{-}} &= \inf \{ t: \gamma(t)=\gamma_{+} \vert \gamma(0)=\gamma_{-}\}
\end{split}
\end{equation*}
are given by
\begin{equation*}
\begin{split}
\mathbb{P}(\tau^{\gamma_{+}} < t) &= 1-e^{-\frac{\lambda_{+}}{\epsilon} t} \\
\mathbb{P}(\tau^{\gamma_{-}} < t) &= 1-e^{-\frac{\lambda_{-}}{\epsilon} t}
\end{split}
\end{equation*}
respectively.
The positive parameter $\epsilon$
determines the transition rates,
accounting for 
the time-scale separation 
between 
input signal
$\gamma$
and 
output response
$u$.
In case of small $\epsilon$,
there is rapid switching 
between $\gamma_{+}$ and $\gamma_{-}$.
On the other hand,
switching is a rare event when $\epsilon$ is large.
In the general notation above we have $x=(u,\gamma).$

For $x_n=(u_n,\gamma_n)=(u(nT),\gamma(nT))$ 
we assume the noisy observations are of the form
\begin{equation}
\label{eq:obs}
\qquad y_n = u_n + \eta_n, \qquad \eta_n \sim \mathcal{N}({0},R_n) 
\end{equation}
where
$\{\eta_n\}_{n \ge 0}$ is an 
independent and identically distributed 
centred Gaussian.
The filtering distribution 
$\mathbb{P}(x_n|Y_n)$,
determined by
(\ref{eq:fm}) and (\ref{eq:obs}),
does not allow for a closed-form representation.
In the following, 
we address the problem
through {\em filtering with model error}: that is,
instead of a straightforward application
to the genuine system, we 
replace the process $x$ by a different Markov model
which is more amenable to filtering explicitly than is the SSM.
We tune the parameters of the new models to 
maximize their statistical resemblances with the SSM.
It is important to note that in this paper,
due to the low dimensionality of SSM,
the introduction of reduced models used for filtering
presumably does not
lead to a
significant saving of computational costs.
However the aim is
to understand the application of the methodology developed
by Majda and coworkers which is targetted at situations 
where the true signal is very expensive to simulate, 
whilst the models used for filtering are orders of magnitude cheaper.
Furthermore we investigate
a new theory-based conceptual framework 
to illustrate this body of work, and to develop generalizations of it,
working in a simple setting where the true signal of interest
comes from the SSM.

There are four forms of filters with 
model error considered in this paper (acronyms explained later). 
The MSM and DSM are particularly relevant when $\epsilon$ is smaller,
while the dMSM and dDSM are designed especially for larger $\epsilon$.
The MSM is found from the SSM by replacing
the switching process $\gamma$ by its mean (constant in time) value,
giving rise to a process ${\bar u}$ instead of $u$.
The  DSM is found by replacing
the switching process $\gamma$ by the solution of an 
Ornstein-Uhlenbeck (OU) 
process, giving rise to a process  ${\widehat u}$ instead of $u.$
The dMSM is found by replacing $u$ by a process with a constant $\gamma$
in time, but choosing that constant randomly, according to carefully 
chosen weights. This leads to  replacement of
$u$ by a process ${\bar u}'.$ And finally the dDSM  
is found by replacing $u$ by 
${\widehat u}'$
in which $\gamma$ is given by
one of two OU processes for all time, but 
choosing the OU process randomly, according to carefully
chosen weights.
From now on, it will help to keep in mind that MSM and DSM are approximations of SSM for smaller $\epsilon$,
and dMSM and dDSM are approximations of SSM for larger $\epsilon$.

\subsection{Our Contributions}
Existing and extensive numerical studies
naturally give rise to two fundamental questions 
about filtering with model error: (i)
what are the precise 
conditions under which
a given filter with model error is the best 
choice out of some class of filters;
and (ii)
how to choose the free parameters so 
as to maximize the consequent filtering accuracy.
To address these questions 
we investigate the accuracy of
the filters with model error via careful numerical experiments,
and introduce a systematic approach for parameter 
determination.
Specifically, our contributions in the present paper are as follows:

\begin{itemize}

\item 
in addition to studying the filters with model error
MSM and DSM, introduced in  
\cite{gershgorin2010test, gershgorin2010improving},
we also introduce our own filters with model error: 
dMSM and dDSM;

\item 
we build a Gaussian filter and a Gaussian mixture filter for SSM;

\item we show the consistency of the reduced models 
in the extremely small (large) $\epsilon$ regime 
by proving limit theorems that connect the filter signal models
MSM (dMSM) and DSM (dDSM)
with the true signal model SSM;

\item 
we use asymptotic analysis in the 
small (large) $\epsilon$ regime
to obtain analytic formulae
for the adaptive parameters of the simplified models
MSM (dMSM) and DSM (dDSM);

\item 
we employ optimization to solve minimization problem 
that yields suitable parameters for the simplifications
when the scale-separation is
not extreme but
moderate or weak;

\item 
we perform direct numerical simulations to show 
the accuracy and feasibility of the methods. 

\end{itemize}

\subsection{Organization of the Paper}
The paper 
is organized 
as follows.
We precisely define the models used for filtering 
in section~\ref{sec:appssm}.
Our main results are in
section~\ref{sec:DSMparamters},
where various tools, tuned to the relevant
parameter regime for $\epsilon$, are 
deployed to improve filtering accuracy.
We perform numerical experiments in section~\ref{sec:ns} 
and draw conclusions in section~\ref{sec:conclusions}.
Lengthy calculations 
concerning the analysis of models 
are gathered in the appendices, in order to improve accessibility of
the paper.

\section{Filtering With Model Error: Simplifications of SSM}
\label{sec:appssm}
Here we define
four adaptive approximate models for SSM,
based on the analysis of
the qualitative behaviors of the switching process,
and use them to build filters.
Subsection~\ref{subsec:ssr} is concerned with the case when $\epsilon$ is small
(scale-separation regime)
and
subsection~\ref{subsec:rer} is when $\epsilon$ is large
(rare-event regime).

\subsection{Scale-Separation Regime}
\label{subsec:ssr}

\subsubsection{Mean Stochastic Model (MSM)}
In many multi-scale problems, 
the governing equation in which the driving signal is significantly faster
is replaced by an equation 
with non-oscillatory coefficient found as a limit (usually in a weak sense) 
of scale-separation \cite{pavliotis2008multiscale,
bensoussan2011asymptotic,
cioranescu2000introduction}.
This work suggests that, when $\epsilon$ is sufficiently small,
the mean stochastic model (MSM)
\begin{equation}
\begin{split}
\label{eq:msm}
\textbf{(MSM)}\qquad
\left\{ \begin{array}{ll}
d\bar{u} &= - \bar{\gamma}  \bar{u}\,dt+\sigma_u dB_u \\
\bar{\gamma} &= \text{const}
\end{array} \right. 
\end{split}
\end{equation}
can be a good
approximation
of SSM.
Using MSM for filtering we note that, provided $\bar{u}_0$ is Gaussian,
all distributions 
$\mathbb{P}(\bar{u}_{n-1}\vert  Y_{n-1}) \mapsto
\mathbb{P}(\bar{u}_{n}\vert  Y_{n-1}) \mapsto
\mathbb{P}(\bar{u}_n\vert  Y_n)$ 
are Gaussians and may be updated by the Kalman filter \cite{Kalman60}.

\begin{figure}
  \centering
\includegraphics[width=0.91\textwidth]{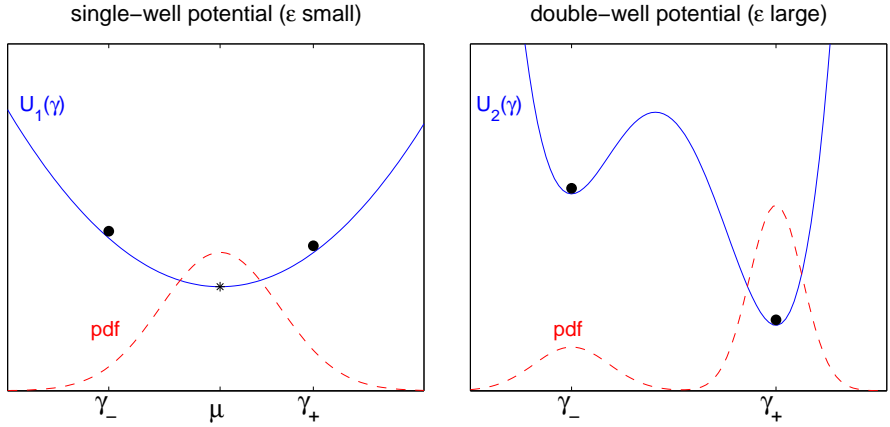}
\caption{
Regularized
modeling
of
the qualitative behaviors of
switching process. 
}
\label{fig:ssm} 
\end{figure}

\subsubsection{Diffusive Stochastic Model (DSM)}
The diffusive stochastic model (DSM) is given by
\begin{equation*}
\begin{split}
\textbf{(DSM)}\qquad
\left\{ 
\begin{array}{ll}
d\widehat{u} &= - \widehat{\gamma}  \widehat{u}\,dt+\sigma_u dB_u \\
d\widehat{\gamma} & = -\frac{\nu}{\epsilon}(\widehat{\gamma}-\mu) \,dt+
\frac{\sigma}{\sqrt{\epsilon}} dB_{\gamma} \\
\end{array} \right. 
\end{split}
\end{equation*}
Note that $\widehat{\gamma}$ is in an Ornstein-Uhlenbeck process: 
solution of the Langevin equation
\begin{equation}
\begin{split}
\label{eq:le}
d{\widetilde{\gamma}} & = -\frac{1}{\epsilon}\nabla U({\widetilde{\gamma}})dt+
\frac{\sigma}{\sqrt{\epsilon}} dB_{\gamma}
\end{split}
\end{equation}
with the potential
\begin{equation}
\label{eq:pot}
\begin{split}
U(x)=U_1(x):= \frac{\nu}{2}(x-\mu)^2.
\end{split}
\end{equation}
We aim to tune this process to match the response of the system,
in the observed variable $u$. The reason for interest in this model
is that, although exact filtering is not possible, 
it is possible to compute an approximate Gaussian filter,
based on exact propagation of the first two moments.
Indeed provided $(\widehat{u}(0),\widehat{\gamma}(0))$ is joint Gaussian,
the mean and covariance of 
$(\widehat{u}(T),\widehat{\gamma}(T))$ are exactly solvable.
Denoting $\widehat{\gamma}_n \equiv \widehat{\gamma}(n T)$,
the resultant moment mapping can be used for uncertainty propagation:
$\mathbb{P}(\widehat{u}_{n-1}, \widehat{\gamma}_{n-1}\vert  Y_{n-1}) \mapsto
\mathbb{P}(\widehat{u}_{n}, \widehat{\gamma}_{n}\vert  Y_{n-1}) $.
Under this Gaussian approximation,
the Kalman filter may be applied to obtain 
$\mathbb{P}(\widehat{u}_{n}, \widehat{\gamma}_{n}\vert  Y_{n-1})
\mapsto
\mathbb{P}(\widehat{u}_n, \widehat{\gamma}_{n}\vert  Y_n)$.
The resulting filter is named 
the stochastic parametrization extended Kalman filter (SPEKF)
in \cite{gershgorin2008nonlinear} where it was introduced.
Finally, a proper marginalization at every step yields the 
object of interest: $\mathbb{P}(\widehat{u}_n\vert  Y_n)$.

\subsection{Rare-Event Regime}
\label{subsec:rer}

\subsubsection{Dual-mode Mean Stochastic Model (dMSM)}
\label{sec:dsm}
When $\epsilon$ is large enough, transitions in $\gamma$ are rare.
To study this case, we build the following 
dual-mode mean stochastic model (dMSM):
\begin{equation}
\begin{split}
\label{eq:msmd}
\textbf{(dMSM)}\qquad
\left\{ \begin{array}{ll}
d\bar{{u}}' &= - \bar{{\gamma}}'  \bar{{u}}'\,dt+\sigma_u dB_u \\
\bar{{\gamma}}' &=  
\left\{ \begin{array}{ll}
{\gamma}_{+} \;\text{with probability} \;\bar{\rho}_{+} \;\text{for}\;  t \geq 0 \\
{\gamma}_{-} \;\text{with probability} \;\bar{\rho}_{-}= 1-\bar{\rho}_{+}\;\text{for}\;  t \geq 0 \\
\end{array} \right. 
\end{array} \right. 
\end{split}
\end{equation}
as the reduced modeling of SSM.
This can be viewed as an example of the more general 
switching linear dynamical system model \cite{ghahramani2000variational}.

If the probability distribution of $\bar{u}'(0)$ 
is the sum of weighted Gaussian kernels, then
note that
\begin{equation}
\label{eq:probmsmd}
\begin{split}
\mathbb{P}(\bar{u}'(T))
=\mathbb{P}(\bar{\gamma}'=\gamma_{+})\mathbb{P}(\bar{u}'(T)\vert \bar{\gamma}'=\gamma_{+})
+\mathbb{P}(\bar{\gamma}'=\gamma_{-})\mathbb{P}(\bar{u}'(T)\vert \bar{\gamma}'=\gamma_{-}).
\end{split}
\end{equation}
Under this assumption on $\bar{u}'(0)$, then, we may use 
the Gaussian mixture filter 
to obtain the exact filtering solution of dMSM.
The procedure
$\mathbb{P}(\bar{u}'_{n-1}\vert  Y_{n-1}) \mapsto
\mathbb{P}(\bar{u}'_{n}\vert  Y_{n-1})$
is performed using (\ref{eq:probmsmd}),
and a parallel application of the Kalman filter to each 
Gaussian kernel, along with updating of the weights of each kernel, 
completes the update $\mathbb{P}(\bar{u}'_{n}\vert  Y_{n-1})
\mapsto
\mathbb{P}(\bar{u}'_n\vert  Y_n)$.

In practice, the geometric growth in the number of kernels
in the number of prediction steps
prevents tractable exact inference  as data is accumulated
sequentially.
One resolution that we adopt here is through
the projection of the filtering solution onto 
the space of tractable distributions.
Following the idea of
assumed density filtering
\cite{maybeck1982stochastic},
a large mixture of Gaussians is replaced by a smaller mixture of Gaussians
at regular time-intervals, while filtering progresses \cite{5977695}.

\subsubsection{Dual-mode Diffusive Stochastic Model (dDSM)}
As in the DSM we now try to use a diffusion process to model
the switching process $\gamma$, in order to benefit from the
possibility of propagating second moments exactly, as is done in the DSM.
When $\epsilon$ is large, however, the process
(\ref{eq:le}) with the single-well potential \eqref{eq:pot} 
is not suitable for mimicking rare transitions.
We instead consider 
a double-well potential $U_2(x)$ for $U(x)$
(illustrated in the right-panel of Fig.~\ref{fig:ssm}).
In this scenario, the motion of $\widetilde{\gamma}$
is captured within either of the potential wells for significant 
time periods, but random perturbations allow it to effectively
jump over the potential barrier
and enter the parallel metastable state.

Based upon the quadratic expansions
\begin{equation*}
\begin{split}
U_2(x) 
& \simeq 
\left\{ \begin{array}{ll}
U_2(\mu_{+})+\frac{\nu_{+}}{2}(x-\mu_{+})^2 \qquad
\text{when} \;\;|x-\mu_{+}| \;\; \text{is small} \\
U_2(\mu_{-})+\frac{\nu_{-}}{2}(x-\mu_{-})^2 \qquad
\text{when} \;\;|x-\mu_{-}| \;\; \text{is small}\\
\end{array} \right. 
\end{split}
\end{equation*}
we 
build a new model 
\begin{equation}
\begin{split}
\label{eq:dsmdeq}
\textbf{(dDSM)}\qquad
\left\{ \begin{array}{ll}
d\widehat{{u}}' &= - \widehat{{\gamma}}'  \widehat{{u}}'\,dt+\sigma_u dB_u \\
d\widehat{{\gamma}}' & 
=
\left\{ \begin{array}{ll}
-\frac{\nu_{+}}{\epsilon}(\widehat{{\gamma}}'
-\mu_{+}) \,dt+
\frac{\sigma_{+}}{\sqrt{\epsilon}} dB_{\gamma} 
\;\text{with probability} \;\widehat{\rho}_{+}\\
-\frac{\nu_{-}}{\epsilon}(\widehat{{\gamma}}'
-\mu_{-}) \,dt+
\frac{\sigma_{-}}{\sqrt{\epsilon}} dB_{\gamma} 
\;\text{with probability} \;\widehat{\rho}_{-}= 1-\widehat{\rho}_{+}
\end{array} \right. 
\end{array} \right. 
\end{split}
\end{equation}
where
the uncertainty is separately delivered by two independent sets of SDEs.  
Eq.~(\ref{eq:dsmdeq}) is named by the dual-mode diffusive stochastic model (dDSM).

When
$(\widehat{u}'(0),\widehat{\gamma}'(0))$
is a Gaussian mixture,
utilizing the exact solvability of
the first two moments of
the propagated distributions (as for DSM),
the probability of
$(\widehat{u}'(T),\widehat{\gamma}'(T))$
can be approximated as Gaussian mixture with the number of kernels doubled,
similarly to dMSM. As for dMSM we may perform a reduction of
the number of mixtures to retain computational tractability.
In this way, the approximate Gaussian mixture filter
$\mathbb{P}(\widehat{u}'_{n-1}, \widehat{\gamma}'_{n-1}\vert  Y_{n-1}) \mapsto
\mathbb{P}(\widehat{u}'_{n}, \widehat{\gamma}'_{n}\vert  Y_{n-1}) 
\mapsto
\mathbb{P}(\widehat{u}'_n, \widehat{\gamma}'_{n}\vert  Y_n)$ 
is established.

\section{Model Validations}
\label{sec:DSMparamters}
In this section, we proceed
(i) to validate the proposed models,
and (ii) to determine the adaptive parameters.
We classify
the $\epsilon$ parameter regime
into the six regions;
the scale-separation limit $\{\epsilon \to 0\}$, 
the sharp scale-separation regime $\{\epsilon \ll 1\}$, 
the imprecise scale-separation regime $\{\epsilon <1 \}$, 
the moderately rare-event regime $\{\epsilon >1\}$, 
the extremely rare-event regime $\{\epsilon \gg 1\}$, 
the rare-event limit $\{\epsilon \to \infty\}$.
Subsection~\ref{sec:equivalence}
is devoted to the
study of the case
$\{\epsilon \to 0, \epsilon \to \infty \}$, 
and
subsection~\ref{subsec:dsmcoeff}
to
$\{\epsilon \ll 1, \epsilon \gg 1 \}$, 
and
subsection~\ref{subsec:mini}
to
$\{\epsilon < 1, \epsilon > 1\}$.

\subsection{Convergence Results}
\label{sec:equivalence}
Here we demonstrate 
the consistency of the simplified models
by showing that
$u_T, \widehat{u}_T \to \bar{u}_T$
(subscript notation will be abused and $u_T = u(T)$)
as $\epsilon \to 0$
and that
$u_T, \widehat{u}'_T \to \bar{u}'_T$ as $\epsilon \to \infty$ 
in senses 
elucidated
in what follows. All
proofs are deferred to the Appendix~\ref{sec:equivalenceproof}.

\subsubsection{Scale-Separation Limit}
The main results here are the following theorem and corollary; 
the constants are defined in the developments following their
statements. 

\bigskip

\begin{Theorem}
\label{ssmthm}
Assume that ${u}_0$, $\bar{u}_0$
and $\widehat{u}_0$ 
are identically distributed Gaussian random variables, and assume that 
$({u}_0,{\gamma}_0)$
and $(\widehat{u}_0,\widehat{\gamma}_0)$ are 
independent pairs of random variables.
If $\bar{\gamma}$ and $\mu$ are equal to
$\bar{\gamma}_\infty
\equiv
\frac{\lambda_{-}\gamma_{+}+\lambda_{+}\gamma_{-}}{\lambda_{-} +\lambda_{+}}$,
then, for any fixed $T>0$, 
as $\epsilon \to 0$
the mean and variance of $u_T$ and $\widehat{u}_T$ 
converge to those of $\bar{u}_T$.
\end{Theorem}

\begin{proof}[of Theorem \ref{ssmthm}]
This follows from Lemmas~\ref{thm:uconv2}, \ref{thm:ssmmgf}, \ref{thm:dsmmgf},
using the explicit calculations which are presented after the corollary below.
\end{proof}

\bigskip

\begin{Corollary}
\label{cor1}
Under the conditions in Theorem~\ref{ssmthm},
the mean and variance in
the Gaussian filters for
$u_n\vert Y_n$
(defined in Appendix~\ref{subsec:ssmgaussianfilter})
and
$\widehat{u}_n\vert Y_n$,
converge to
those of $\bar{u}_n|Y_n$,
for fixed $n > 0$,
as $\epsilon \to 0$.
\end{Corollary}

\begin{proof}[of Corollary \ref{cor1}]
This follows from 
the data assimilation formula of the
Kalman filter
\cite{anderson1979optimal}.
\end{proof}

\bigskip

Let ${u}^\epsilon$ solve
\begin{equation}
\begin{split}
\label{eq:langevin}
d{u}^\epsilon &= -{\gamma}^\epsilon {u}^\epsilon dt + \sigma_u dB_u
\end{split}
\end{equation}
for a random process
${\gamma}^\epsilon$.
For
$\Gamma^\epsilon_t \equiv \int^t_0 {\gamma}^\epsilon (s) ds$
the integral process of
${\gamma}^\epsilon$
we have the
variation-of-constants yields
\begin{equation}
\begin{split}
\label{eq:voc}
u_T^\epsilon & = e^{-\Gamma^\epsilon_T}u_0^\epsilon 
+ \sigma_u \int^T_0  e^{-(\Gamma^\epsilon_T-\Gamma^\epsilon_t)}dB_u(t).
\end{split}
\end{equation}
Application of the It\^o formula shows that
the mean and covariance are given by
\begin{equation}
\begin{split}
\label{eq:uqmom}
\left\langle {u}^\epsilon_T \right\rangle & = \left\langle e^{-{\Gamma}^\epsilon_T}{u}^\epsilon_0 \right\rangle 
\\
\text{Var}(u^\epsilon_T) &=
\left\langle \left( e^{-{\Gamma}^\epsilon_T}u_0^\epsilon \right)^2  \right\rangle
-\left\langle e^{-{\Gamma}^\epsilon_T}u^\epsilon_0 \right\rangle^2
+ \sigma_u^2  \int^T_0  \left\langle e^{-2({\Gamma}^\epsilon_T-{\Gamma}^\epsilon_t)} \right\rangle dt.
\end{split}
\end{equation}
Here and  henceforth, $\langle \cdots\rangle$ denotes the statistical average.
Eq.~(\ref{eq:uqmom}) reveals that
the moment generating function (MGF) of integral process of $\gamma^\epsilon$ 
are particularly relevant to
the first two moments propagation of $u^\epsilon$
governed by (\ref{eq:langevin}).

\bigskip

\begin{Lemma}
\label{thm:uconv2}
Let 
$\bar{u}_t$ satisfy MSM
(\ref{eq:msm}).
If
\begin{equation}
\label{eq:mgfconvs}
\left\langle e^{\alpha( \Gamma^\epsilon_T-\Gamma^\epsilon_t )} \right\rangle 
\to 
e^{\alpha \bar{\gamma}(T-t) }
\quad \text{for} \;\;\; \alpha=-1, -2
\;\;\;\text{and} \;\;\; 0\leq t\leq T
\end{equation}
and
if
\begin{equation}
\begin{split}
\label{eq:mcconvconds}
\left\langle \left( e^{-{\Gamma}^\epsilon_T}u_0^\epsilon \right)^m  \right\rangle
\to 
\left\langle \left( e^{-\bar{\gamma} T} \bar{u}_0 \right)^m \right\rangle 
\quad \text{for} \;\;\; m=1,2
\end{split}
\end{equation}
as $\epsilon \to 0$
then
the mean and variance of 
${u}^\epsilon_T$ converge to those of $\bar{u}_T$.
Further
if
\begin{equation}
\begin{split}
\label{eq:l2convconds}
\left\langle 
\left( e^{-\Gamma^\epsilon_T}{u}^\epsilon_0 - e^{-\bar{\gamma}T}\bar{u}_0  \right)^2 \right\rangle 
\to 0
\end{split}
\end{equation}
as $\epsilon \to 0$
then
${u}^\epsilon_T$ converges to $\bar{u}_T$ 
in $L^2(\Omega;{\mathbb R}).$
The convergence rates are determined by 
those associated with Eqs.~(\ref{eq:mgfconvs}), (\ref{eq:mcconvconds})
and (\ref{eq:l2convconds}).
\end{Lemma}

\bigskip

Let 
$\Gamma_t \equiv \int^t_0 \gamma(s) ds$ 
and
$\widehat{\Gamma}_t \equiv \int^t_0 \widehat{\gamma}(s) ds$
be the integral processes associated to SSM and DSM respectively.
Because 
$\bar{\Gamma}_t \equiv \int_0^t \bar{{\gamma}}(s) ds=\bar{{\gamma}} t$
($\bar{{\gamma}}(s)=\bar{{\gamma}}$ is constant) in the case of the MSM,
we expect
$u_T, \widehat{u}_T \to \bar{u}_T$
provided both integral processes,
$\Gamma_t$ 
and
$\widehat{\Gamma}_t $,
behave like the probability distribution $\delta_{\bar{\gamma}t}$
in the small $\epsilon$ limit.
It turns out this is indeed the case 
due to averaging. The next two lemmas highlight this behavior.

\bigskip

\begin{Lemma}[SSM]
\label{thm:ssmmgf}
Let
$\gamma_t \sim \rho_{+}(t)\delta_{\gamma_{+}} +\rho_{-}(t)\delta_{\gamma_{-}}$
then, for any fixed $t>0$,
$\rho_{\pm}(t)\to \frac{\lambda_{\mp}}{\lambda_{-} +\lambda_{+}}$
as $\epsilon \to 0$.
Let
$\gamma_\infty \sim
\frac{\lambda_{-}}{\lambda_{-} +\lambda_{+}}\delta_{\gamma_{+}}
+\frac{\lambda_{+}}{\lambda_{-} +\lambda_{+}}\delta_{\gamma_{-}}$
and
$\bar{\gamma}_\infty \equiv \langle \gamma_\infty\rangle
=
\frac{\lambda_{-}\gamma_{+}+\lambda_{+}\gamma_{-}}{\lambda_{-} +\lambda_{+}}$
then, for 
any fixed $T>t>0$,
we have 
$
\left\langle e^{\alpha( {\Gamma}_T-{\Gamma}_t )} \right\rangle 
\to 
e^{\alpha \bar{\gamma}_\infty(T-t) }
$
as $\epsilon \to 0$.
\end{Lemma}

\bigskip

\begin{Lemma}[DSM]
\label{thm:dsmmgf}
Let
$\widehat{\gamma}_\infty \sim \mathcal{N}(\mu,\sigma^2/2\nu)$
then, for any fixed $t>0$,
the mean and variance of
$\widehat{\gamma}_t$ converge to those of $\widehat{\gamma}_\infty$ 
as $\epsilon \to 0$.
Furthermore, 
we have, for any fixed $T>t>0$, 
$
\left\langle e^{\alpha( \widehat{\Gamma}_T-\widehat{\Gamma}_t )} \right\rangle \to e^{\alpha \mu(T-t) }
$
as $\epsilon \to 0$.
\end{Lemma}

\subsubsection{Rare-Event Limit}

The main results in this regime are the following
theorem and corollary.

\bigskip

\begin{Theorem}
\label{ssmthm2}
Assume that 
${u}_0$, $\bar{u}'_0$ 
and $\widehat{u}'_0$ 
are 
identically distributed 
Gaussian 
random variables, 
and assume that $({u}_0,{\gamma}_0)$
and $(\widehat{u}'_0,\widehat{\gamma}'_0)$
are independent pairs of random variables.
Then, for any fixed $T>0$, 
the mean and variance of
$\mathbb{P}({u}_T\vert {\gamma}_0=\gamma_{\pm})$,
$\mathbb{P}(\widehat{u}'_T\vert \widehat{\gamma}'_0=\gamma_{\pm})$
converges to 
those of 
$\mathbb{P}(\bar{u}'_T\vert \bar{\gamma}' =\gamma_{\pm})$
as $\epsilon \to \infty$.

Furthermore,
let
${\gamma}_0 \triangleq \gamma_\infty$,
and
if
$\bar{\gamma}'$ 
and
$\widehat{\gamma}'_0$ 
are identically distributed with $\gamma_\infty$,
then
the weight, mean and variance
of 
components
in
the Gaussian mixture approximation for
${u}_T$, $\widehat{u}'_T$
converge to those of
$\bar{u}'_T$
as $\epsilon \to \infty$.

\end{Theorem}

\begin{proof}
This follows from Lemmas~\ref{thm:uconv3}, \ref{lem:ssmd}, \ref{lem:dsmd}
below.
\end{proof}

\bigskip

\begin{Corollary}
Under the conditions in Theorem~\ref{ssmthm2},
the weight, mean and variance 
of mixture components
in
the Gaussian mixture filters for
$u_n\vert Y_n$
(defined in Appendix~\ref{subsec:ssmgaussiansumfilter})
and
$\widehat{u}'_n\vert Y_n$,
converge to
those of 
$\bar{u}'_n|Y_n$,
for fixed $n > 0$,
as $\epsilon \to \infty$.
\end{Corollary}

\begin{proof}
This follows from parallel application of the Kalman filter update
to the mixture components.
\end{proof}

\bigskip

\begin{Lemma}
\label{thm:uconv3}
Let
$\bar{u}'_t$ 
solve dDSM 
(\ref{eq:msmd}).
If, for each fixed $T>t>0$,
\begin{equation}
\label{eq:mgfconvd}
\left\langle e^{\alpha( \Gamma^\epsilon_T-\Gamma^\epsilon_t )} 
\vert \gamma^\epsilon_0=\gamma_{\pm}\right\rangle 
\to 
e^{\alpha {\gamma_{\pm}}(T-t) }
\quad \text{for} \;\;\; \alpha=-1, -2
\;\;\;\text{and} \;\;\; 0\leq t\leq T
\end{equation}
and if
\begin{equation}
\begin{split}
\label{eq:mcconvcondd}
\left\langle \left( e^{-{\Gamma}^\epsilon_T}u_0^\epsilon \right)^m \vert \gamma^\epsilon_0=\gamma_{\pm}\right\rangle
\to 
\left\langle \left( e^{-{\gamma}_{\pm} T} \bar{u}_0 \right)^m \right\rangle 
\quad \text{for} \;\;\; m=1,2
\end{split}
\end{equation}
as $\epsilon \to \infty,$ then
the mean and variance of 
${u}^\epsilon_T\vert \gamma^\epsilon_0=\gamma_{\pm}$ converge to 
those of
$\bar{u}_T\vert \bar{\gamma}'=\gamma_{\pm}$.
The convergence rates are determined by 
those associated with Eqs.~(\ref{eq:mgfconvd}), (\ref{eq:mcconvcondd}).

Furthermore, if ${\gamma}^\epsilon_0\triangleq \bar{\gamma}'$,
then 
the weight, mean and variance
of 
components
in
the Gaussian mixture approximation for
${u}^\epsilon_T$
converge to those of $\bar{u}_T$ from
$\mathbb{P}({u}^\epsilon_T)
=\mathbb{P}({\gamma}^\epsilon_0=\gamma_{+})\mathbb{P}({u}^\epsilon_T\vert {\gamma}^\epsilon_0=\gamma_{+})
+\mathbb{P}({\gamma}^\epsilon_0=\gamma_{-})\mathbb{P}({u}^\epsilon_T\vert {\gamma}^\epsilon_0=\gamma_{-})$.

\end{Lemma}

\bigskip

To ensure the convergences of SSM and dDSM to dMSM,
as $\epsilon$ grows, both
${\Gamma}_t$ and
$\widehat{\Gamma}'_t \equiv \int^t_0 \widehat{\gamma}'(s)ds$
need to converge to
$\bar{\Gamma}'_t \equiv \int_0^t \bar{{\gamma}}'(s) ds
\sim \bar{\rho}_{+}\delta_{\gamma_{+}t}+ \bar{\rho}_{-}\delta_{\gamma_{-}t}$.

\bigskip

\begin{Lemma}
[SSM]
\label{lem:ssmd}
For fixed $T>t>0$
$\left\langle e^{\alpha (\Gamma_T-\Gamma_t)} \vert {\gamma}_0=\gamma_{\pm} \right\rangle \to
e^{ \alpha \gamma_{\pm}(T-t)}$
as $\epsilon \to \infty$.
\end{Lemma}

\bigskip

\begin{Lemma}
[dDSM]
For fixed $T>t>0$
\label{lem:dsmd}
$\left\langle e^{\alpha (\widehat{\Gamma}'_T-\widehat{\Gamma}'_t)} 
\vert \widehat{\gamma}'_0=\gamma_{\pm} \right\rangle 
\to
e^{ \alpha \gamma_{\pm}(T-t)}$
as $\epsilon \to \infty$.
\end{Lemma}

\subsection{Asymptotic Matching}
\label{subsec:dsmcoeff}
The convergence results 
in the preceding subsection demonstrate that
the filtering performances of the approximate filters,
and the exact filter, would be similar to one another 
in that
$\bar{\gamma}  = \bar{\gamma}_\infty$,
$\mu = \bar{\gamma}_\infty$
(when $\epsilon \ll 1$)
and
$\bar{\gamma}'\triangleq {\gamma}_\infty$, 
$\widehat{\rho}_{\pm} \propto {\lambda_{\mp}}$,
$\mu_{\pm} = \gamma_{\pm}$
(when $\epsilon \gg 1$).
The former result relates to  
the robustness of the DSM filter inherited from the adaptive parameters
$\{ \mu,\sigma\}$, demonstrated here when $\epsilon$ is small,
and demonstrated through extensive numerical simulations in
\cite{gershgorin2010test, gershgorin2010improving}.

However, when $\epsilon$ deviates considerably from the
two extreme values ($\epsilon =0$ and $\epsilon =\infty$),
the choice of associated parameters in the filtering models
is indeed one critical 
factor for a successful 
filtering with model error.
The current and next subsections 
concern the determination of 
$\Theta\equiv\{\mu,\nu,\sigma\}$ for DSM,
and
$\Theta'\equiv\{\widehat{\rho}_{\pm},\mu_{\pm},\nu_{\pm},\sigma_{\pm}\}$ for dDSM.
Unlike earlier works in this area where
these associated parameters are 
chosen from a number of parallel
direct numerical simulations
comparing the original dynamics and its simplifications,
our approach will specify the parameters
in a systematic analysis-based manner.

\subsubsection{Sharp Scale-Separation Regime}
In this parameter regime,
because DSM is associated to a nonlinear 
approximate
Kalman filter,
we attempt to equate
the first and second order statistics
of SSM and DSM,
\begin{equation}
\begin{split}
\label{eq:momcriteria}
\left\{ 
\begin{array}{ll}
\langle u_T \rangle & = \langle \widehat{u}_T \rangle
\\
\text{Var}(u_T)  & = \text{Var}(\widehat{u}_T)
\end{array} \right. 
\end{split}
\end{equation}
for high accuracy.
It is worth noticing that, in view of (\ref{eq:uqmom}),
if the MGFs agree with one another, that is if
\begin{subnumcases}
{\label{eq:criteria22}}
\label{eq:criteria}
\qquad \; \left\langle e^{\alpha \Gamma_T}\right\rangle 
= \left\langle e^{\alpha \widehat{\Gamma}_T}\right\rangle  \qquad \quad \text{for} \;\;\; \alpha=-1, -2
\\
\label{eq:mcriteria2}
\left\langle e^{\alpha (\Gamma_T-\Gamma_t)}\right\rangle 
 =
\left\langle e^{\alpha (\widehat{\Gamma}_T-\widehat{\Gamma}_t)}\right\rangle 
\quad \text{for}\; \; \; \alpha=-2 \;\;\;\text{and}\;\;\;  0\leq t \leq T
\end{subnumcases}
and if $({u}_0,{\gamma}_0)$ and $(\widehat{u}_0,\widehat{\gamma}_0)$
are uncorrelated, 
and if ${u}_0 \triangleq \widehat{u}_0$,
then Eq.~(\ref{eq:momcriteria}) holds.
Motivated by convergence to the common limit, as demonstrated above,
we here strive to asymptotically satisfy (\ref{eq:criteria22}) 
when $\epsilon \ll 1$.

To that end,
we derive the approximation
\begin{equation}
  \begin{split}
\label{eq:fmcharasy1}
\left\langle e^{\alpha (\Gamma_T-\Gamma_t)} \right\rangle 
& \simeq \exp\Bigg( \alpha \bar{\gamma}_\infty(T-t) +\alpha^2\frac{3}{8}\frac{
(\gamma_{-}-\gamma_{+})^2(\lambda_{-}^2+\lambda_{+}^2)}{\lambda_{-}\lambda_{+}
(\lambda_{+}+ \lambda_{-})} (T-t) \epsilon\\
& +
\alpha \left(
\mathbb{P}(\gamma_0=\gamma_{+})
\frac{(\gamma_{+}-\gamma_{-})}{4\lambda_{+}} 
+
\mathbb{P}(\gamma_0=\gamma_{-})
\frac{(\gamma_{-}-\gamma_{+})}{4\lambda_{-}} 
\right)
\epsilon  
+\mathcal{O}(\epsilon^2)
\Bigg) 
\quad \epsilon < T-t
  \end{split}
\end{equation}
in the Appendix~\ref{subsec:scaleseparation}.
We also derive the approximations
\begin{subequations}
\label{eq:spmcharasy0}
\begin{align}
\label{eq:spmcharasy1}
\left\langle e^{\alpha \widehat{\Gamma}_T}\right\rangle 
& = \exp\left( \alpha \left( \mu T + \langle \widehat{\gamma}_0-\mu\rangle
\frac{\epsilon}{\nu}\right)
+\alpha^2 \frac{\sigma^2}{2\nu^2}T\epsilon +\mathcal{O}(\epsilon^2) \right) 
\quad \epsilon < T \\
\label{eq:spmcharasy2}
\left\langle e^{\alpha (\widehat{\Gamma}_T-\widehat{\Gamma}_t)} \right\rangle 
& = \exp\left( \alpha \mu(T-t) 
+\alpha^2 \frac{\sigma^2}{2\nu^2}(T-t)\epsilon + \mathcal{O}(\epsilon^2) \right) 
\quad \epsilon < t 
\end{align}
\end{subequations}
in the Appendix~\ref{subsec:dsmssr}.
Importantly, the exponents of MGFs are
of the second-order with respect to $\alpha T$ 
up to $\mathcal{O}(\epsilon)$,
indicating that both $\Gamma_T$ and $\widehat{\Gamma}_T$ 
are statistically closer to 
Gaussian in this parameter regime.

From a comparison between
(\ref{eq:fmcharasy1}) and (\ref{eq:spmcharasy0}),
we realize
Eq.~(\ref{eq:criteria22})
is 
asymptotically
met
provided
$\bar{\gamma}_\infty=\mu$
and
\begin{equation}
\begin{split}
\label{eq:sc1}
\frac{3}{8}\frac{ (\gamma_{-}-\gamma_{+})^2(\lambda_{-}^2+\lambda_{+}^2)}{\lambda_{-}\lambda_{+}
(\lambda_{+}+ \lambda_{-})} & =\frac{\sigma^2}{2\nu^2}
\end{split}
\end{equation}
and
\begin{equation}
\begin{split}
\label{eq:sc2}
\mathbb{P}(\gamma_0=\gamma_{+})
\frac{(\gamma_{+}-\gamma_{-})}{4\lambda_{+}} 
+
\mathbb{P}(\gamma_0=\gamma_{-})
\frac{(\gamma_{-}-\gamma_{+})}{4\lambda_{-}} 
=\frac{\langle \widehat{\gamma}_0-\mu \rangle}{\nu}.
\end{split}
\end{equation}

Eqs.~(\ref{eq:sc1}), (\ref{eq:sc2})
can be solved to determine a unique set of
$\{\nu,\sigma^2\}$ but
might result in $\nu < 0$ which is unphysical.
In order to avoid 
this 
possibility, 
we impose 
the equivalence between variances of stationary processes
$\gamma_\infty$ and 
$\widehat{\gamma}_\infty$ 
\begin{equation}
\begin{split}
\label{eq:varmatch}
\frac{\lambda_{+}\lambda_{-}(\gamma_{+}-\gamma_{-})^2}{(\lambda_{-} +\lambda_{+})^2}
& = \frac{\sigma^2}{2\nu} \\
\end{split}
\end{equation}
instead of
Eq.~(\ref{eq:sc2}).
From 
Eqs.~(\ref{eq:sc1}) and (\ref{eq:varmatch}),
we obtain
\begin{equation}
\begin{split}
\label{eq:choice}
\Theta_{\text{naive}} \equiv
\begin{cases}
\mu & = \bar{\gamma}_\infty \\
\nu 
&=
\frac{8}{3}\frac{ \lambda_{-}^2\lambda_{+}^2 }{
(\lambda_{-}+\lambda_{+})
(\lambda_{+}^2+ \lambda_{-}^2)}  
\qquad \text{when} \quad \epsilon \ll 1
\\
\sigma^2 
&= 
\frac{16}{3}\frac{ \lambda_{-}^3\lambda_{+}^3 (\gamma_{-}-\gamma_{+})^2 }{
(\lambda_{-}+\lambda_{+})^3 (\lambda_{+}^2+ \lambda_{-}^2)}\\
\end{cases}
\end{split}
\end{equation}
which we term the naive set of DSM parameters,
valid when $\epsilon \ll 1$.

\subsubsection{Extremely Rare-Event Regime}
Using a similar argument to that employed in the case of DSM,
we set
${\rho}_{\pm}=\widehat{\rho}_{\pm}=\frac{\lambda_{\mp}}{\lambda_{-} +\lambda_{+}}$
and attempt to satisfy
\begin{equation*}
\begin{split}
\left\{ 
\begin{array}{ll}
\langle u_T \vert \gamma_0=\gamma_{\pm} \rangle & = \langle \widehat{u}_T \vert \widehat{\gamma}'_0=\gamma_{\pm} \rangle
\\
\text{Var}(u_T\vert \gamma_0=\gamma_{\pm})  & = \text{Var}(\widehat{u}_T\vert \widehat{\gamma}'_0=\gamma_{\pm})
\end{array} \right. 
\end{split}
\end{equation*}
hence
\begin{subnumcases}
{\label{eq:Dcriteria}}
\label{eq:Dcriteria1}
\quad\;  \quad\left\langle e^{\alpha \Gamma_T} \vert \gamma_0=\gamma_{\pm} \right\rangle 
 =
\left\langle e^{\alpha \widehat{\Gamma}'_T}\vert \widehat{\gamma}'_0=\gamma_{\pm}\right\rangle  
 \qquad \text{for} \;\; \alpha=-1, -2
\\
\label{eq:Dcriteria2}
 \left\langle e^{\alpha (\Gamma_T-\Gamma_t)}\vert {\gamma}_0=\gamma_{\pm}\right\rangle 
 =
\left\langle e^{\alpha (\widehat{\Gamma}'_T-\widehat{\Gamma}'_t)}
\vert \widehat{\gamma}'_0=\gamma_{\pm}\right\rangle 
\; \text{for}\; \; \alpha=-2 \;\;\text{and}\;\;  0\leq t \leq T
\end{subnumcases}
for dDSM.

In the case $\epsilon \gg 1$, we derive
\begin{equation}
  \begin{split}
\label{eq:ssmmgfjt5}
\left\langle e^{\alpha \Gamma_T} \vert {\gamma}_0=\gamma_{\pm} \right\rangle 
& \simeq
\exp\left( \alpha \gamma_{\pm}T-\frac{1}{\epsilon}\lambda_{\pm}T\right)\\
  \end{split}
\end{equation}
in the Appendix~\ref{subsec:ssmrer}
and
\begin{equation}
\begin{split}
\label{eq:dsmmgfp}
\left\langle e^{\alpha \widehat{\Gamma}'_T} \vert \widehat{\gamma}'_0=\gamma_{\pm} \right\rangle 
& \simeq \exp\left( 
\alpha 
\left(
{\gamma}_{\pm}T - \frac{1}{2\epsilon} ( {\gamma}_{\pm}-\mu_{\pm}) \nu_{\pm} T^2 \right) 
+\frac{\alpha^2}{2} \frac{(\sigma_{\pm})^2}{3\epsilon}T^3 
\right)\\
\end{split}
\end{equation}
in the Appendix~\ref{subsec:dsmrer}.
Note 
the exponents in
(\ref{eq:ssmmgfjt5}) and (\ref{eq:dsmmgfp})
are of different forms, indicating both
${\Gamma}_T$
and
$\widehat{\Gamma}'_T$
are distant from Gaussian in this parameter regime.

Differently from the case of DSM,
we here manage to asymptotically satisfy
Eq.~(\ref{eq:Dcriteria1}) alone, yielding
\begin{equation}
\begin{split}
\label{eq:ddsmpara}
\Theta'_{\text{naive}} \equiv
\begin{cases}
\widehat{\rho}_{\pm}=\frac{\lambda_{\mp}}{\lambda_{-} +\lambda_{+}}\\
\mu_{\pm}=2T\frac{\lambda_{+}\lambda_{-}(\gamma_{+}-\gamma_{-})^2}{(\lambda_{-} +\lambda_{+})^2}
+\gamma_{\pm}
\qquad \text{when} \quad \epsilon \gg 1
\\
\nu_{\pm}= \frac{3\lambda_{\pm}}{2T^2}
\frac{(\lambda_{-} +\lambda_{+})^2}{\lambda_{+}\lambda_{-}(\gamma_{+}-\gamma_{-})^2}
\\
(\sigma_{\pm})^2 = \frac{3\lambda_{\pm}}{T^2}\\
\end{cases}
\end{split}
\end{equation}
which we term the naive set of dDSM parameters,
valid when $\epsilon \gg 1$.
Unlike
Eq.~(\ref{eq:choice}),
due to the dependence on $T$, 
the set of parameters
(\ref{eq:ddsmpara})
is valid only for fixed-time prediction.
The Gaussian mixture from dDSM
with $\Theta'_{\text{naive}}$
leads to accurate mean approximations
but the accuracy of the variance approximation 
is not guaranteed in view of 
Eq.~(\ref{eq:uqmom}) where integration over $[0,T]$ is involved.

\subsection{Minimizing Sum-of-Squares}
\label{subsec:mini}
In the parameter regime $\epsilon \sim O(1)$,
due to the absence of small or large parameters
allowing for asymptotic analysis,
we invoke a minimization principle to determine 
the set of
parameters
$\Theta$ and $\Theta'$.

\subsubsection{Imprecise Scale-Separation Regime}
When $\epsilon < 1$,
we aim to find $\Theta$ which minimizes 
the sum-of-squares 
\begin{equation}
\begin{split}
\label{eq:qdf}
J(\epsilon) \equiv
\kappa \big\vert \langle u_T \rangle - \langle \widehat{u}_T\rangle \big\vert^2
+
\big\vert \text{Var}(u_T) - \text{Var}( \widehat{u}_T) \big\vert^2
\end{split}
\end{equation}
where
$\kappa \geq 0$ is introduced 
to ensure appropriate scaling 
of the two terms in the objective function.
To be more precise,
given $(\widehat{u}_0, \widehat{\gamma}_0)$,
Eq.~(\ref{eq:qdf})
is an algebraic relation in terms of $\Theta$
once we impose ${u}_0:= \widehat{u}_0$
and 
${\gamma}_0:=\gamma_\infty$
(see Appendices \ref{app:ssm} and \ref{app:dsm}).
Note that
a minimizer of $J(\epsilon)$ comes as close as possible to 
fulfilling Eq.~(\ref{eq:momcriteria}). 
It is worth mentioning that, differently from
the MFG matching (\ref{eq:criteria22})
for which 
$(\widehat{u}_0, \widehat{\gamma}_0)$ should be at most weakly correlated
for the approach to be valid,
the minimization methodology can be used
irrespective of 
their potentially strong correlation.

We identify a (local) minimizer by taking 
$\Theta_{\text{naive}}$ as an initial starting point,
and applying an optimizer such as gradient descent.
This minimization can be performed using continuation
in $\epsilon$, starting from $\epsilon \ll 1$ where
the initial guess will be accurate.
Because the solution of this minimization is computed
at each assimilation time step we name it {\em dynamic calibration}
and denote the resulting time-dependent parameters by
$\Theta_{\text{dynamic}}$.
Of course the key issue
in sequential filtering that we are addressing is
to maintain an accurate description of the 
evolving probability distribution
with reasonable computational cost. In this context
it is impractical to compute $\Theta_{\text{dynamic}}$
at every observation time.  In practice, one can 
take a time average of a range of dynamic calibrations.
We refer to this as static calibration and denote the
resulting parameter by $\Theta_{\text{static}}$.

\subsubsection{Moderately Rare-Event Regime}
As for the extremely rare-event regime,
we carry out the same procedure
for each stable and unstable Gaussian kernel.
As for the imprecise scale-separation regime we also
minimize an expression analogous to Eq.~(\ref{eq:qdf})
in which the conditioned mean and covariance are used instead.
We first find $\Theta'_{\text{dynamic}}$
from $\Theta'_{\text{naive}}$, and next find
$\Theta'_{\text{static}}$ from $\Theta'_{\text{dynamic}}$.
Unlike the method based on matching MGF asymptotics, where
the potential inaccuracy of variance approximations are present,
this method simultaneously accounts for accuracy in both 
the mean and covariance approximations.

\section{Numerical Simulations}
\label{sec:ns}
Having obtained 
three different versions of adaptive parameters 
(naive set, static calibration,
dynamic calibration)
for DSM and dDSM, we here investigate 
the filtering performances of the suggested models
using numerical simulations.

Very importantly, one distinguished
advantage of the framework we are currently adopting lies in 
the analytic tractability of the state space model.
In Appendix~\ref{subsec:selecton},
we derive the closed form solution (when $\lambda_{+} = \lambda_{-}$)
and the series solution (when $\lambda_{+} \neq \lambda_{-}$)
for MGFs of the SSM integral process.
In Appendix~\ref{subsec:ssmfilter}, we use them 
to design the Gaussian filter 
(suitable when $\epsilon$ is small) and
the Gaussian sum filter (suitable when $\epsilon$ is large)
for SSM. Those results from the direct filtering of SSM are 
then to be used as the reference solutions in subsequent experiments.
We emphasize that the presence of these
reference probability distributions
enables very careful examination of filter accuracy
in our numerical experiments, beyond measuring 
the distance between a realization of the truth signal 
and the mean of an approximate filtering solution and 
beyond what is seen
in most other works concerning the computational evaluations of filters;
this in turn gives further depth to our demonstrations.

In all our experiments, we use the following parameter values to specify the
SSM truth model: 
$\sigma_u=0.1549$,
$\gamma_{+}=2.27$,
$\gamma_{-}=-0.04$,
$\lambda_{+}=1$ and
$\lambda_{-}=2$
(these choices follow those in \cite{gershgorin2010test}).
Fixing inter-observation time
$T=1$,
we study the cases of
$\epsilon=10^{-1}$, $10^{0}$, $10^{1}$, $10^{2}$.
Each one 
is selected as representative of the parameter regimes:
sharp scale-separation,
imprecise scale-separation,
moderately rare-event,
extremely rare-event,
in the order given.
Since $\mathbb{E}(\tau_k) = 1/r$
for $\tau_k \sim \exp(r)$,
the reciprocal of $\epsilon$ 
equals the average number of transitions from 
the stable mode ($\gamma = \gamma_{+}$)
to the unstable mode ($\gamma = \gamma_{-}$)
on the unit time interval.
As $\lambda_{-}$ is twice $\lambda_{+}$ in this example,
the average time 
spent in the stable mode is twice that spent in the unstable mode.

We take the initial condition of SSM
according to $u_0 \triangleq \mathcal{N}(0.1, 0.0016)$
and $\gamma_0 \triangleq \gamma_\infty$,
independently from one the other.
For MSM (dMSM), we take
$\bar{u}_0( \bar{u}'_0) \triangleq u_0$.
We also take $\bar{\gamma}=\bar{\gamma}_\infty (= 1.5) $
and $\bar{\gamma}'\triangleq {\gamma}_\infty$.
For DSM (dDSM),
we take the independent Gaussian
$(\widehat{u}_0,\widehat{\gamma}_0)$
(or $(\widehat{u}'_0,\widehat{\gamma}'_0)$)
where
$\widehat{u}_0 \triangleq u_0$
and 
$\widehat{\gamma}_0 
\triangleq \mathcal{N}( 
1.2
\bar{\gamma}_\infty,
\text{Var}(\gamma_\infty))$.
We set $\widehat{\rho}^{\pm}(0)=\bar{\rho}_{\pm}$.
For the observational process in Eq.~(\ref{eq:obs}),
we use
$R_n=0.25 E$
where $E \equiv \sigma_u^2/(2\bar{\gamma})$
(in this case the variance of $\bar{u}_n \vert Y_n$
is independent of $n$).

\subsection{Performances of Simplified Filters}
\label{sec:psf}
\subsubsection{Sharp Scale-Separation Regime}
We first study the case of
$\epsilon=10^{-1}$.
For the implementation of DSM with dynamic calibration,
along with
$\Theta_{\text{naive}}$ as a starting point,
a local minimizer 
$\Theta_{\text{dynamic}}$
of Eq.~(\ref{eq:qdf})
is solved
at every observation time.
The choice of $\kappa$ 
in $J(\epsilon)$
plays a substantial role in this problem.
Here and hereafter, 
the value of $\kappa$ is set to zero
for simplicity and consistency of presentations; this allows
the prior mean from dynamic and static calibrations to
be inaccurate but, in filtering, the posterior is the 
main object of interest.
The time average of these parameters for $1\leq n \leq50$ 
is taken as
$\Theta_{\text{static}}$.

In addition to DSM filters,
we apply Gaussian filters for MSM and SSM.
For the latter,
due to distinct $\lambda_{\pm}$,
we need to truncate the series solution 
of the MGF.
Hereafter, the first $30$ terms of the series solution 
will be kept as this ensures accuracy 
by virtue of the fact that $\mathbb{P}(N_T > 30)< 10^{-5}$.

In Fig.~\ref{fig1ep1}, we depict the relative errors 
of the prior 
and posterior 
approximations
in terms of mean and variance.
We see that the approximations of DSM
with 
the parameters tuned by our methods
are significantly more accurate than the MSM approximation.
As expected,
the overall 
errors of the mean and variance relative to those from SSM filtering solution
are given in the order :
DSM (dynamic calibration) $\lesssim$
DSM (static calibration) $<$
DSM (naive set) $<$
MSM.
Admittedly,
this result is merely for a single realization of the observation process.
However we show that the result is 
indeed
robust with respect to 
the chosen observational data set 
in the following manner.

At each observation time step,
the posterior distributions of the 
approximate 
models are determined by the instance of observation,
which is drawn from a Gaussian.
In Fig.~\ref{fig2ep1}, we depict 
the dependence of the corresponding filter accuracy on $y_n$
for $n=20$ and $n=40$.
It is observed that,
for most values 
of $y_n$, Gaussian filters for DSM with dynamic and static calibrations 
significantly outperform MSM,
leading to
highly accurate posterior approximations.
In Fig.~\ref{fig2ep1}, we also depict
the statistical average of the posterior error with respect to $y_n$ for each $n$.
There, one can see the ordering of the accuracies
is exactly the same as in the single realization experiment.

\subsubsection{Imprecise Scale-Separation Regime}
Taking $\epsilon=10^{0}$,
it is not immediately intuitive whether 
either the 
Gaussian description or the Gaussian mixture description
is a better approximation of the SSM.
It turns out that, in this case,
the Gaussian filter for SSM is more suitable as the 
reference solution; 
our investigation of this issue can be found in 
subsection~\ref{sec:sa}.
Accordingly we 
find dynamic and static calibrations,
and
implement Gaussian filters for 
DSM, MSM and SSM.
We depict
Fig.~\ref{fig1e1}
and
Fig.~\ref{fig2e1},
which correspond, respectively, to
Fig.~\ref{fig1ep1}
and
Fig.~\ref{fig2ep1}.
The scenario interpreted from the figures is 
basically 
the same as
the one 
arising
when $\epsilon=10^{-1}$, with one exception that
the naive DSM is less accurate than the MSM.
This is 
no surprise,
because $\epsilon$ is no longer small and 
$\Theta_{\text{naive}}$ is no longer expected to be valid.
Therefore,
the overall errors are ordered as:
DSM (dynamic calibration) $\lesssim$
DSM (static calibration) $<$
MSM $<$
DSM (naive set).

\subsubsection{Moderately Rare-Event Regime}
When $\epsilon=10^1$, it is shown in 
subsection~\ref{sec:sa} that 
the Gaussian sum filter for SSM,
made efficient by merging the mixture approximation of the
posterior into a Gaussian at every observation time,
is indeed better than the Gaussian filter
for the reference solution.
We apply the same kind of Gaussian sum filters for dMSM and dDSM.
For the dDSM implementations, taking
$\Theta'_{\text{naive}}$ as a starting point,
we solve
dynamic calibrations 
for each of two evolving Gaussian kernels.
We then individually average them to obtain a static calibration.

In Fig.~\ref{fig1e10},
we depict the relative error for each of the
Gaussian kernel approximations.
Combining these two cases, we plot
Fig.~\ref{fig2e10} and Fig.~\ref{fig3e10},
which correspond to Fig.~\ref{fig1e1} and
Fig.~\ref{fig2e1} respectively. Importantly,
for comparison, we additionally plot the result from DSM with
$(\mu=\bar{\gamma}_\infty, \nu=0.1\bar{\gamma}_\infty,\sigma=5\sigma_u)$.
These parameters are the ones used in \cite{gershgorin2010test}.
They are selected as suitable from direct numerical simulations
in this parameter regime, and are interestingly very close to 
$\Theta_{\text{naive}}$.
Here the DSM appears as a reasonable approximation of SSM
but
this Gaussian filter 
is characterized by
significantly less accuracy than 
the remaining Gaussian sum filters.

Our simulations further
indicate, in this case, that
the dependency of filter accuracy on the observation
is much more complicated than the previous Gaussian filtering examples
(Fig.~\ref{fig3e10}).
The overall errors are of the order :
dDSM (dynamic calibration) 
$\lesssim$
dDSM (static calibration) 
$<$
dMSM 
$\lesssim$
dDSM (naive set)
$<$
DSM (naive set).

\subsubsection{Extremely rare-event regime}
Like the preceding case, 
we take as reference
the Gaussian sum filter for SSM 
with projection of posterior into the set of Gaussian distributions.
The overall scenario when $\epsilon=10^2$ is similar to the case with $\epsilon=10^1$,
except that dMSM becomes more accurate.
We plot
Fig.~\ref{fig1e100},
Fig.~\ref{fig2e100}
and
Fig.~\ref{fig3e100}
that correspond to
Fig.~\ref{fig1e10},
Fig.~\ref{fig2e10}
and
Fig.~\ref{fig3e10}
respectively.
We see
the overall errors are of the order :
dDSM (dynamic calibration) 
$\simeq$
dMSM 
$<$
dDSM (static calibration) 
$<$
dDSM (naive set).
Note dMSM is quite accurate in this case because $\epsilon$ is very large.

\subsubsection{Summary} 
To summarize we plot the root mean square errors 
of mean and variance
between the reference and approximations 
for all four choices of $\epsilon$ in Fig.~\ref{fig:rmse}.

\subsection{Supplementary Analysis}
\label{sec:sa}
This section discusses our choices, especially in relation  to choice
of reference solution,  made while performing numerical simulations 
in subsection~\ref{sec:psf}; it
can be skipped without harming the understanding of the main messages
of the paper.

\subsubsection{Imprecise Scale-Separation Regime}
In this case where we take $\epsilon=10^{0}$,
to make sure
whether either the 
Gaussian description or the Gaussian sum description
is a better approximation of the SSM,
what we do is to compare the similarity/distance between 
MSM (note the derivation corresponds to $\epsilon=0$) and Gaussian approximation of SSM,
and that between 
dMSM (that corresponds to $\epsilon=\infty$) and Gaussian sum approximation of SSM.

To that end,
we plot the prior distributions from all four cases,
when $n=10$ (and we do the same in the remaining examples),
in 
the left panel of
Fig.~\ref{fig:gsae1}.
We see that the
dMSM has a one-sided fat tail,
which is due to the contribution 
by the Gaussian kernel evolved while $\gamma$ is in the unstable mode.
However this feature is not apparent in the mixture approximation of SSM
(in fact both Gaussian and Gaussian mixture approximations of
SSM are very similar and unimodal).
Furthermore, the $L^1$ distance between MSM and SSM (Gaussian) is 
significantly smaller than 
the one between dMSM and SSM (Gaussian mixture), as shown in the right panel of
Fig.~\ref{fig:gsae1}.
The discussion demonstrates that, in this parameter regime,
the Gaussian filter for SSM is
more suitable as a reference solution than is the
Gaussian sum filter.

\subsubsection{Moderately Rare-Event Regime}
When $\epsilon=10^1$, 
we plot the four relevant prior distributions
in 
the top-left panel of
Fig.~\ref{fig:gsae10100} .
While MSM and SSM (Gaussian)
are distant from one another,
both dMSM and SSM (Gaussian mixture) are characterized by a one-sided fat tail,
in contrast to the case of $\epsilon=10^0$,
and further are very close to one another.
Therefore SSM with the Gaussian sum filter is chosen as 
the appropriate reference solution.

We turn our attention to the validity of Gaussian approximation of the 
Gaussian mixture
posterior.
The top-right panel of
Fig.~\ref{fig:gsae10100} depicts
the posterior of SSM (Gaussian mixture),
which consists of two kernels.
The distribution 
is well approximated by a single Gaussian that has the same mean and variance.
This is due to the sharpness of the likelihood we choose (discussed shortly).
We may thus approximate the filtering solution by a Gaussian 
at every observation time, 
and we can apply Gaussian sum filters 
in a computationally tractable way
without harming accuracy.

\subsubsection{Extremely Rare-Event Regime}
With regard to SSM filter,
the scenario when $\epsilon=10^2$ is the same as the case with $\epsilon=10^1$.
In the bottom of
Fig.~\ref{fig:gsae10100},
the priors of dMSM and SSM (Gaussian sum)
are almost indistinguishable,
and the SSM posterior is accurately approximated by a Gaussian.

We conclude the current section with 
further study of the Gaussian approximation of the posterior.
Recall we have fixed $R_n=0.25E$ thus far.
In this case,
it is shown that
the Gaussian approximation of the posterior can be performed without losing accuracy,
but this may not be the case when $R_n$ is bigger.
In Fig.~\ref{fig:gsa2e100}, we plot the prior and posterior with $R_n=0.75E$.
Due to the flatter likelihood, the posterior with two kernels significantly deviates 
from the Gaussian approximation.
In this case, 
the Gaussian approximation of the posterior cannot 
guarantee the accuracy of the filtering solution.

\begin{figure}
\vspace{-0.1in}
  \centering
\subfigure
{\includegraphics[width=0.46\textwidth]{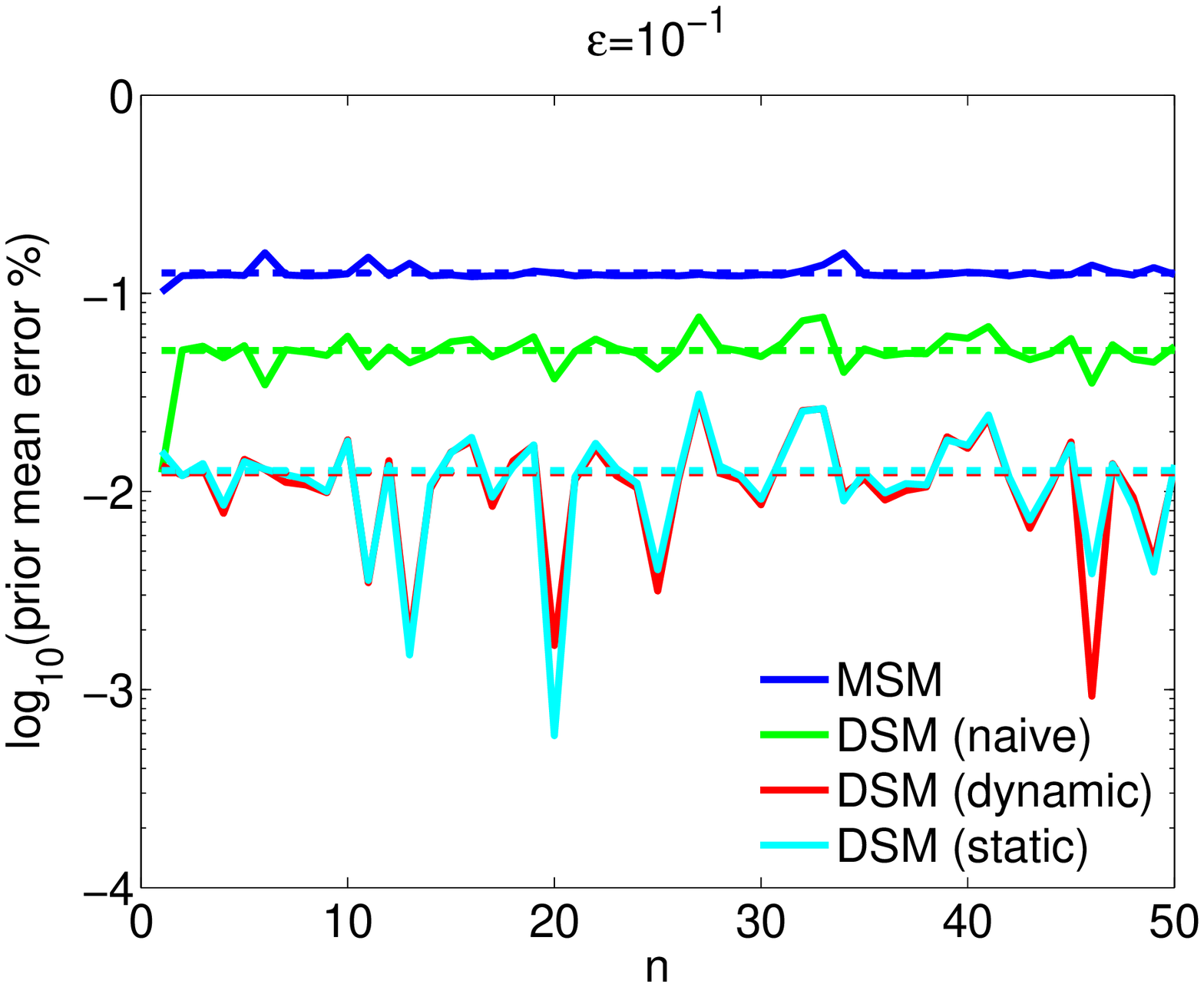} } 
\quad
\subfigure
{\includegraphics[width=0.46\textwidth]{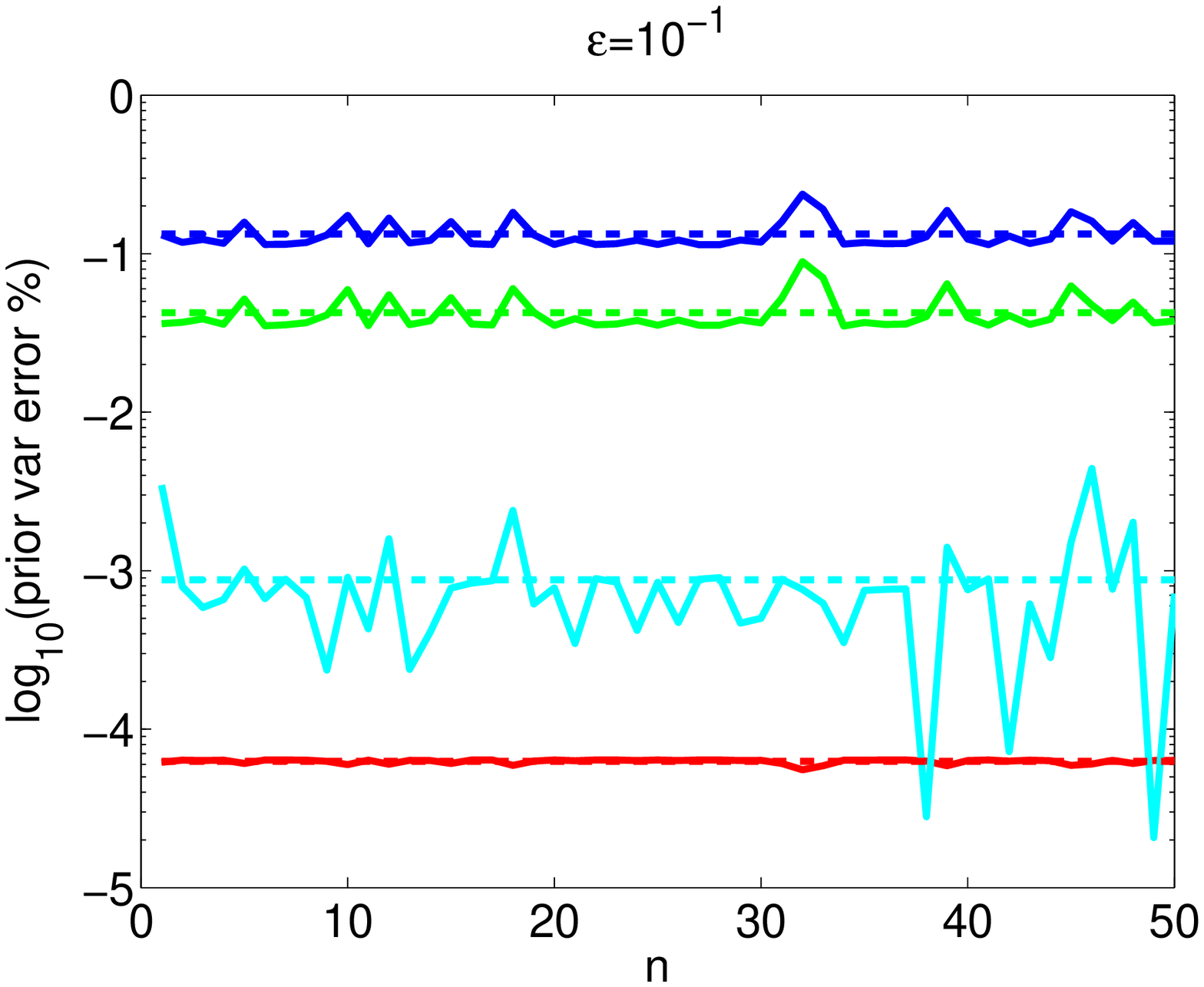} }
\\
\vspace{-0.1in}
\subfigure
{\includegraphics[width=0.46\textwidth]{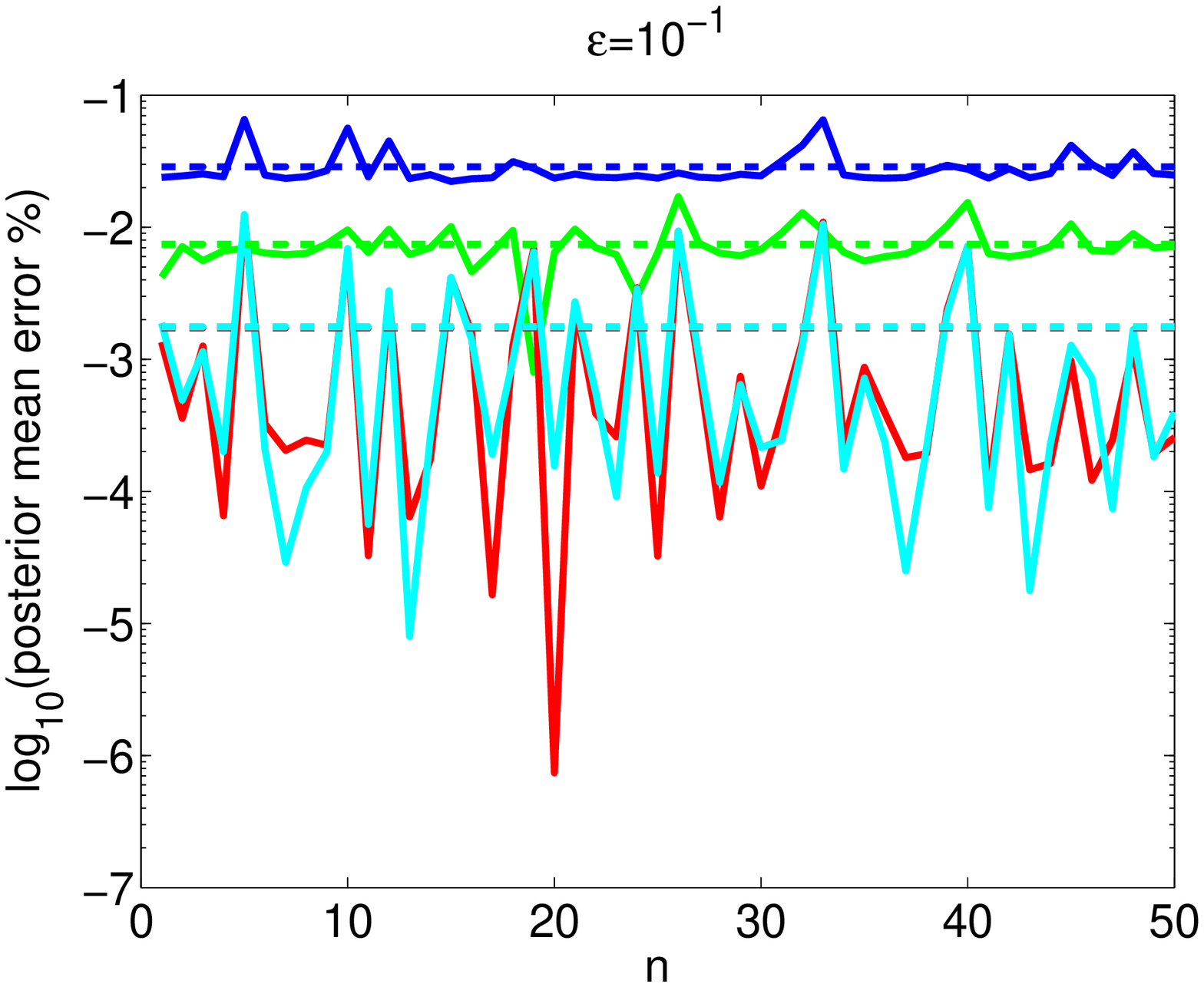}} 
\quad
\subfigure
{\includegraphics[width=0.46\textwidth]{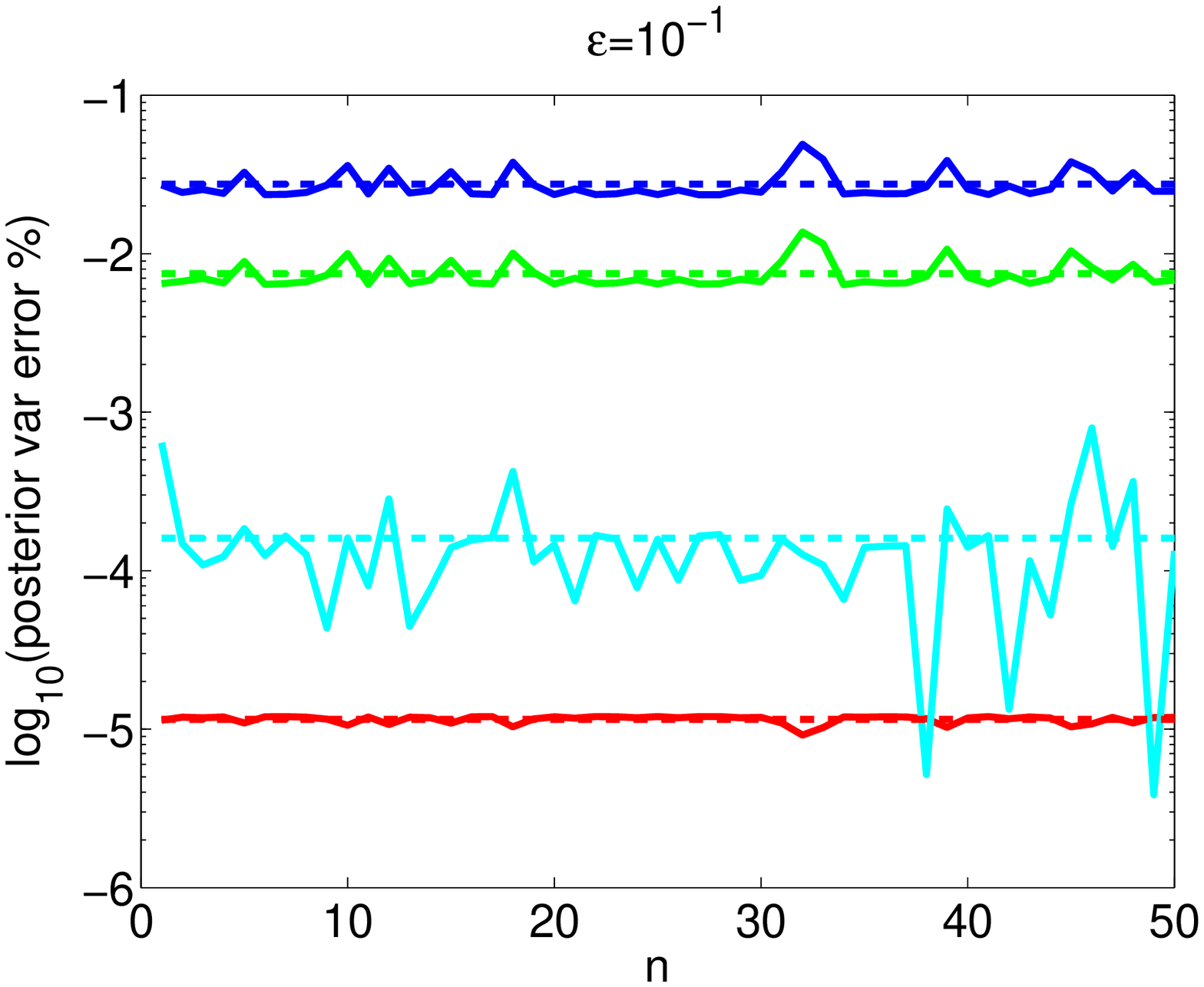}} 
\\
\vspace{-0.1in}
\caption{
The relative errors of the mean and variance approximations of
the prior $u_{n}|Y_{n-1}$ (top) and posterior $u_{n}|Y_{n}$ (bottom) distributions when $\epsilon = 10^{-1}$.
The dashed lines denote the time averages over $0 \leq n \leq 50$.
} 
\label{fig1ep1} 

  \centering
\subfigure
{\includegraphics[width=0.46\textwidth]{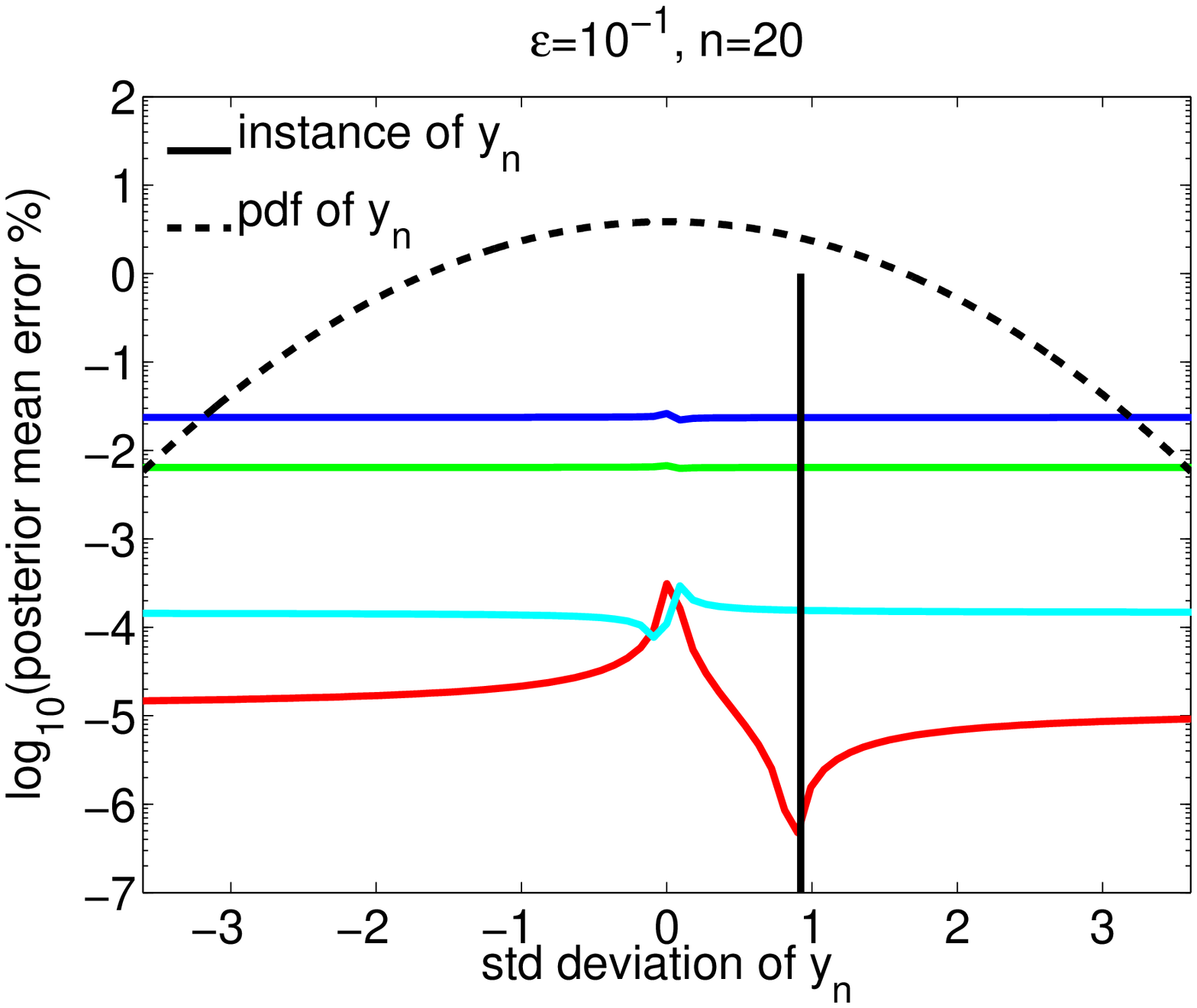} } 
\quad
\subfigure
{\includegraphics[width=0.46\textwidth]{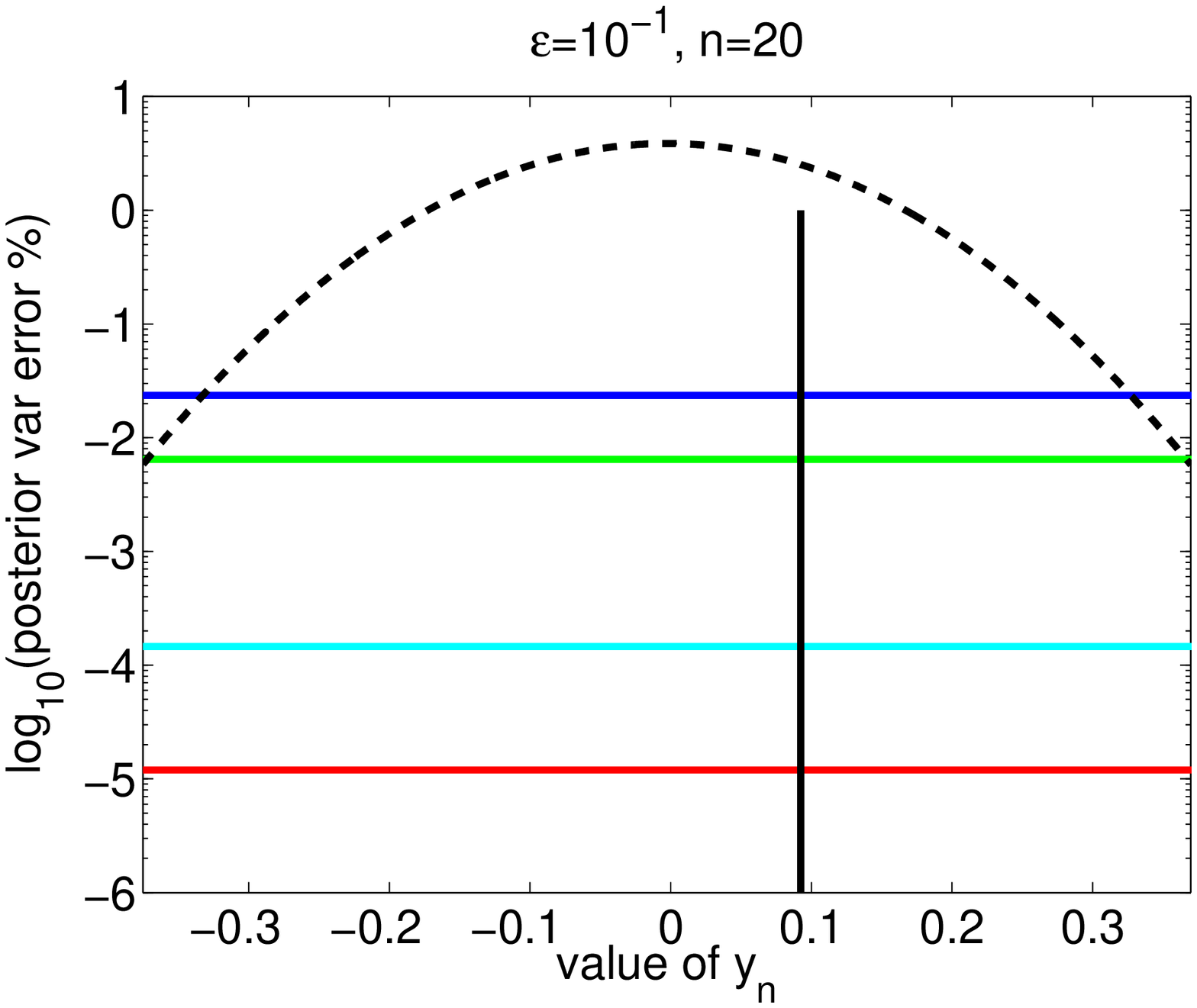} } 
\\
\vspace{-0.1in}
\subfigure
{\includegraphics[width=0.46\textwidth]{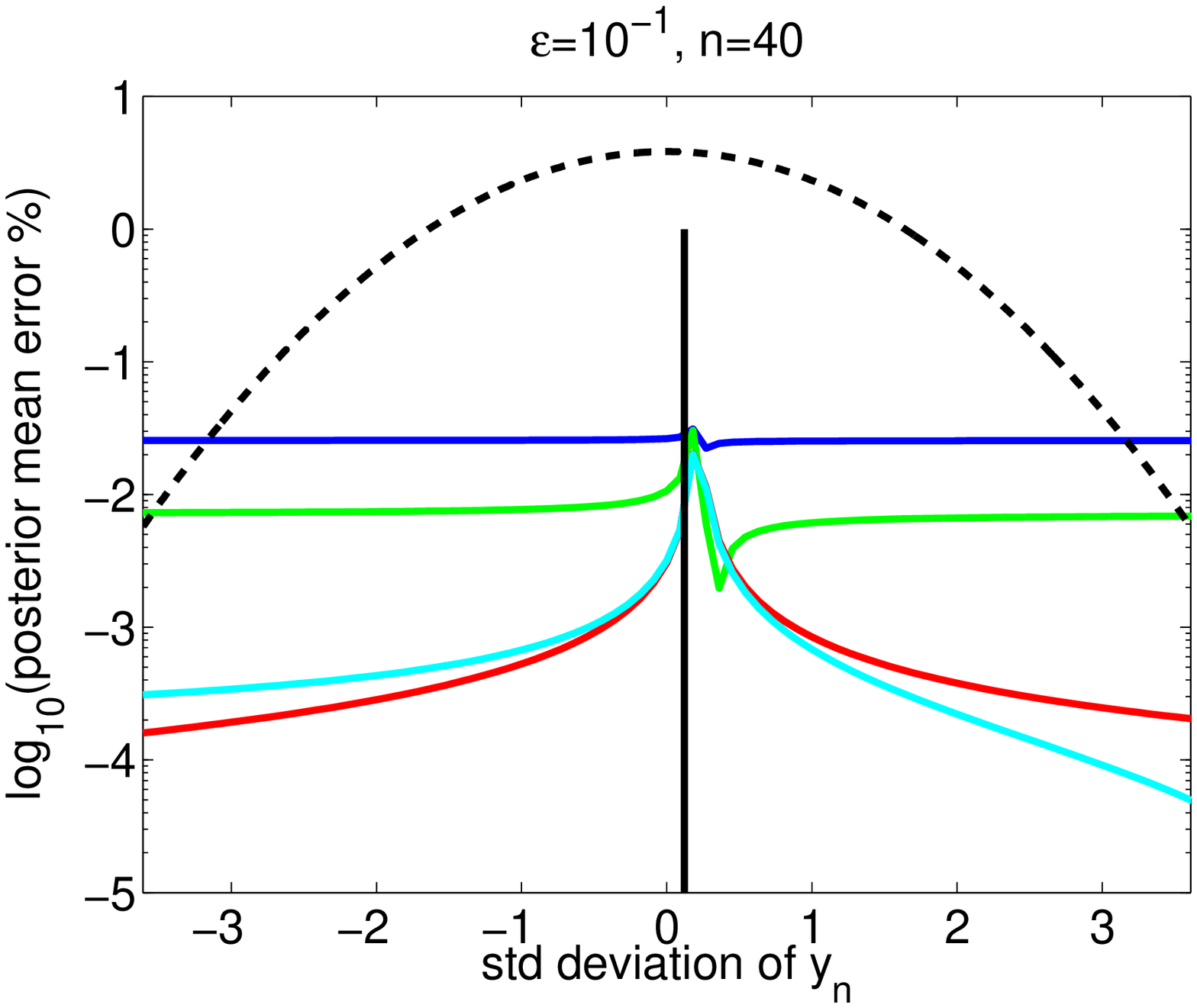} } 
\quad
\subfigure
{\includegraphics[width=0.46\textwidth]{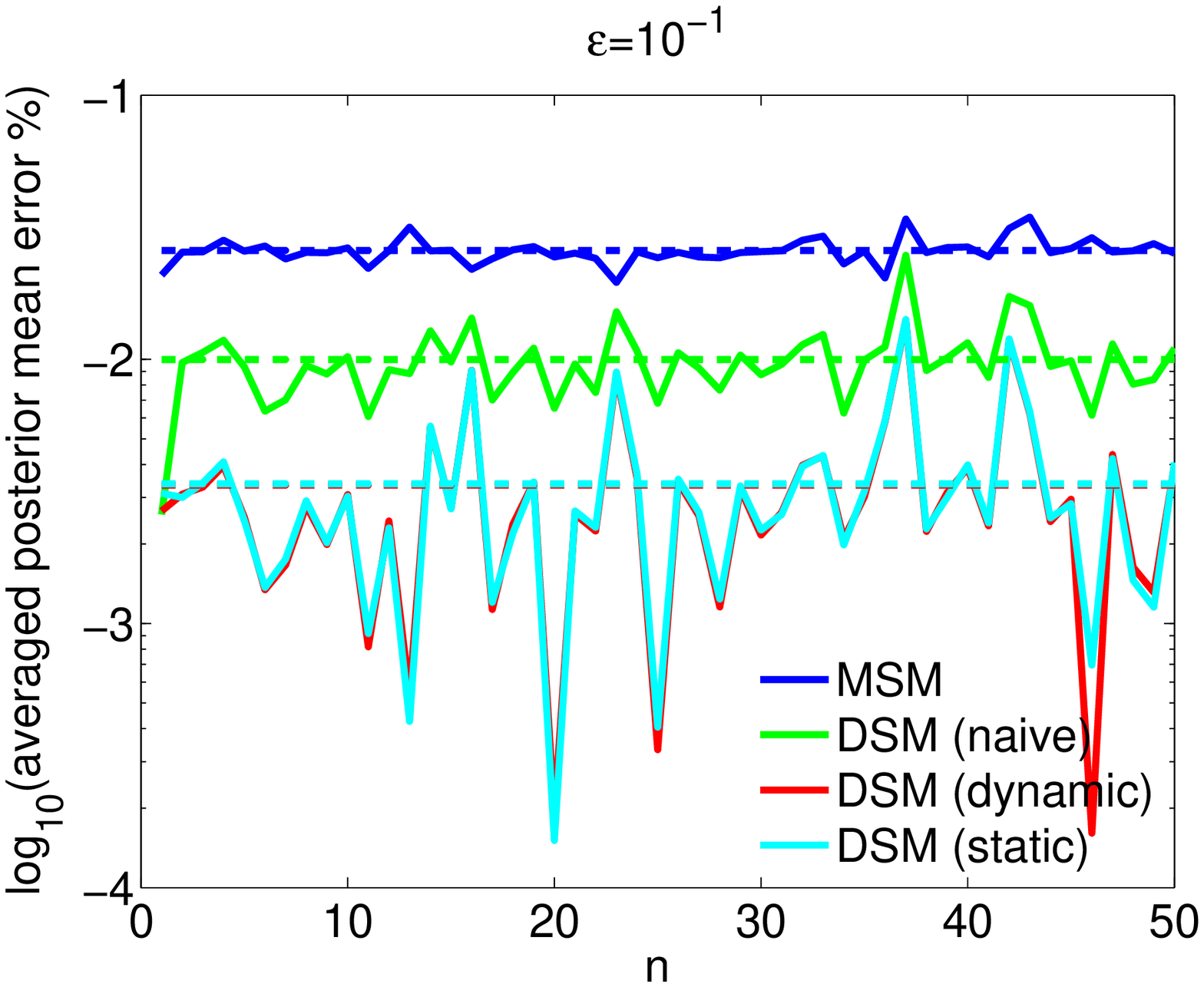} } 
\\
\vspace{-0.1in}
\caption{
The relative errors of the approximations of the posterior $u_n|Y_n$ distributions
that depend on 
the realization of $y_n= u_n + \eta_n$
(top and bottom-left),
and their statistical averages with respect to the law of $y_n$
(bottom-right)
when $\epsilon = 10^{-1}$.
In 
Gaussian filters, the accuracy of the posterior variance does not depend on the 
instance of $y_n$ (top-right).
} 
\label{fig2ep1} 
\end{figure}

\begin{figure}
  \centering
\subfigure
{\includegraphics[width=0.46\textwidth]{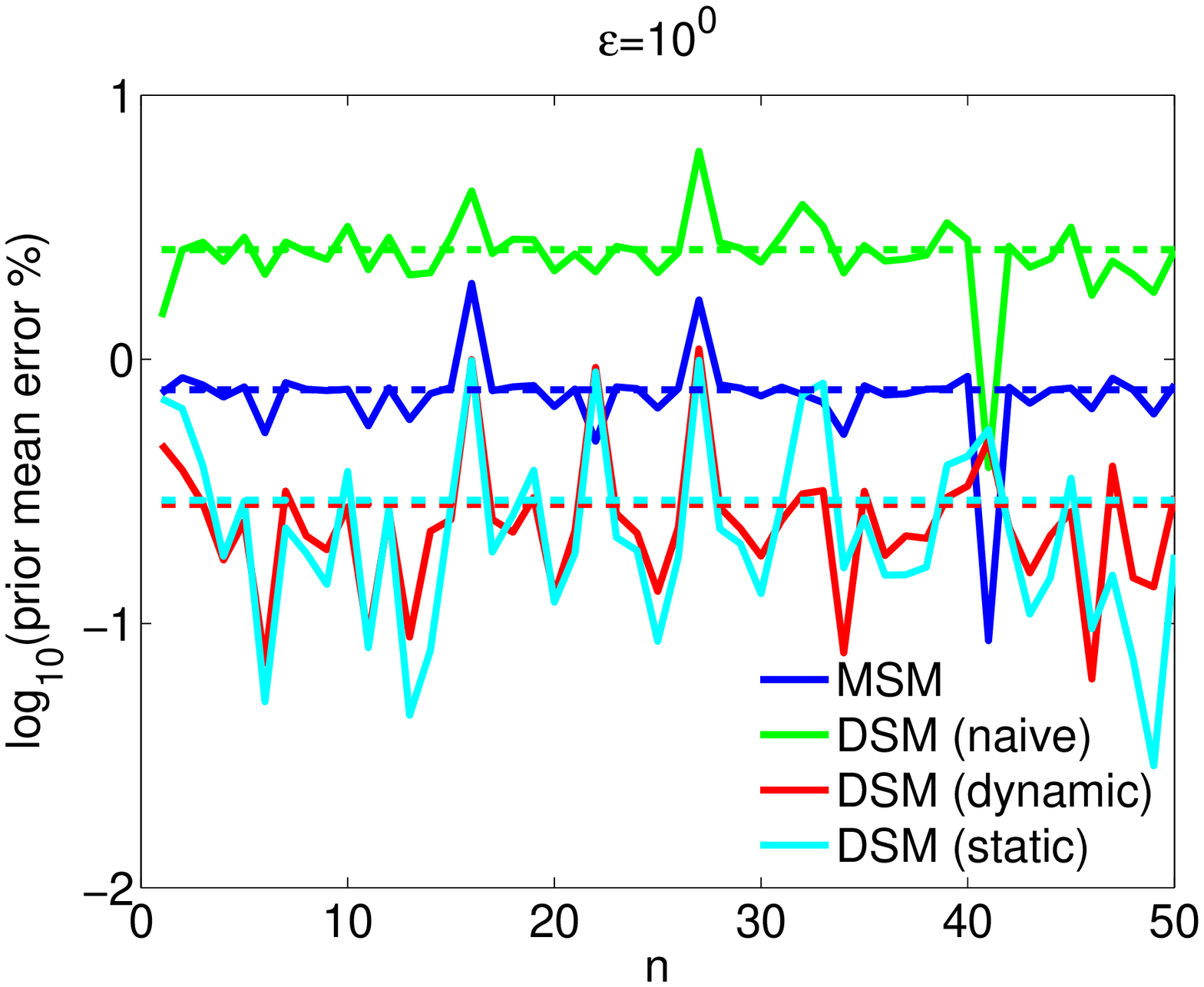} } 
\quad
\subfigure
{\includegraphics[width=0.46\textwidth]{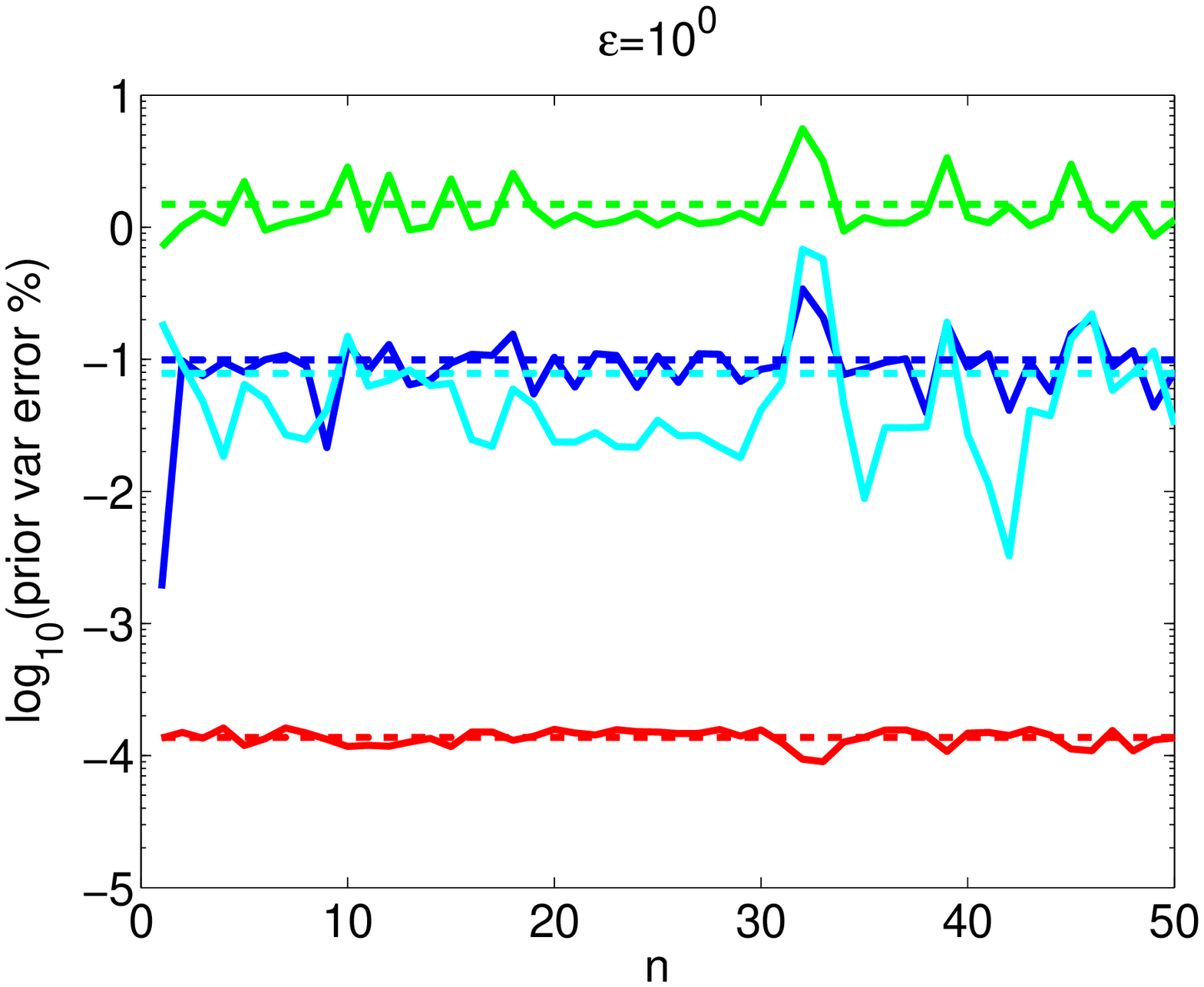} }
\\
\vspace{-0.1in}
\subfigure
{\includegraphics[width=0.46\textwidth]{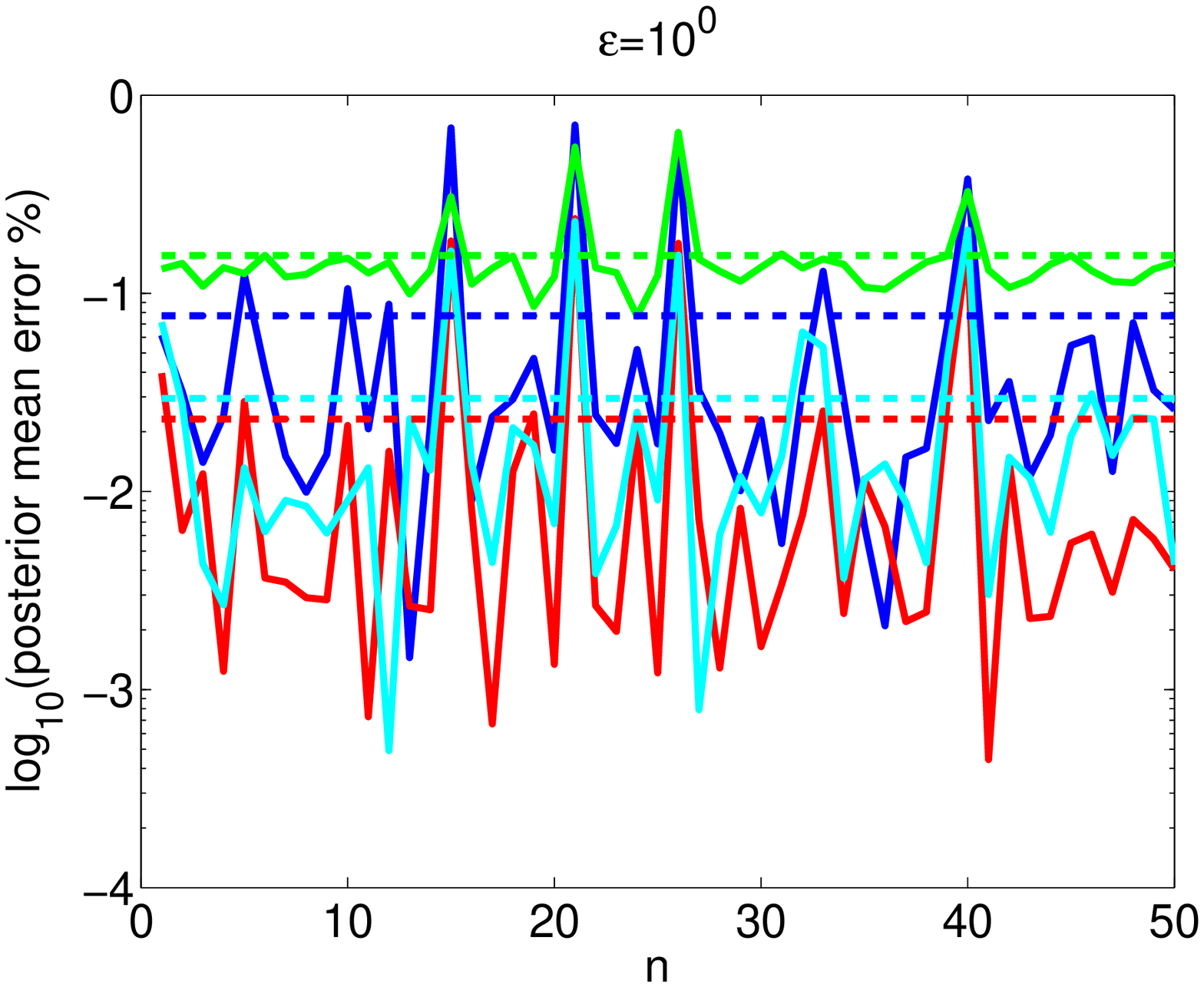}} 
\quad
\subfigure
{\includegraphics[width=0.46\textwidth]{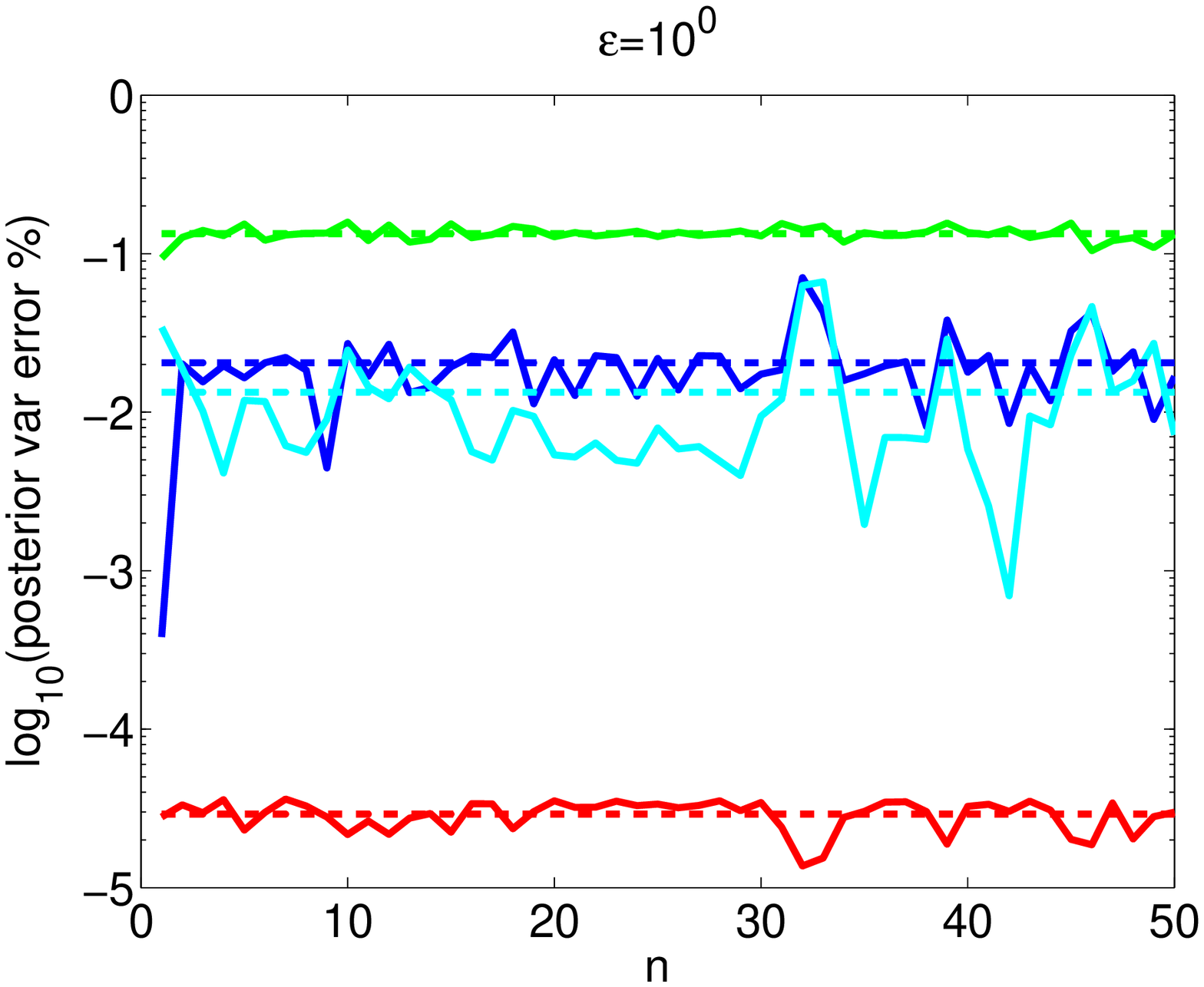}} 
\\
\vspace{-0.1in}
\caption{
The relative errors of the mean and variance approximations of
the prior $u_{n}|Y_{n-1}$ (top) and posterior $u_{n}|Y_{n}$ (bottom) distributions when $\epsilon = 10^{0}$.
} 
\label{fig1e1} 
  \centering
\subfigure
{\includegraphics[width=0.46\textwidth]{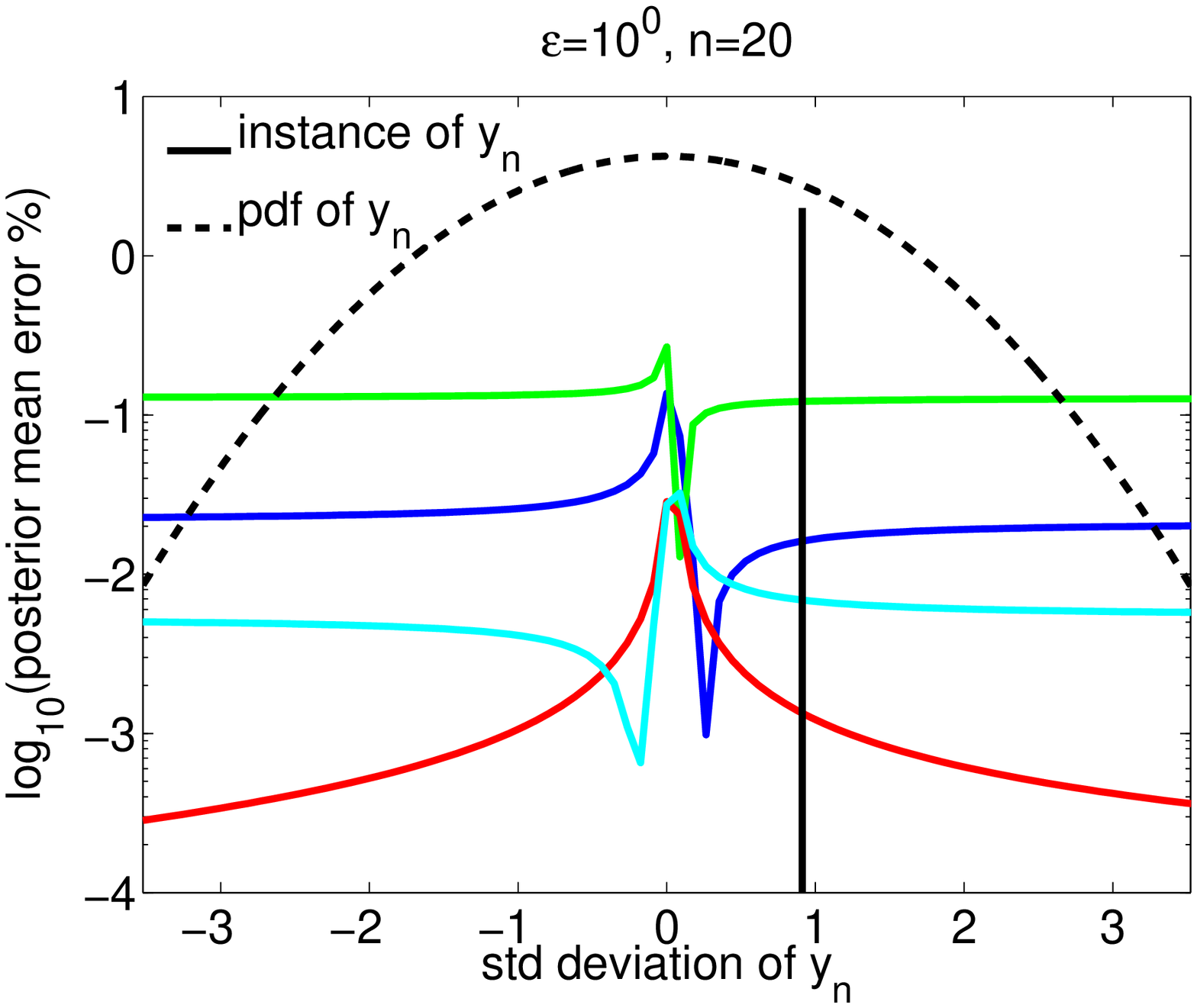} } 
\quad
\subfigure
{\includegraphics[width=0.46\textwidth]{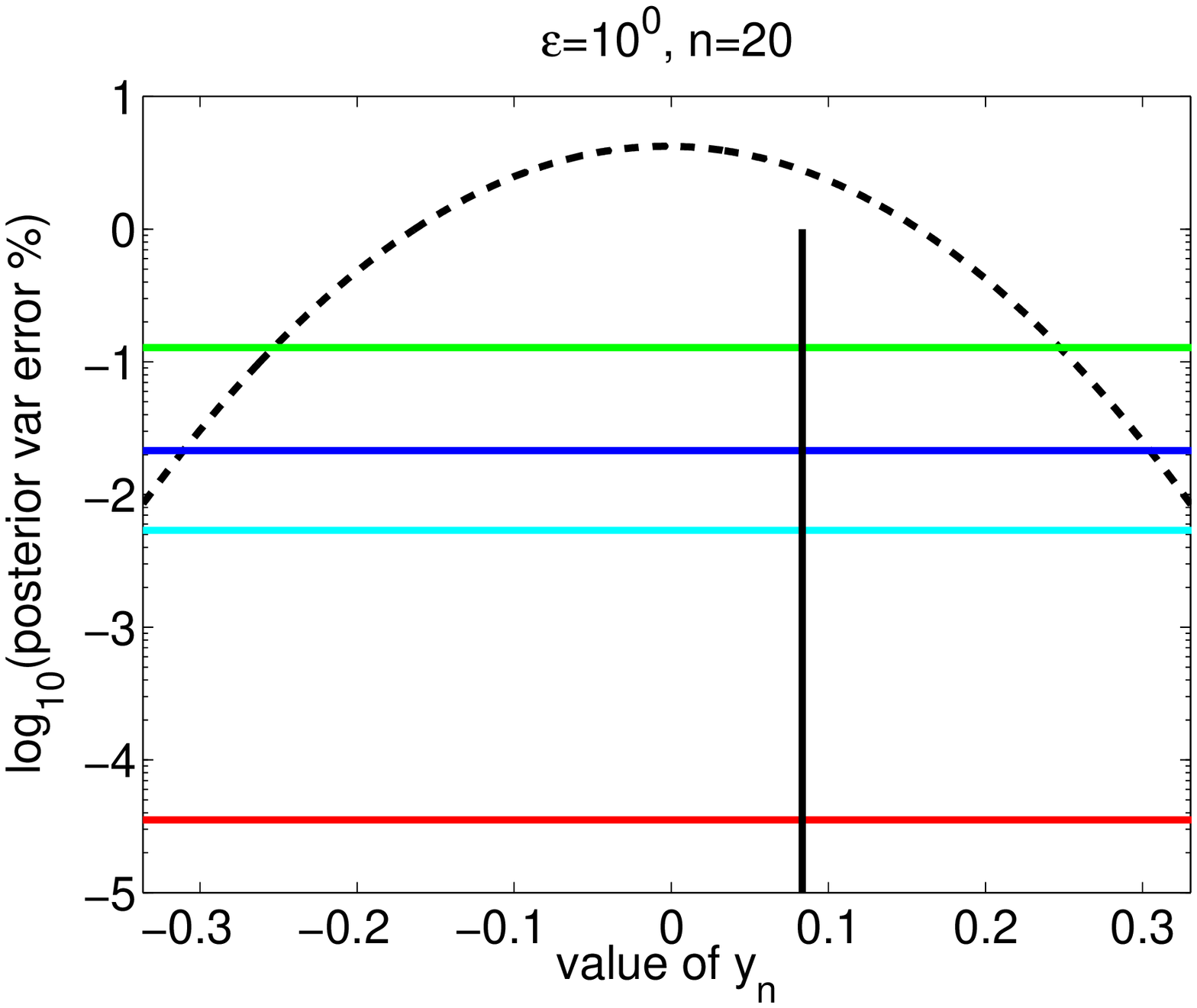} } 
\\
\vspace{-0.1in}
\subfigure
{\includegraphics[width=0.46\textwidth]{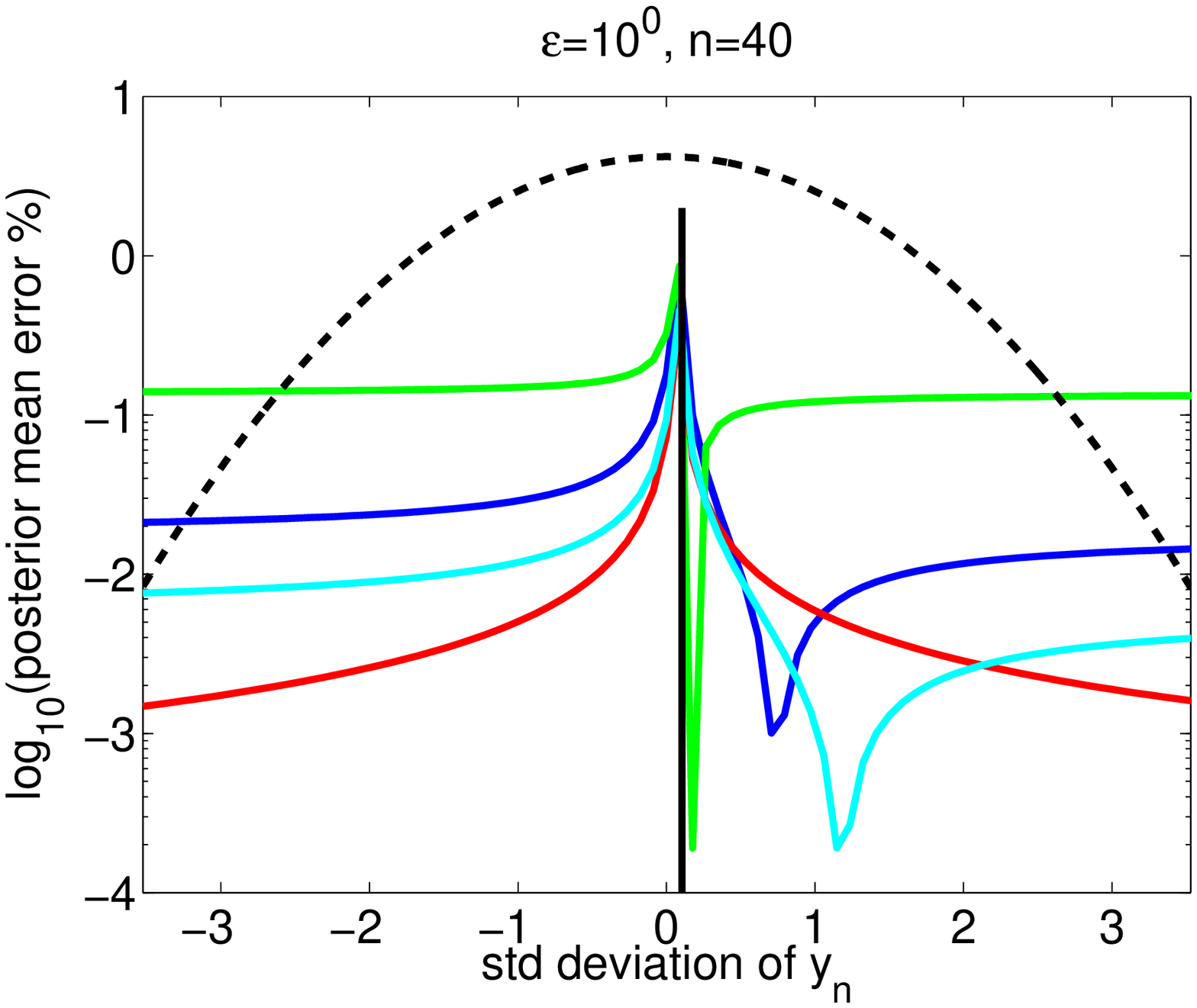} } 
\quad
\subfigure
{\includegraphics[width=0.46\textwidth]{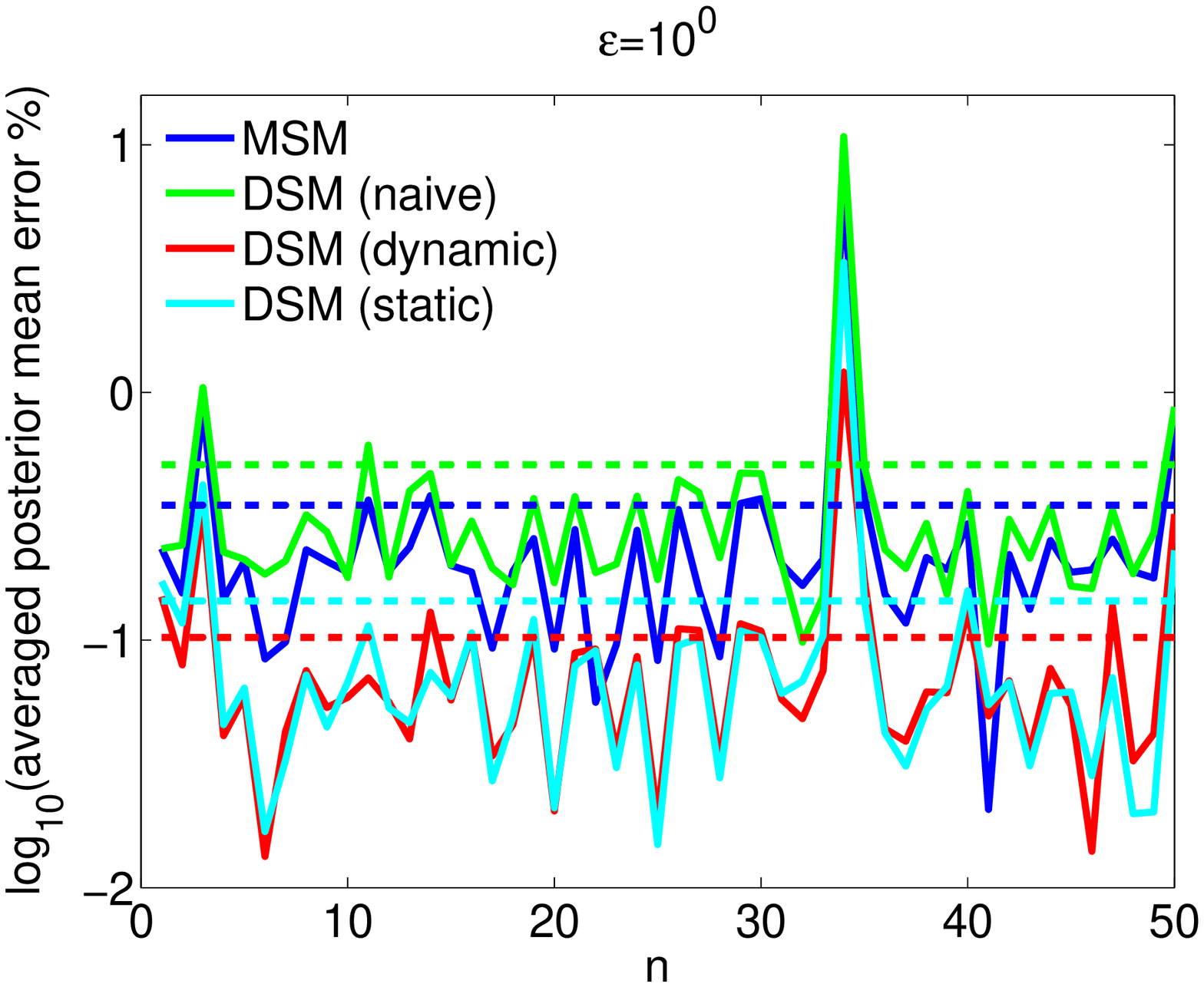} } 
\\
\vspace{-0.1in}
\caption{
The relative errors of the approximations of the posterior $u_n|Y_n$ distributions
that depend on 
the realization of $y_n= u_n + \eta_n$,
and their statistical averages with respect to the law of $y_n$
when $\epsilon = 10^{0}$.
} 
\label{fig2e1} 
\end{figure}

\begin{figure}
  \centering
\subfigure
{\includegraphics[width=0.46\textwidth]{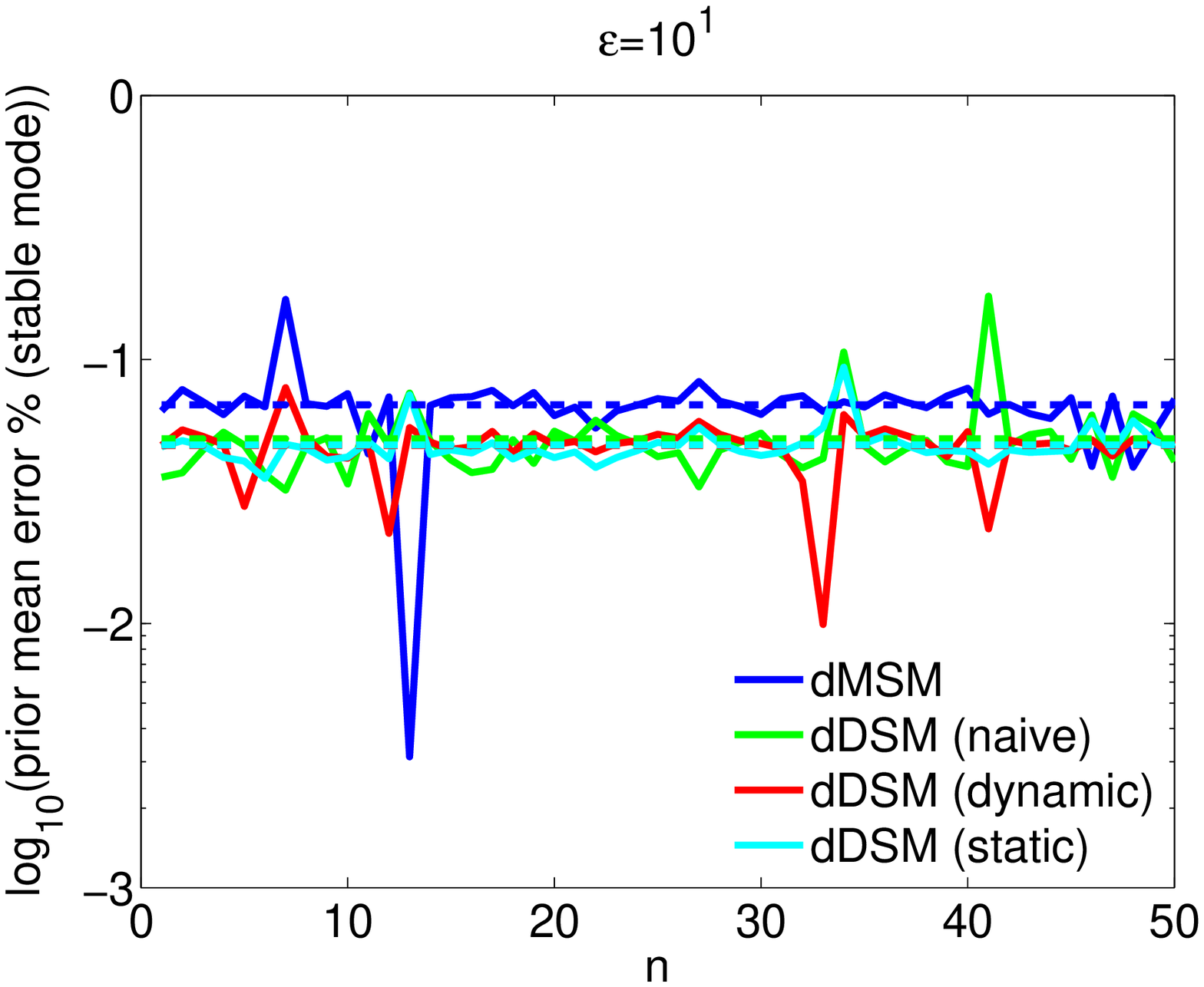} } 
\quad
\subfigure
{\includegraphics[width=0.46\textwidth]{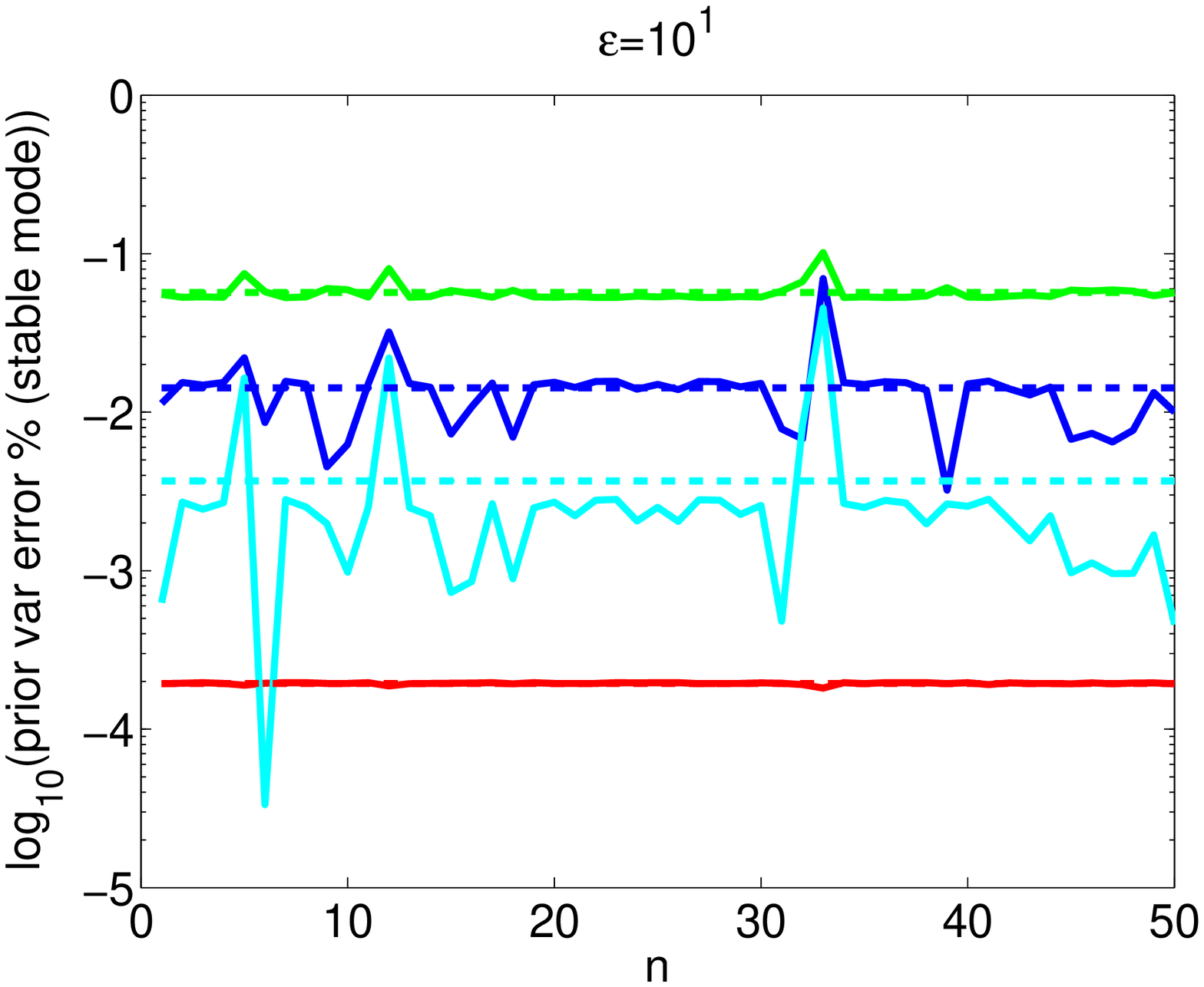} }
\\
\vspace{-0.1in}
\subfigure
{\includegraphics[width=0.46\textwidth]{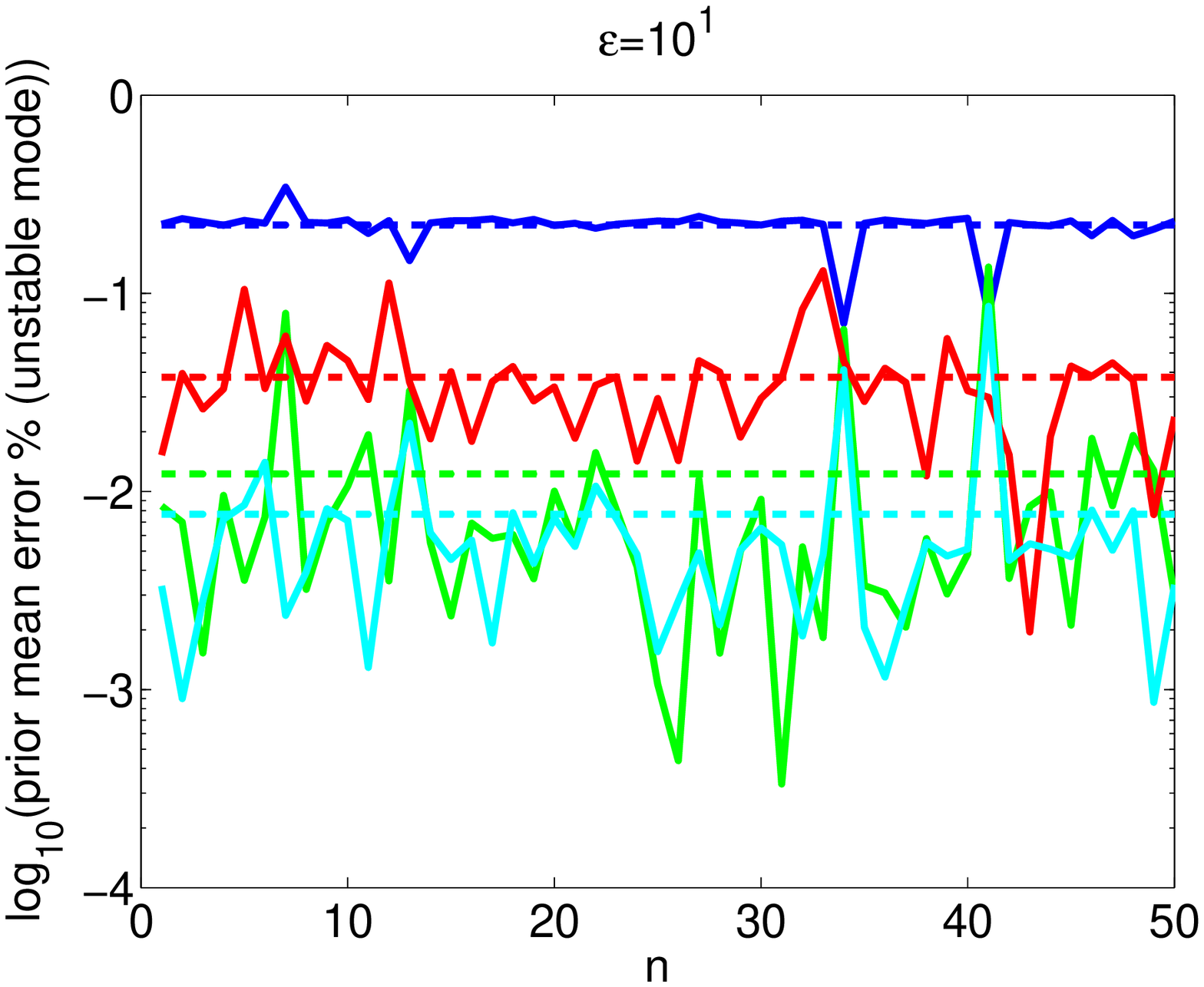}} 
\quad
\subfigure
{\includegraphics[width=0.46\textwidth]{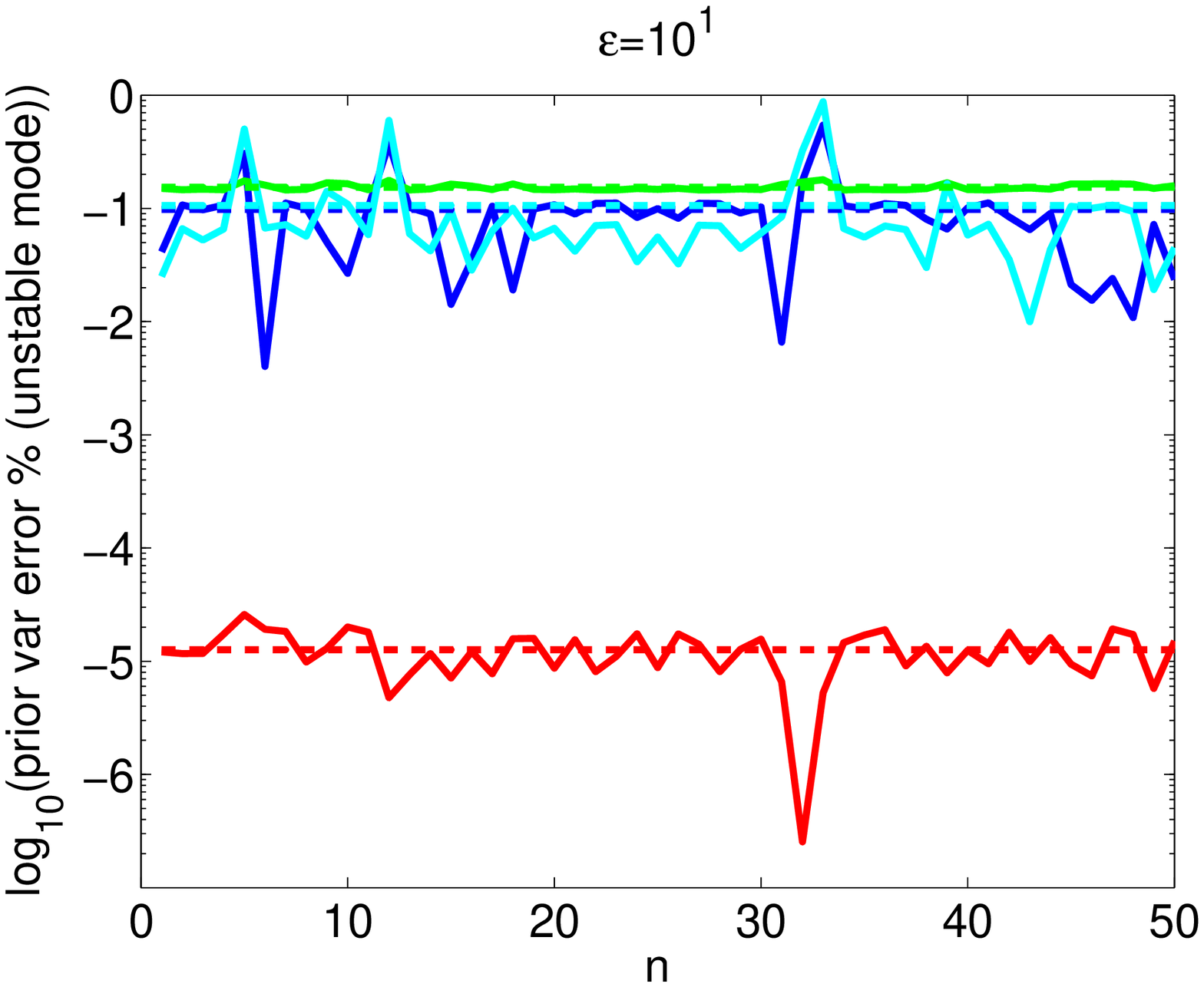}} 
\\
\vspace{-0.1in}
\caption{
The relative errors of the mean and variance approximations of
each Gaussian kernels of the prior distributions
when $\epsilon = 10^{1}$.
} 
\label{fig1e10} 
  \centering
\subfigure
{\includegraphics[width=0.46\textwidth]{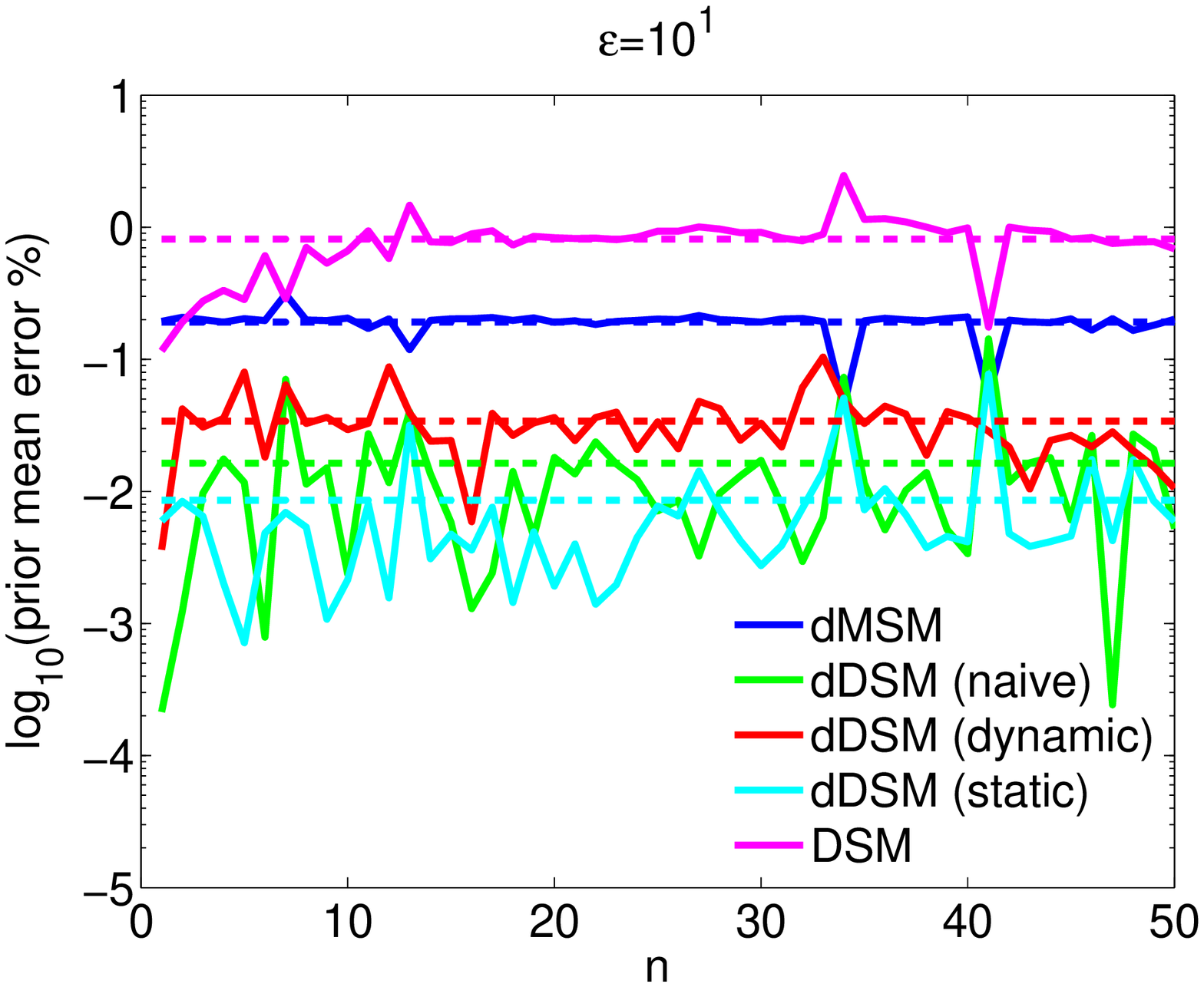} } 
\quad
\subfigure
{\includegraphics[width=0.46\textwidth]{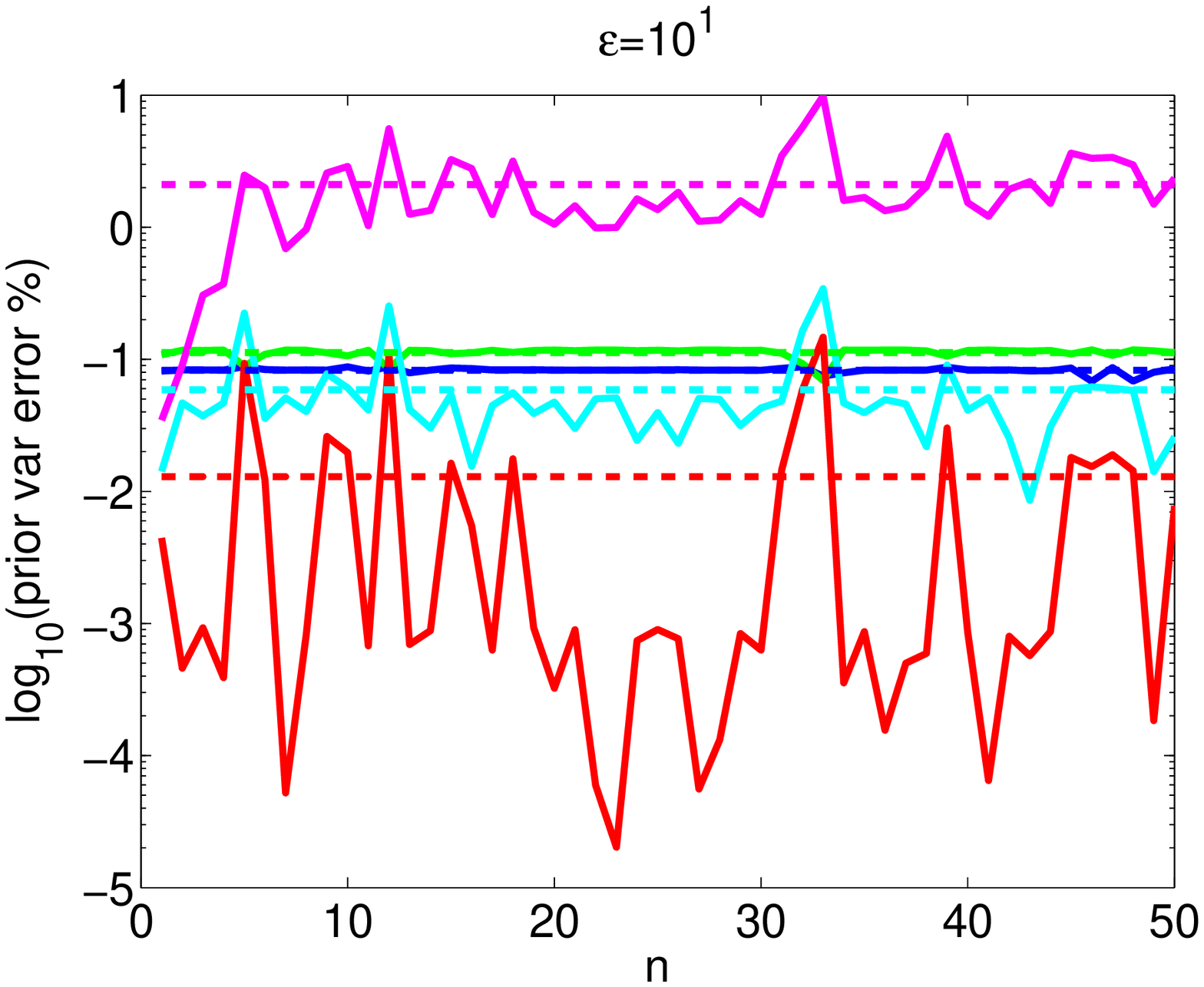} } 
\\
\vspace{-0.1in}
\subfigure
{\includegraphics[width=0.46\textwidth]{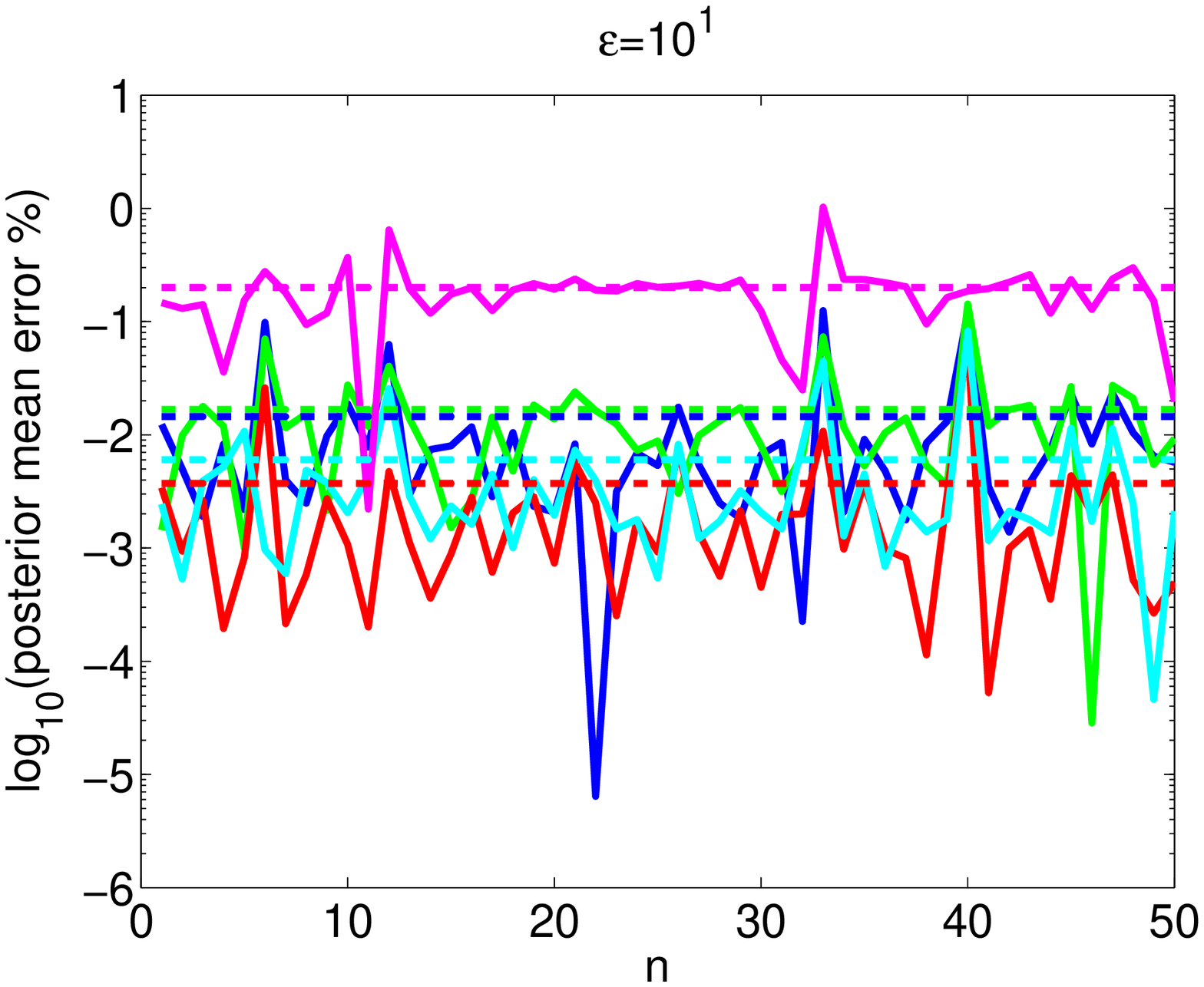} } 
\quad
\subfigure
{\includegraphics[width=0.46\textwidth]{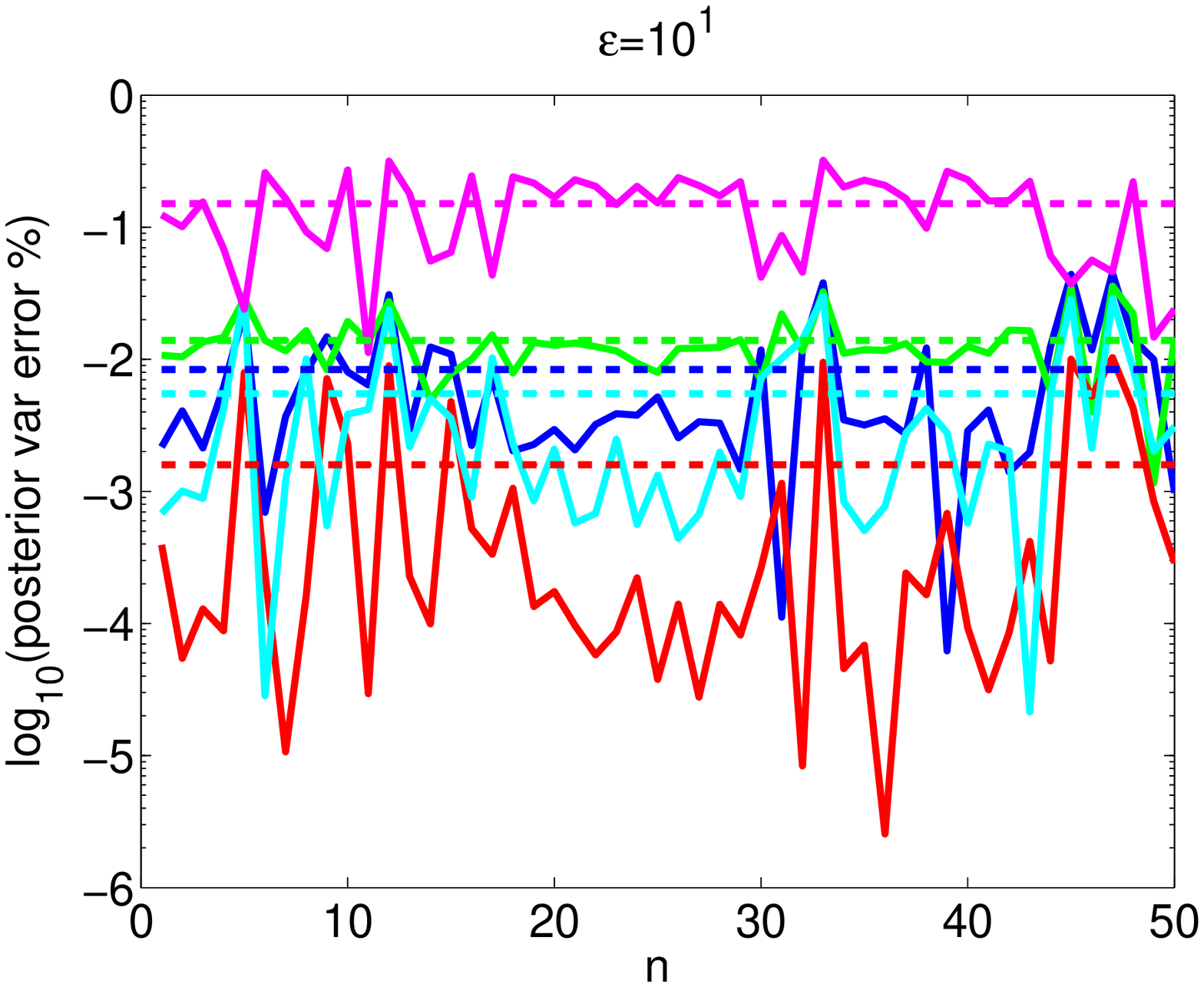} } 
\\
\vspace{-0.1in}
\caption{
The relative errors of the mean and variance approximations of
the prior $u_{n}|Y_{n-1}$ (top) and posterior $u_{n}|Y_{n}$ (bottom) distributions when $\epsilon = 10^{1}$.
} 
\label{fig2e10} 
\end{figure}

\begin{figure}
  \centering
\subfigure
{\includegraphics[width=0.46\textwidth]{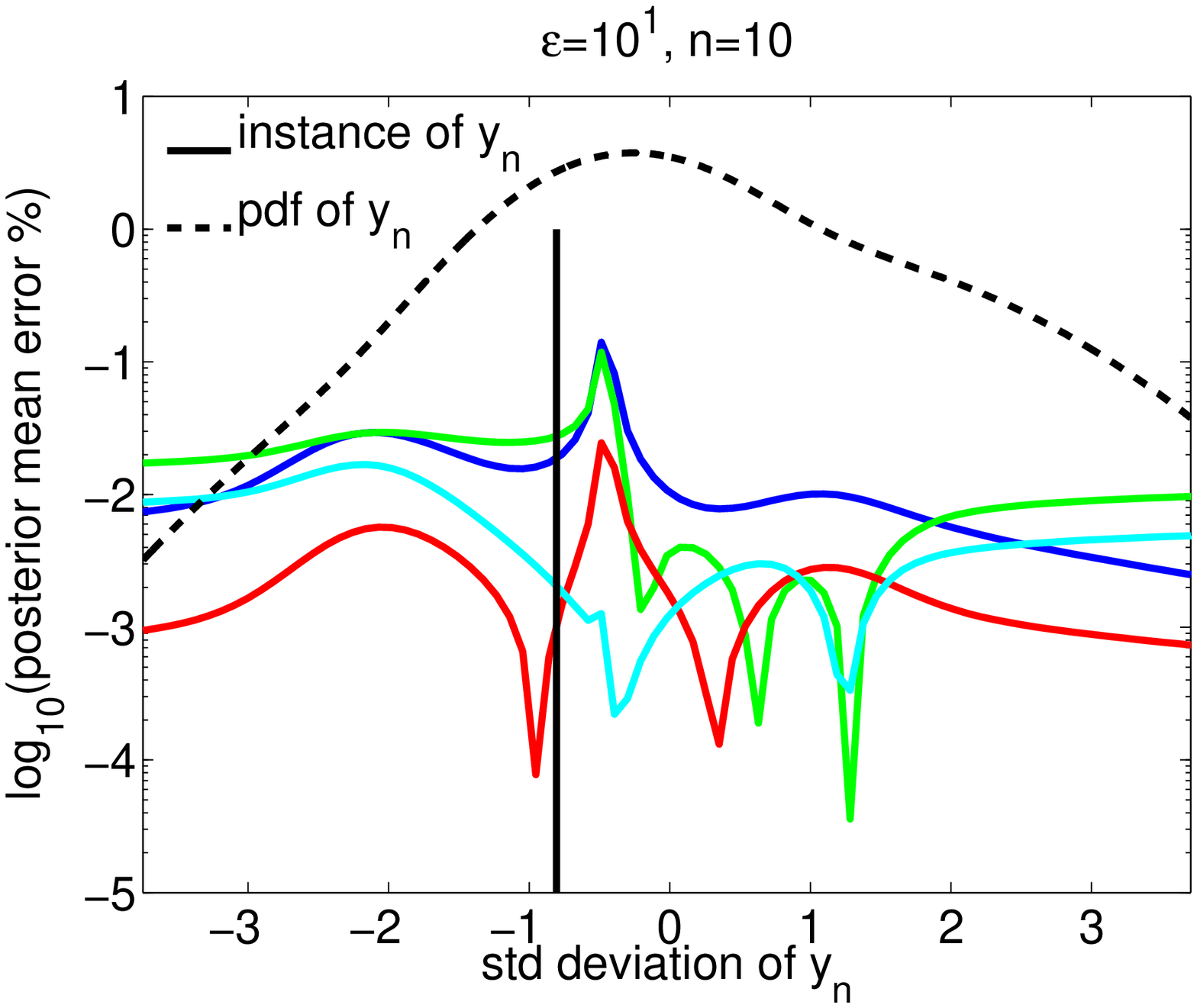} } 
\quad
\subfigure
{\includegraphics[width=0.46\textwidth]{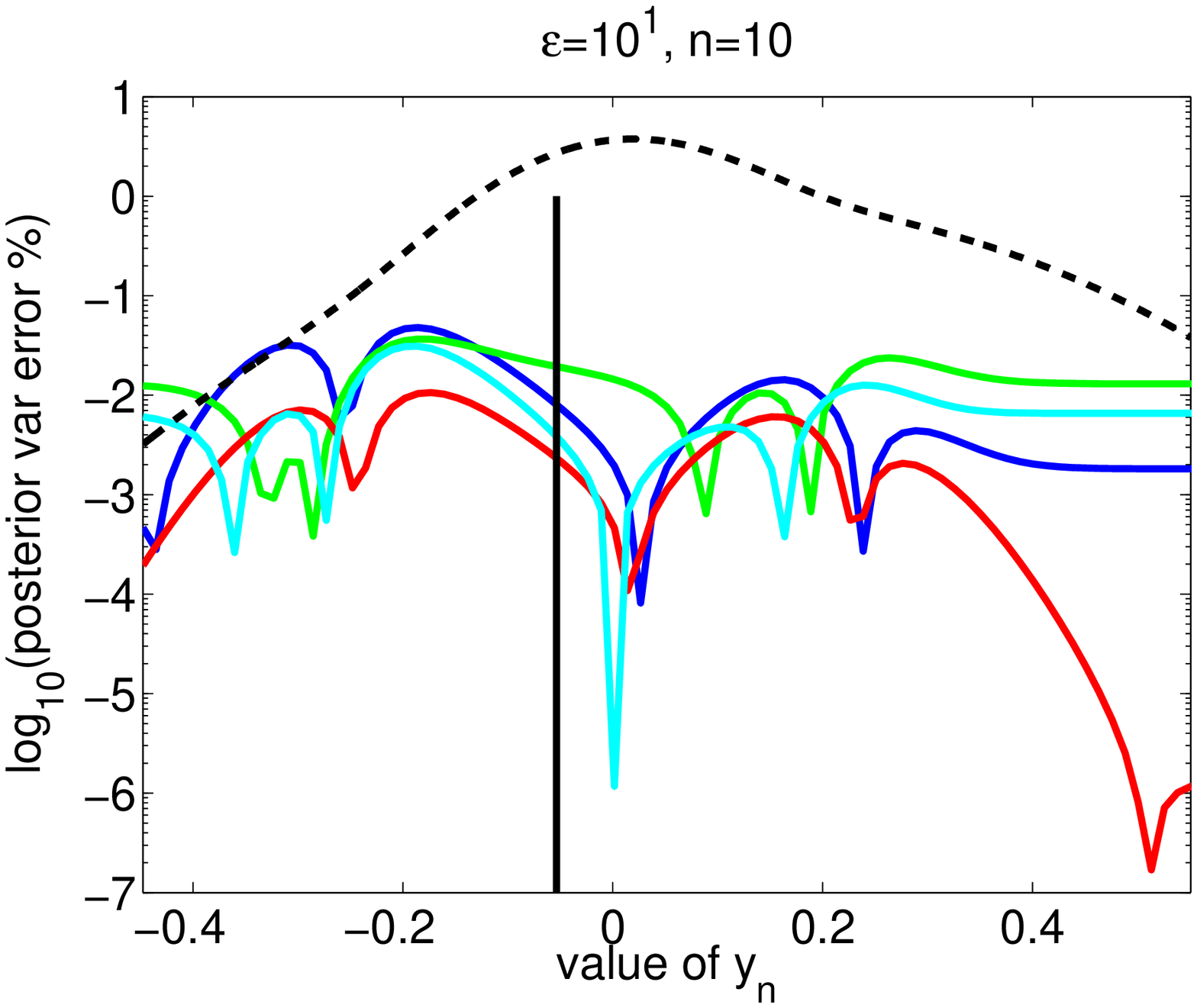} } 
\\
\vspace{-0.1in}
\subfigure
{\includegraphics[width=0.46\textwidth]{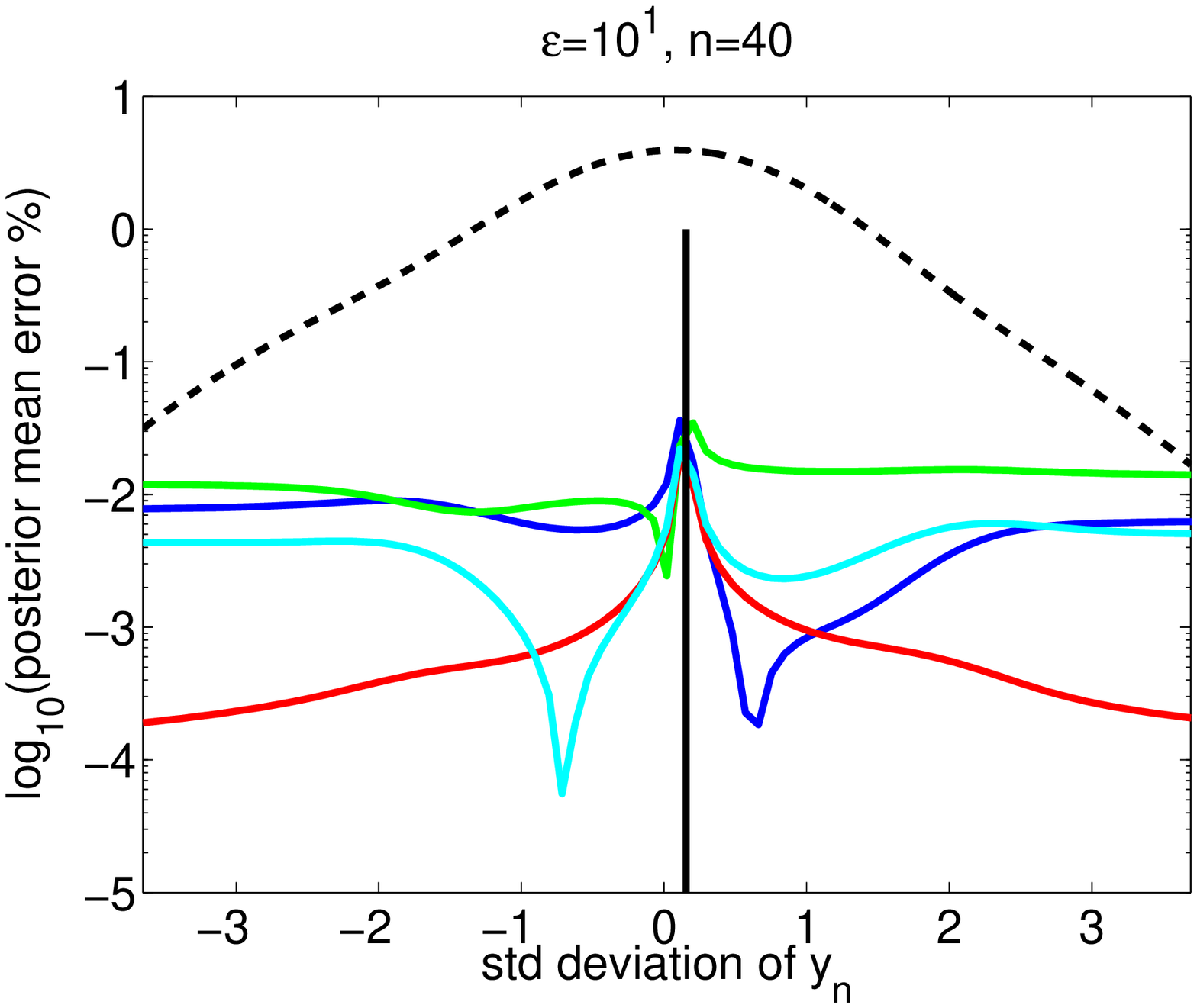} } 
\quad
\subfigure
{\includegraphics[width=0.46\textwidth]{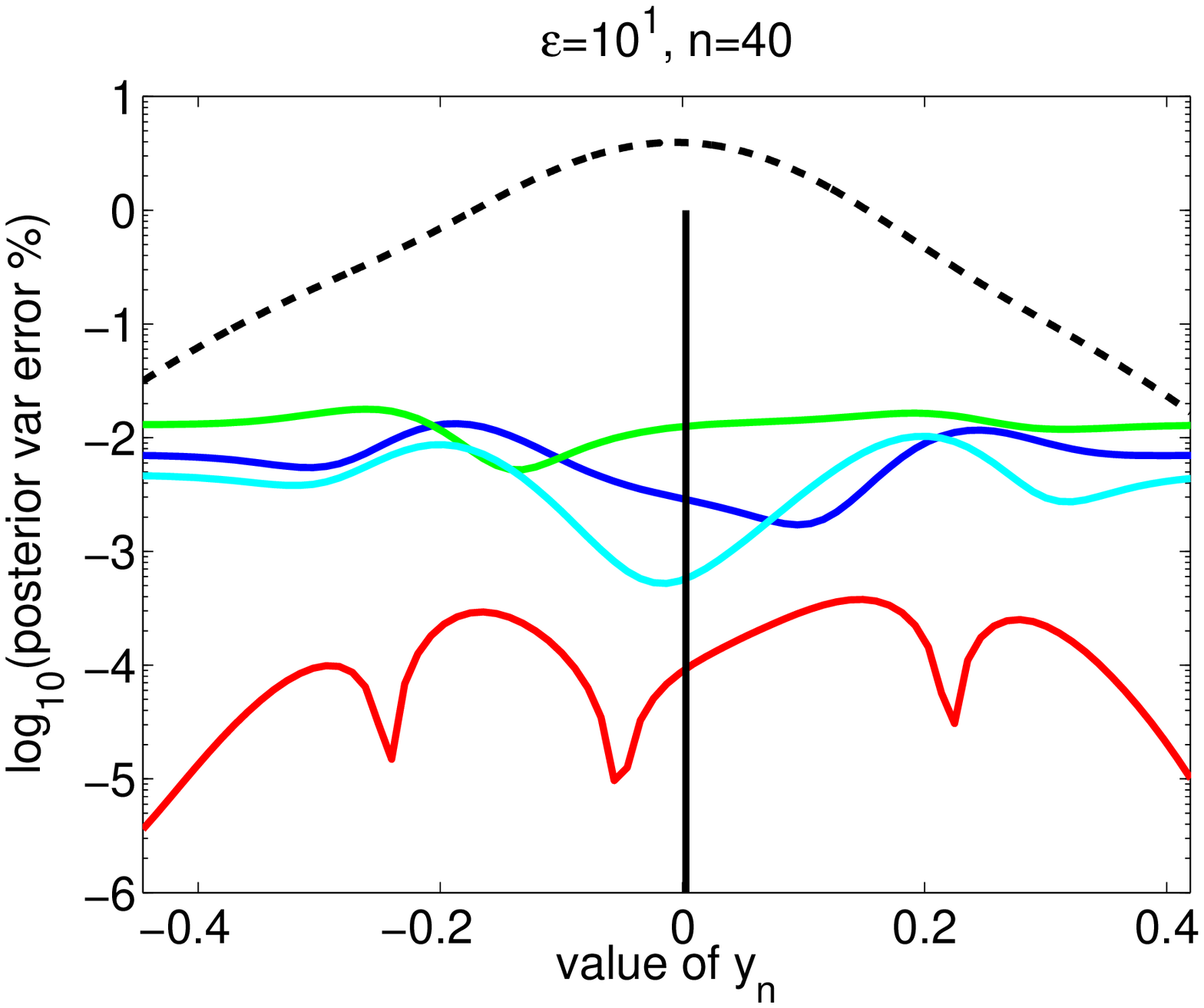} } 
\\
\vspace{-0.1in}
\subfigure
{\includegraphics[width=0.46\textwidth]{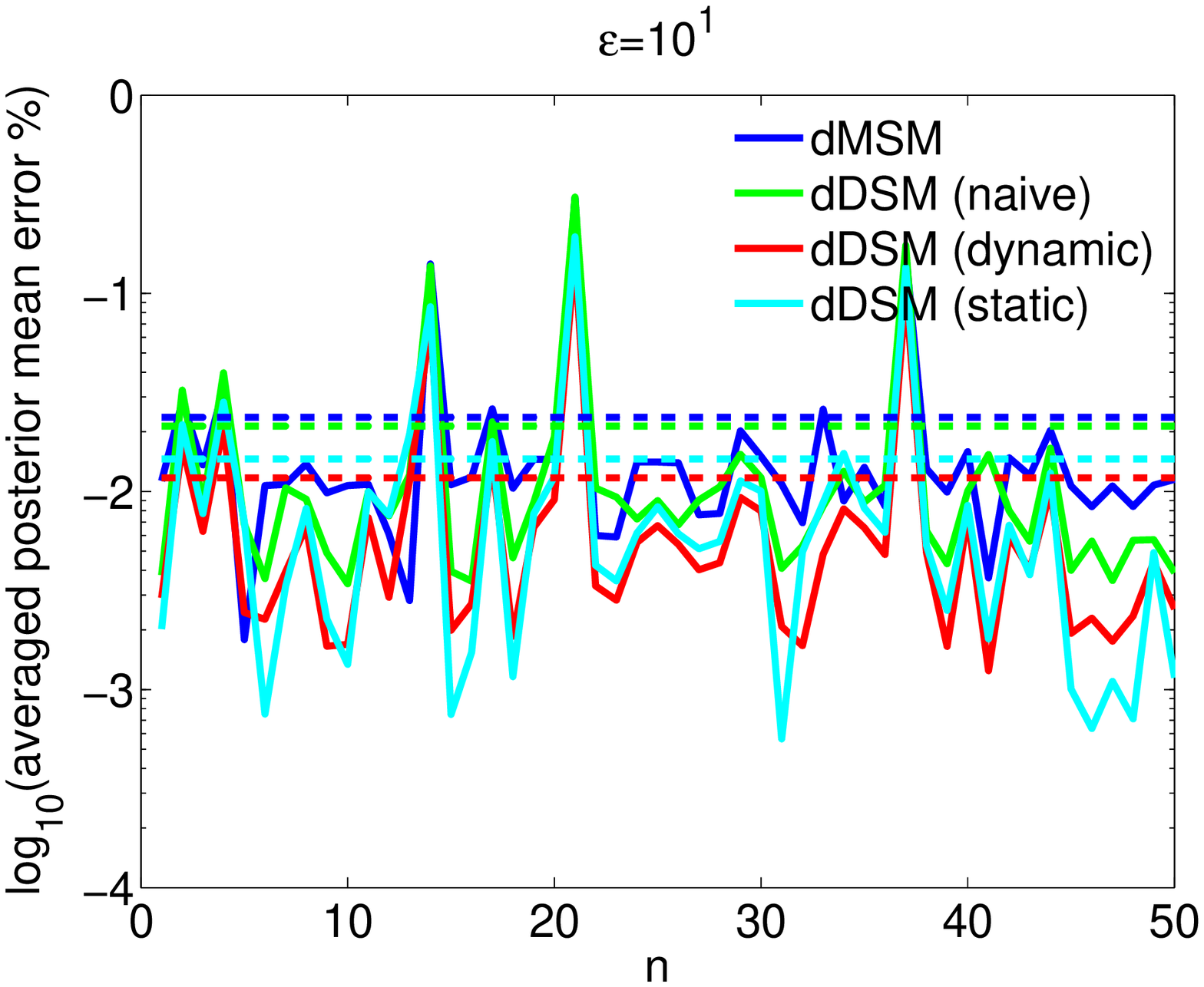} } 
\quad
\subfigure
{\includegraphics[width=0.46\textwidth]{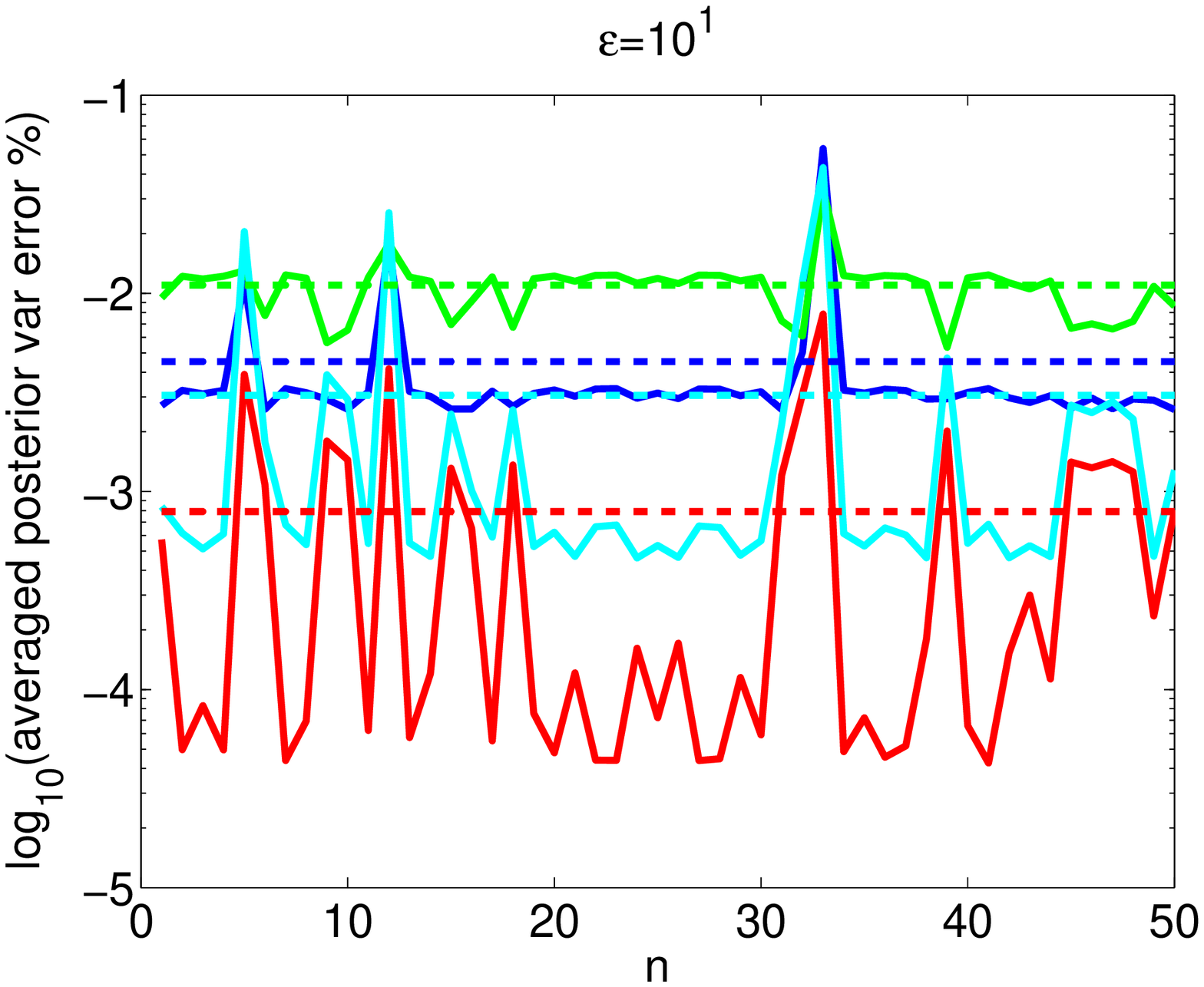} } 
\\
\vspace{-0.1in}
\caption{
The relative errors of the approximations of the posterior $u_n|Y_n$ distributions
that depend on 
the realization of $y_n= u_n + \eta_n$,
and their statistical averages with respect to the law of $y_n$
when $\epsilon = 10^{0}$.
In Gaussian sum filters, the accuracy of the posterior variance depends on 
the instance of $y_n$ (top-right and middle-right).
} 
\label{fig3e10} 
\end{figure}

\begin{figure}
  \centering
\subfigure
{\includegraphics[width=0.46\textwidth]{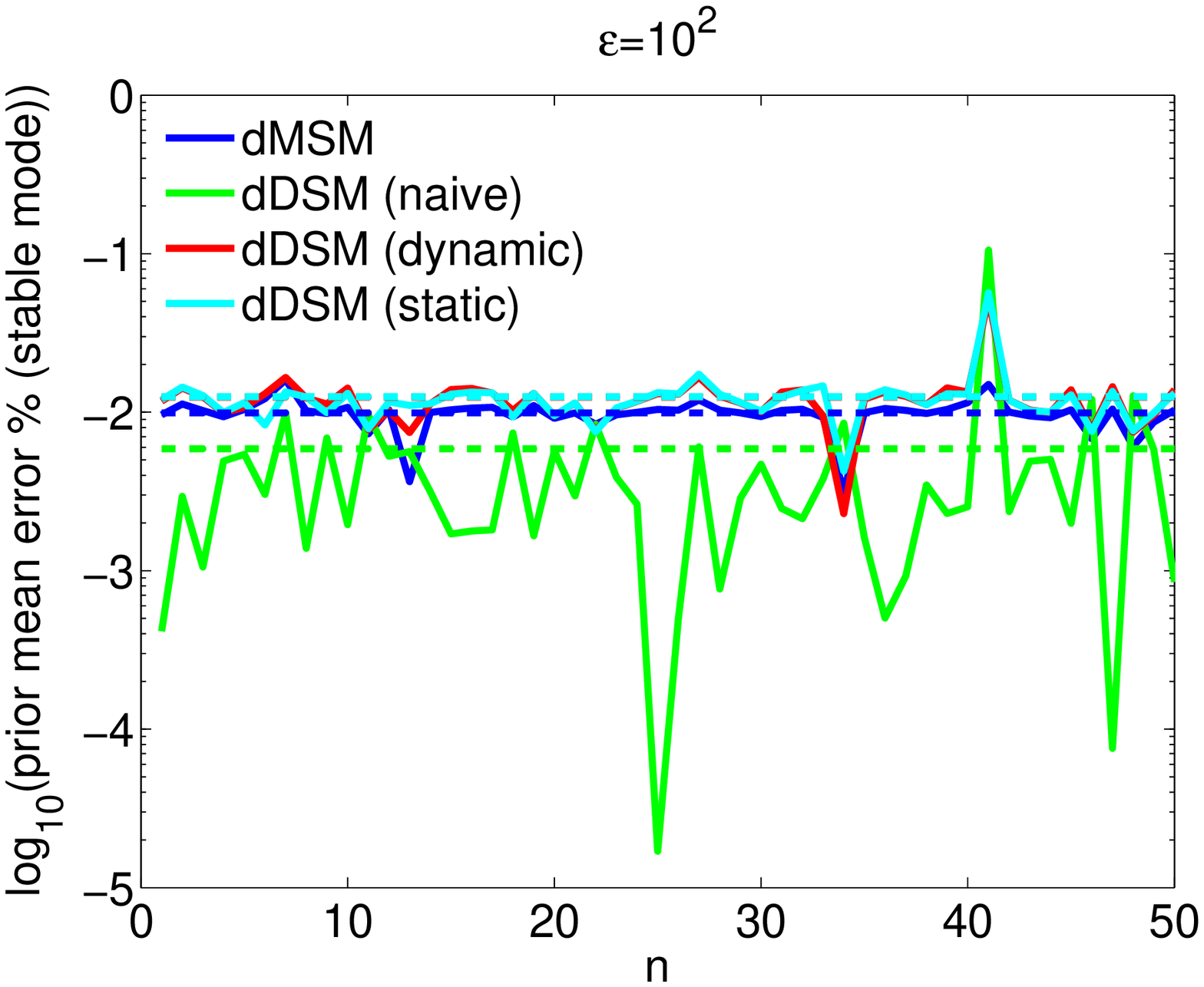} } 
\quad
\subfigure
{\includegraphics[width=0.46\textwidth]{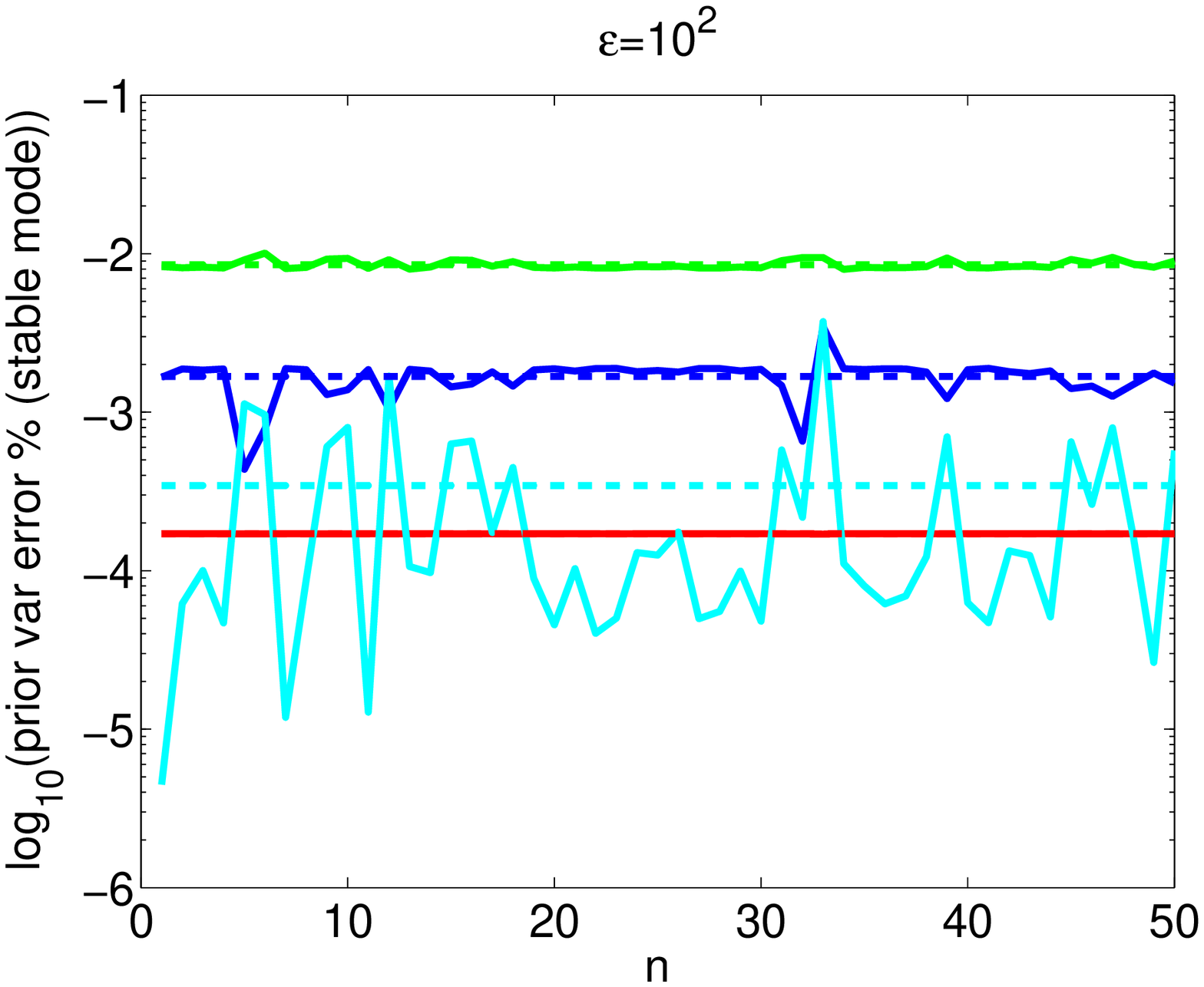} }
\\
\vspace{-0.1in}
\subfigure
{\includegraphics[width=0.46\textwidth]{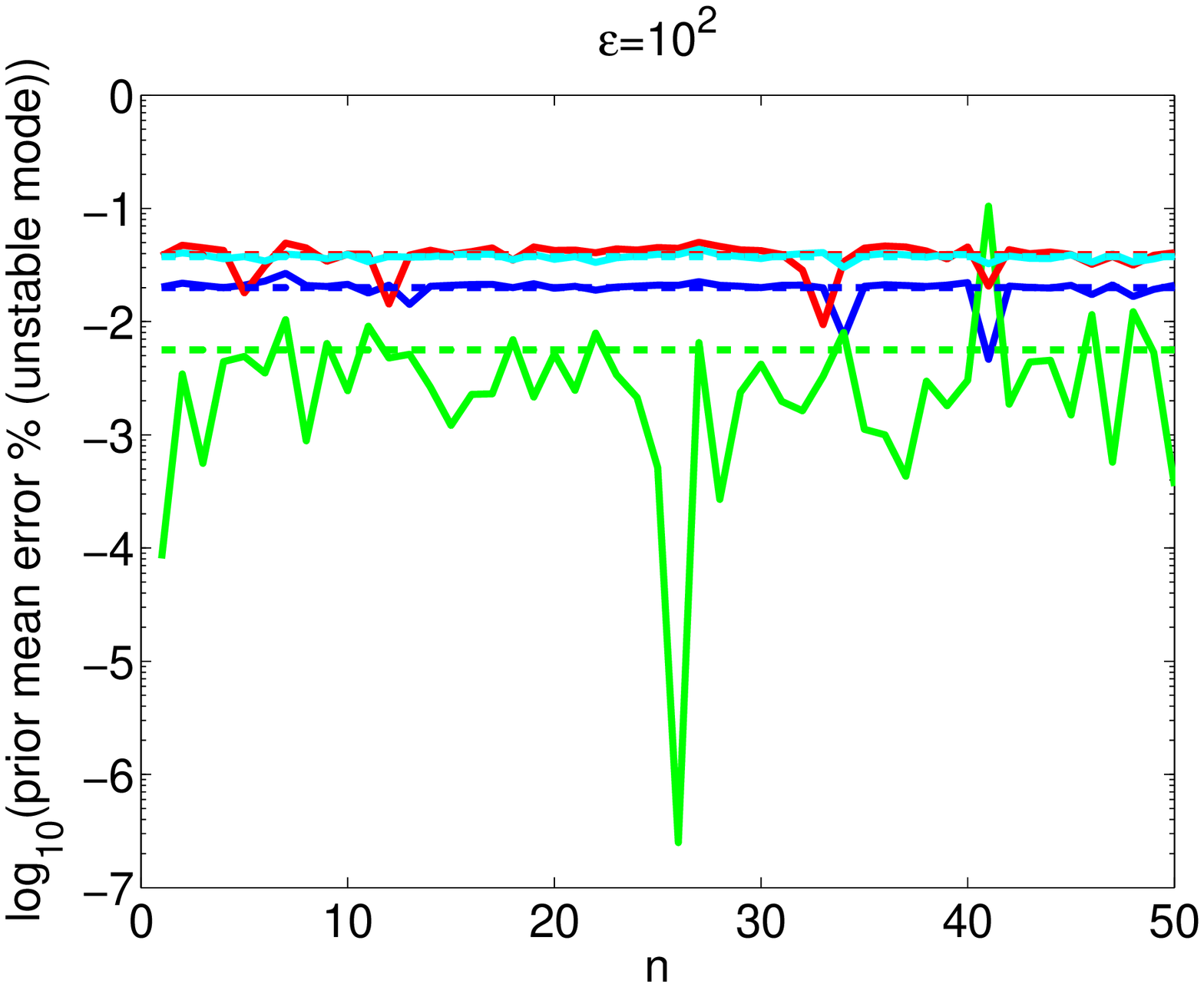}} 
\quad
\subfigure
{\includegraphics[width=0.46\textwidth]{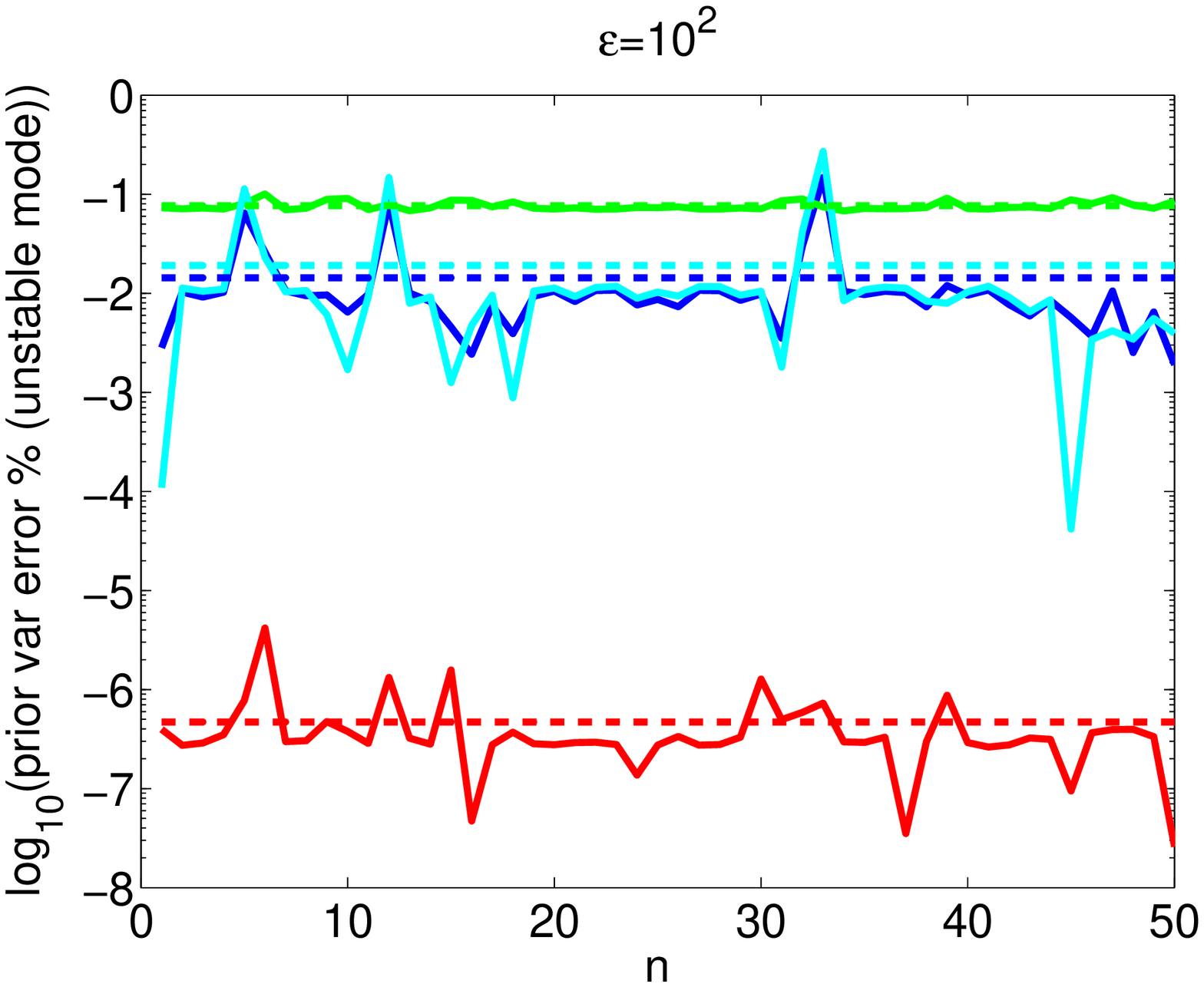}} 
\\
\vspace{-0.1in}
\caption{
The relative errors of the mean and variance approximations of
each Gaussian kernels of the prior distributions
when $\epsilon = 10^{2}$.
} 
\label{fig1e100} 
  \centering
\subfigure
{\includegraphics[width=0.46\textwidth]{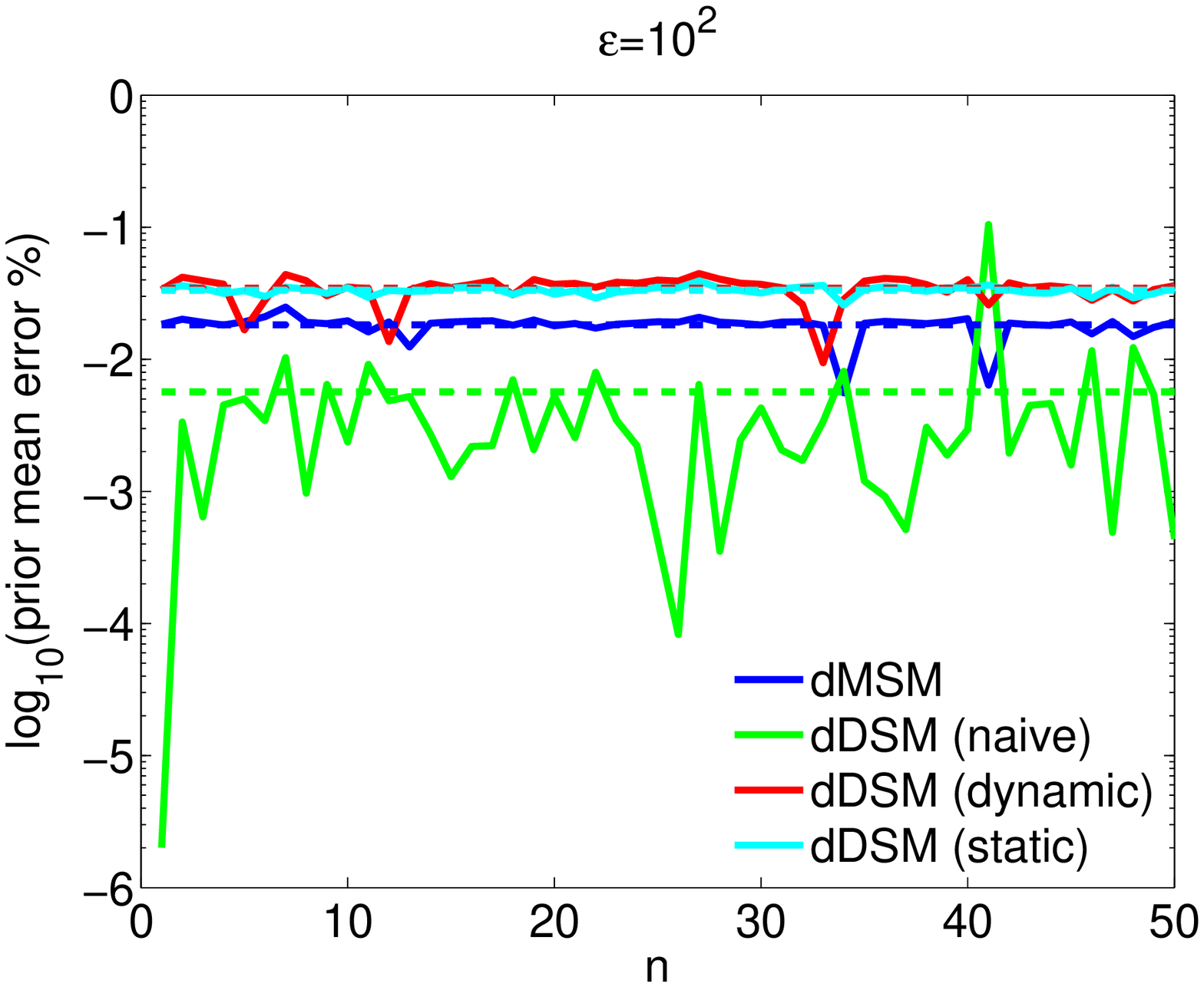} } 
\quad
\subfigure
{\includegraphics[width=0.46\textwidth]{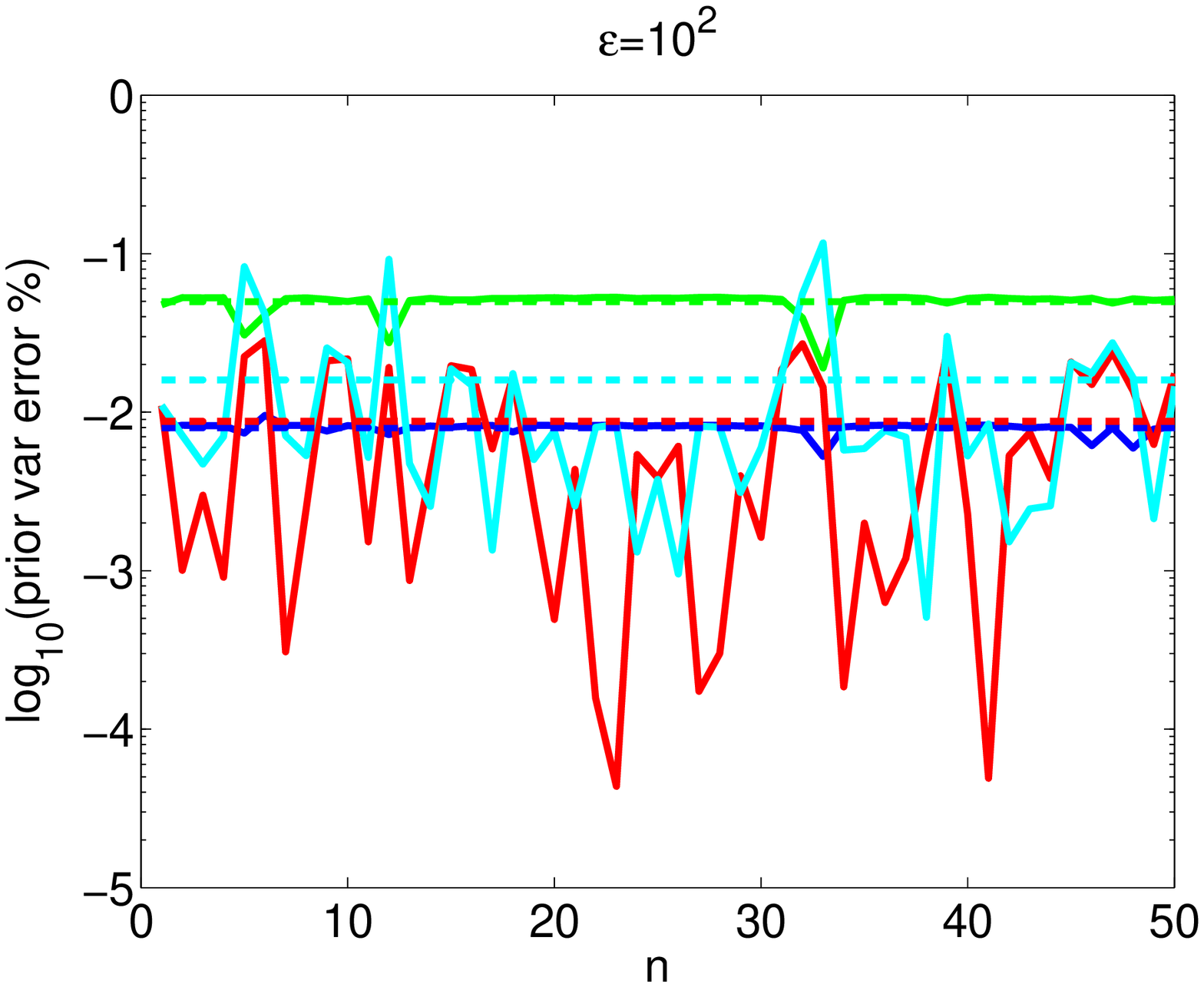} } 
\\
\vspace{-0.1in}
\subfigure
{\includegraphics[width=0.46\textwidth]{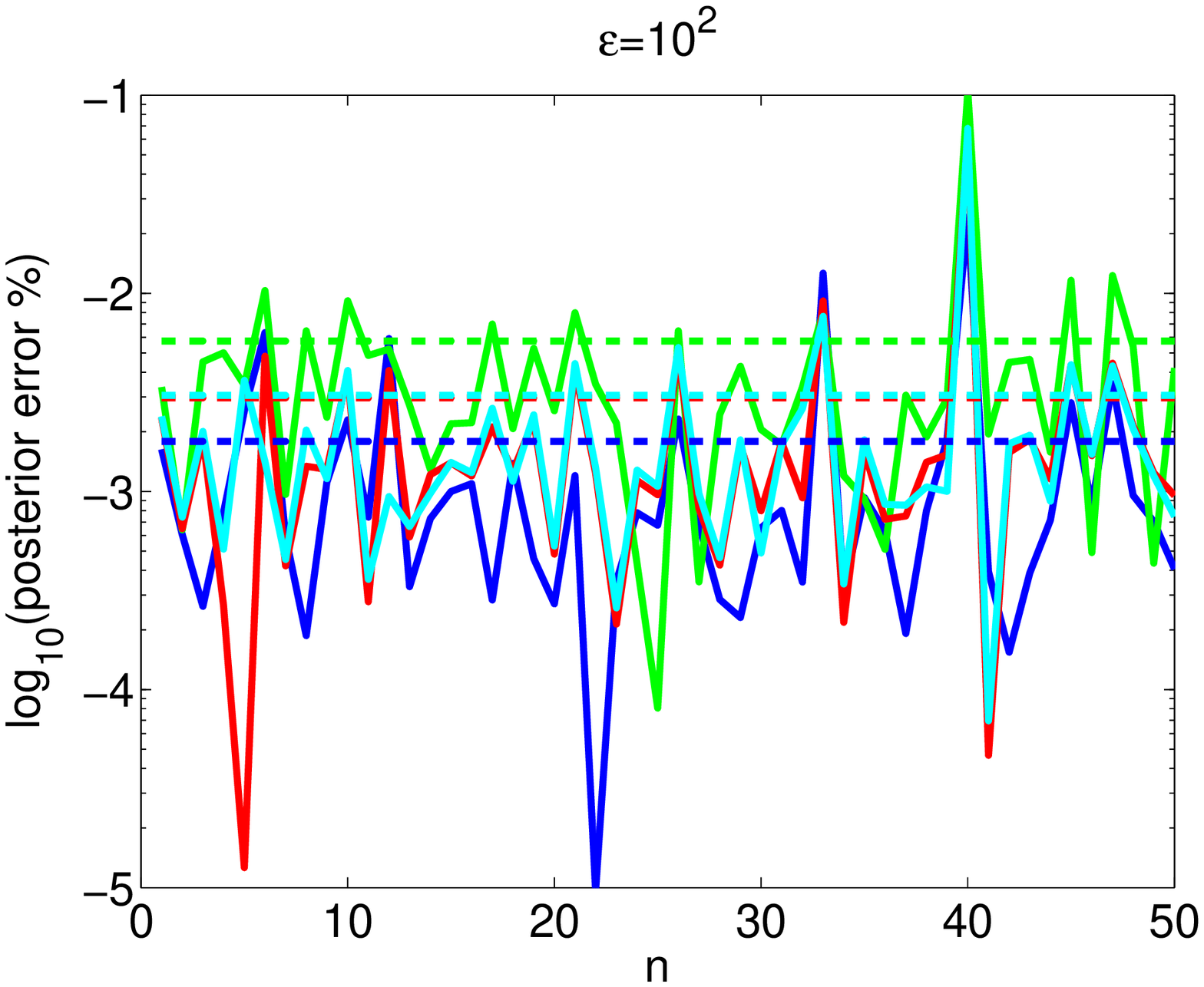} } 
\quad
\subfigure
{\includegraphics[width=0.46\textwidth]{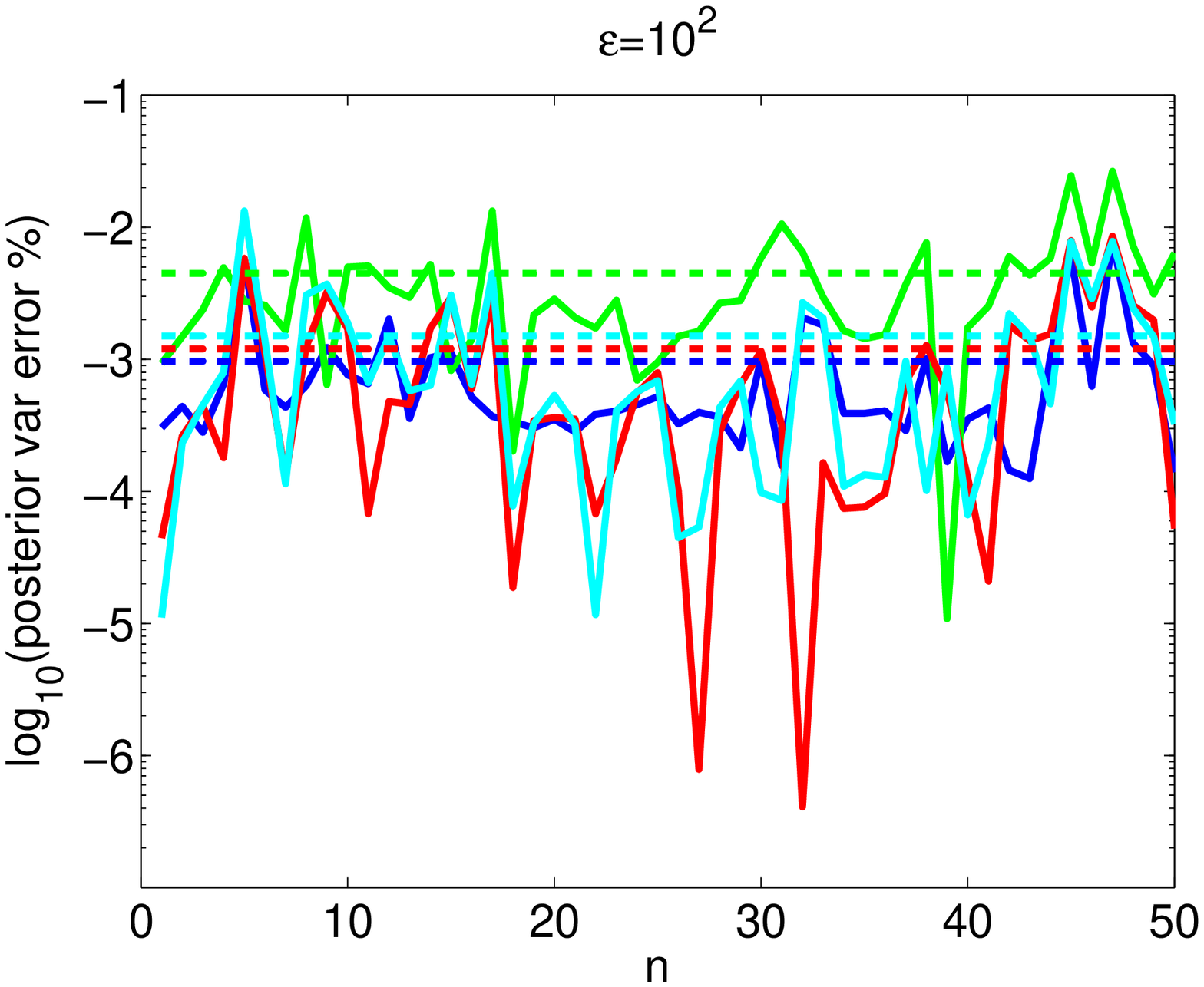} } 
\\
\vspace{-0.1in}
\caption{
The relative errors of the mean and variance approximations of
the prior $u_{n}|Y_{n-1}$ (top) and posterior $u_{n}|Y_{n}$ (bottom) distributions when $\epsilon = 10^{2}$.
} 
\label{fig2e100} 
\end{figure}

\begin{figure}
  \centering
\subfigure
{\includegraphics[width=0.46\textwidth]{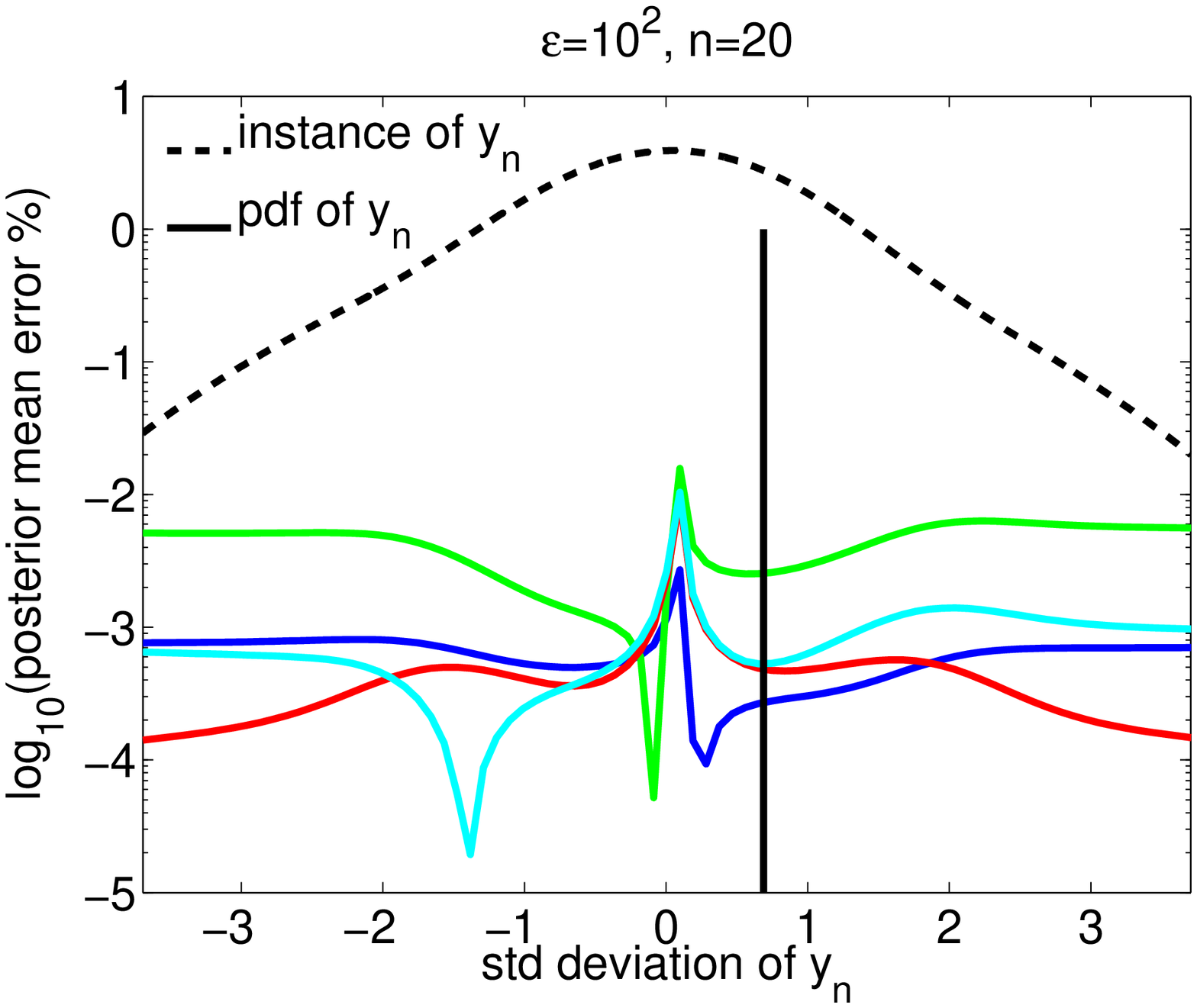} } 
\quad
\subfigure
{\includegraphics[width=0.46\textwidth]{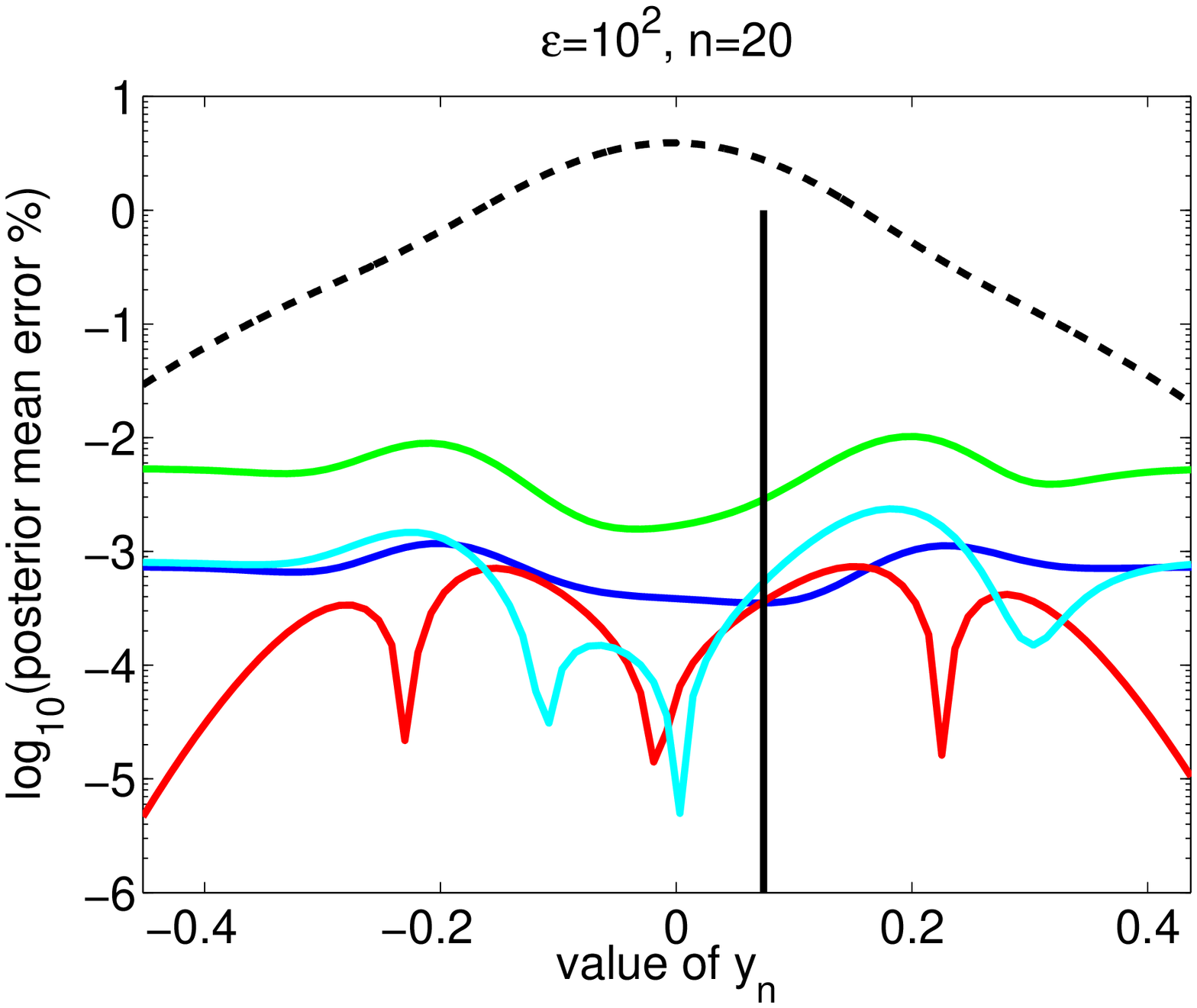} } 
\\
\vspace{-0.1in}
\subfigure
{\includegraphics[width=0.46\textwidth]{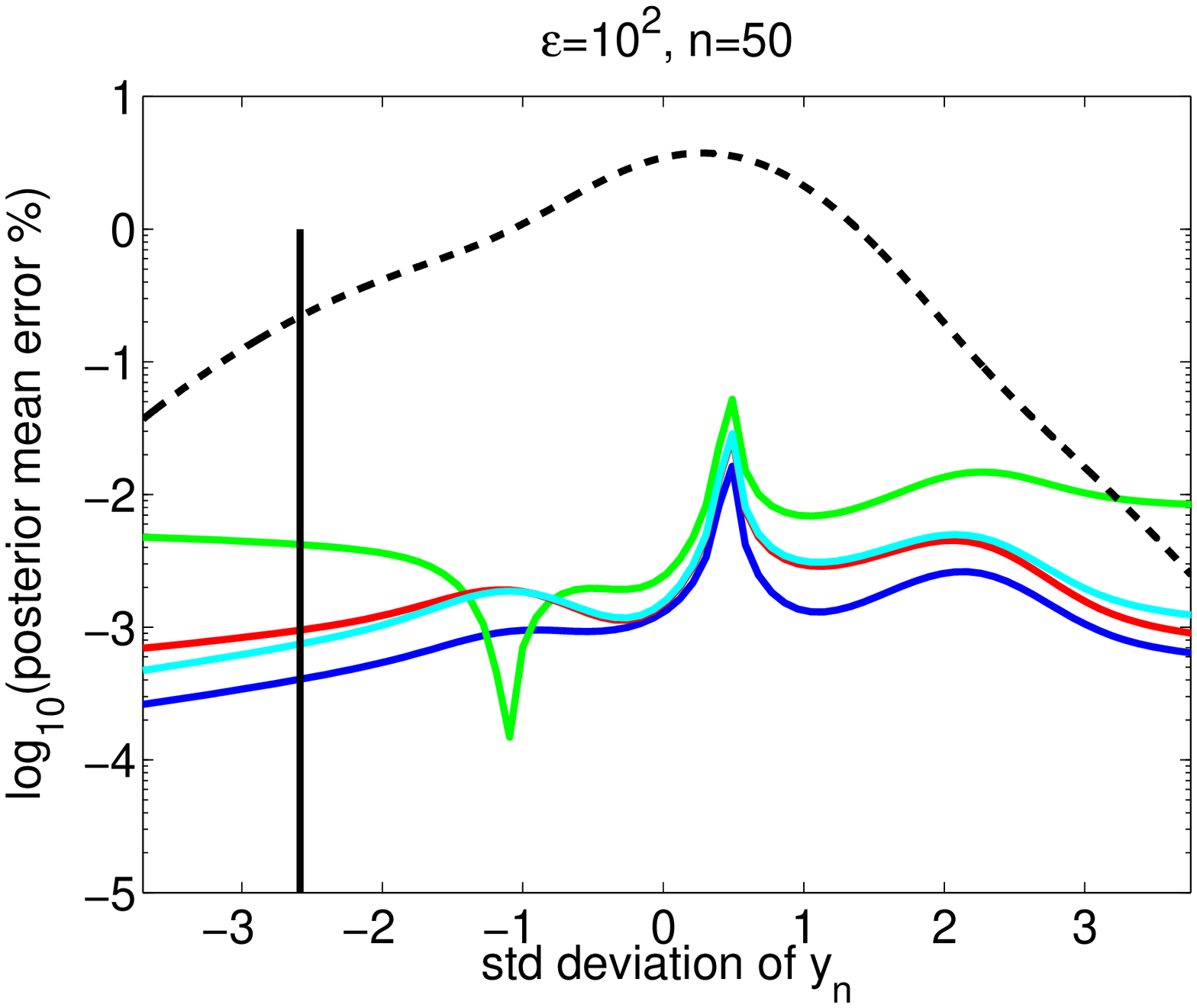} } 
\quad
\subfigure
{\includegraphics[width=0.46\textwidth]{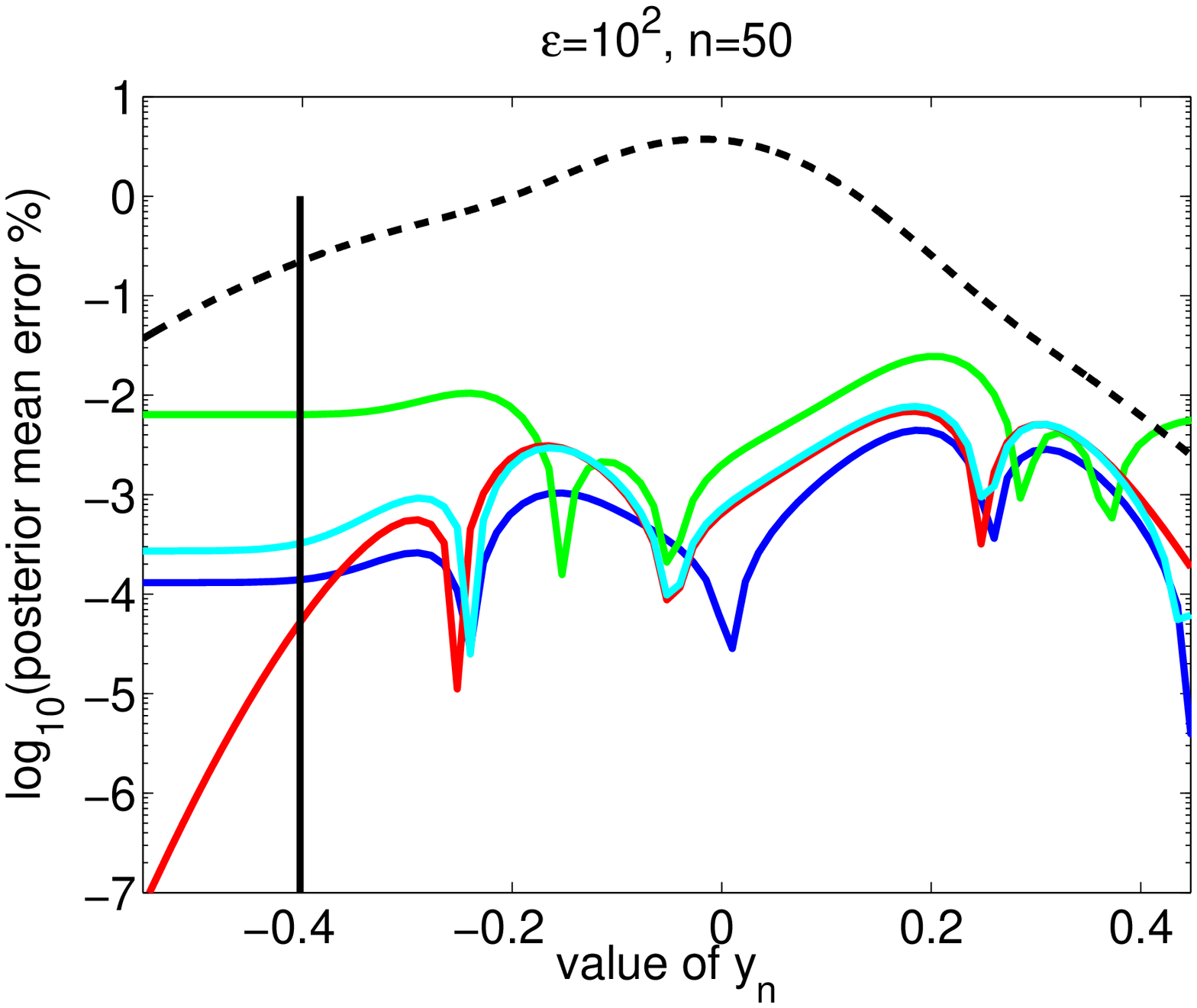} } 
\\
\vspace{-0.1in}
\subfigure
{\includegraphics[width=0.46\textwidth]{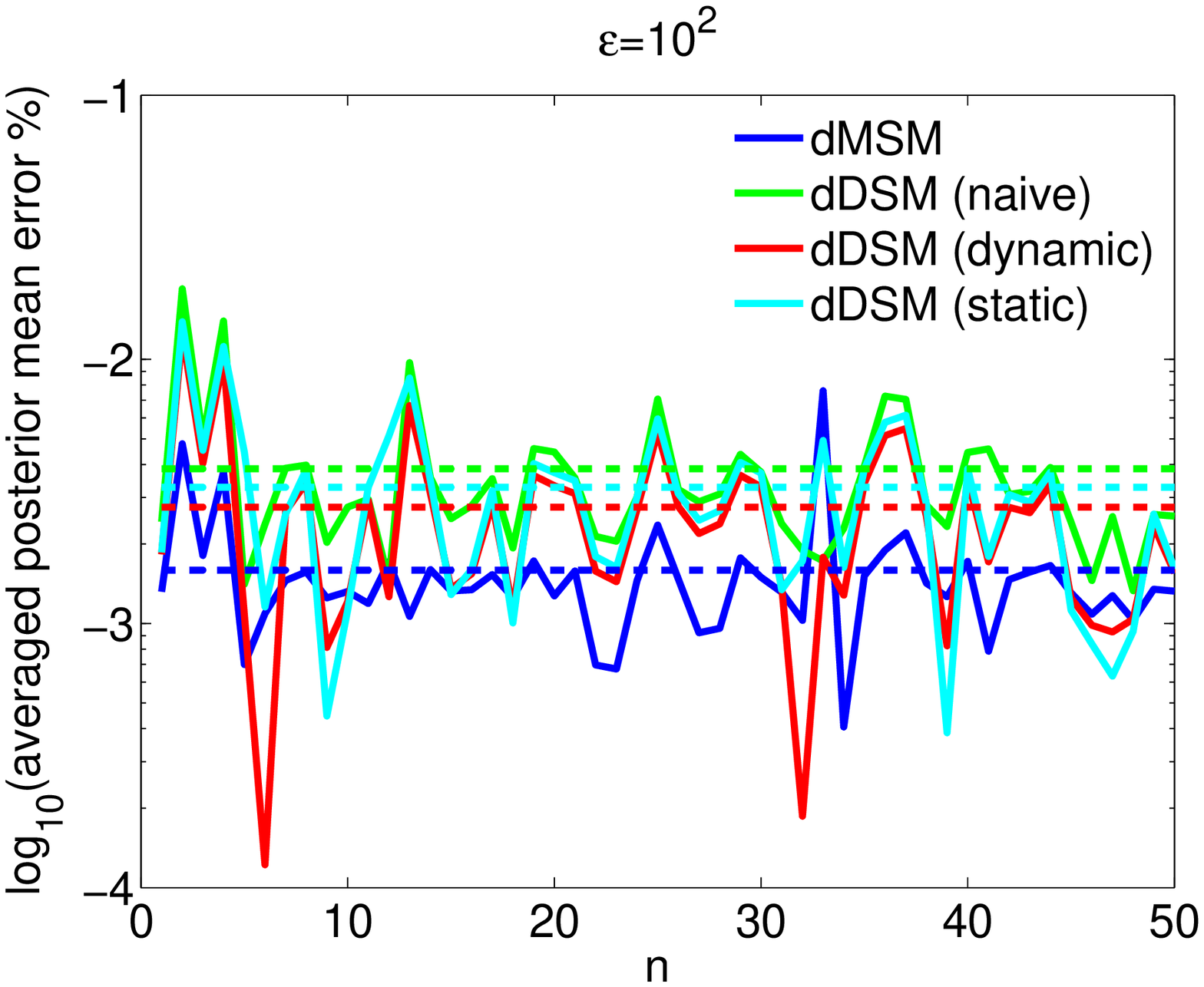} } 
\quad
\subfigure
{\includegraphics[width=0.46\textwidth]{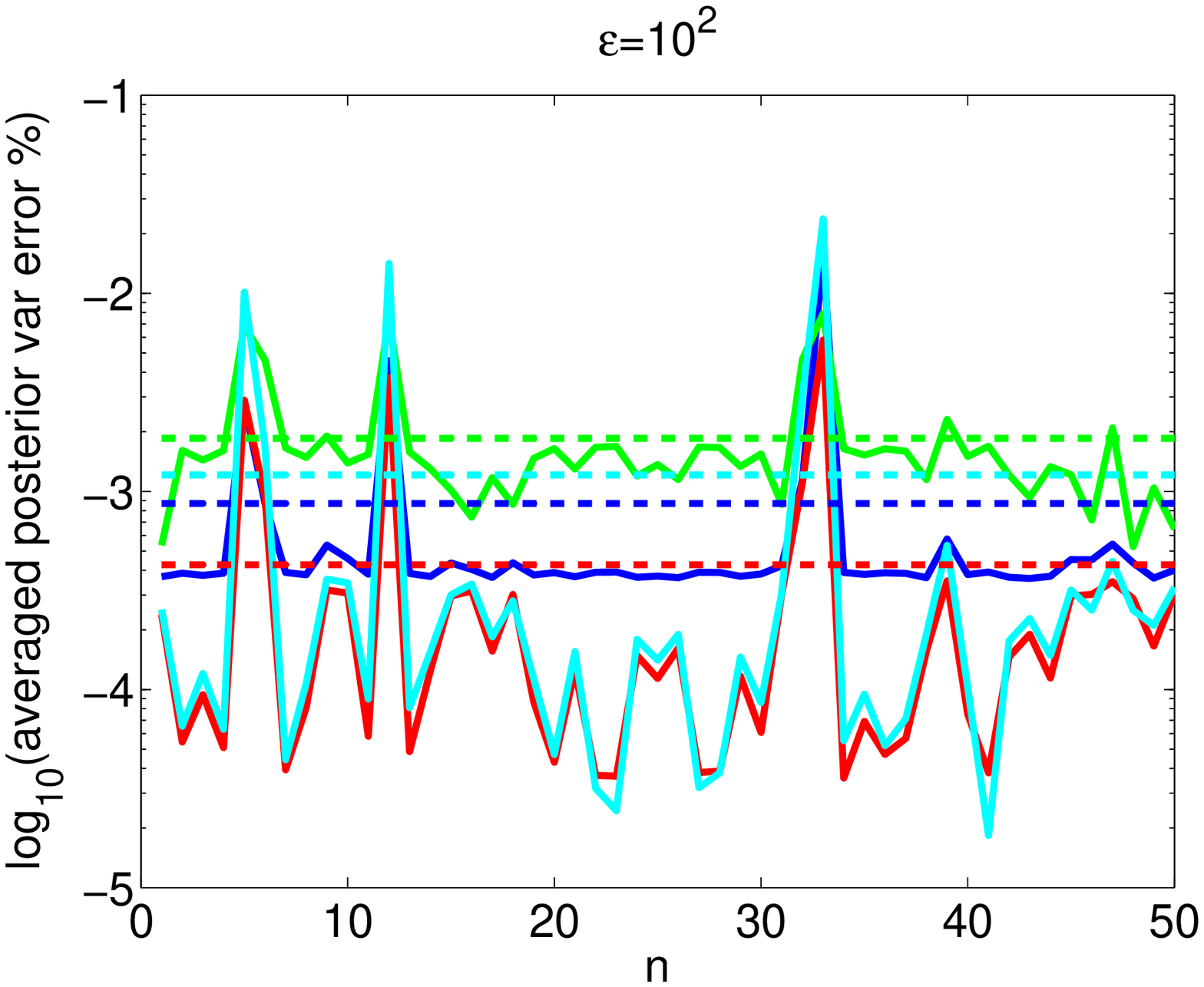} } 
\\
\vspace{-0.1in}
\caption{
The relative errors of the approximations of the posterior $u_n|Y_n$ distributions
that depend on 
the realization of $y_n= u_n + \eta_n$,
and their statistical averages with respect to the law of $y_n$
when $\epsilon = 10^{2}$.
} 
\label{fig3e100} 
\end{figure}

\begin{figure}
  \centering
\subfigure
{\includegraphics[width=0.46\textwidth]{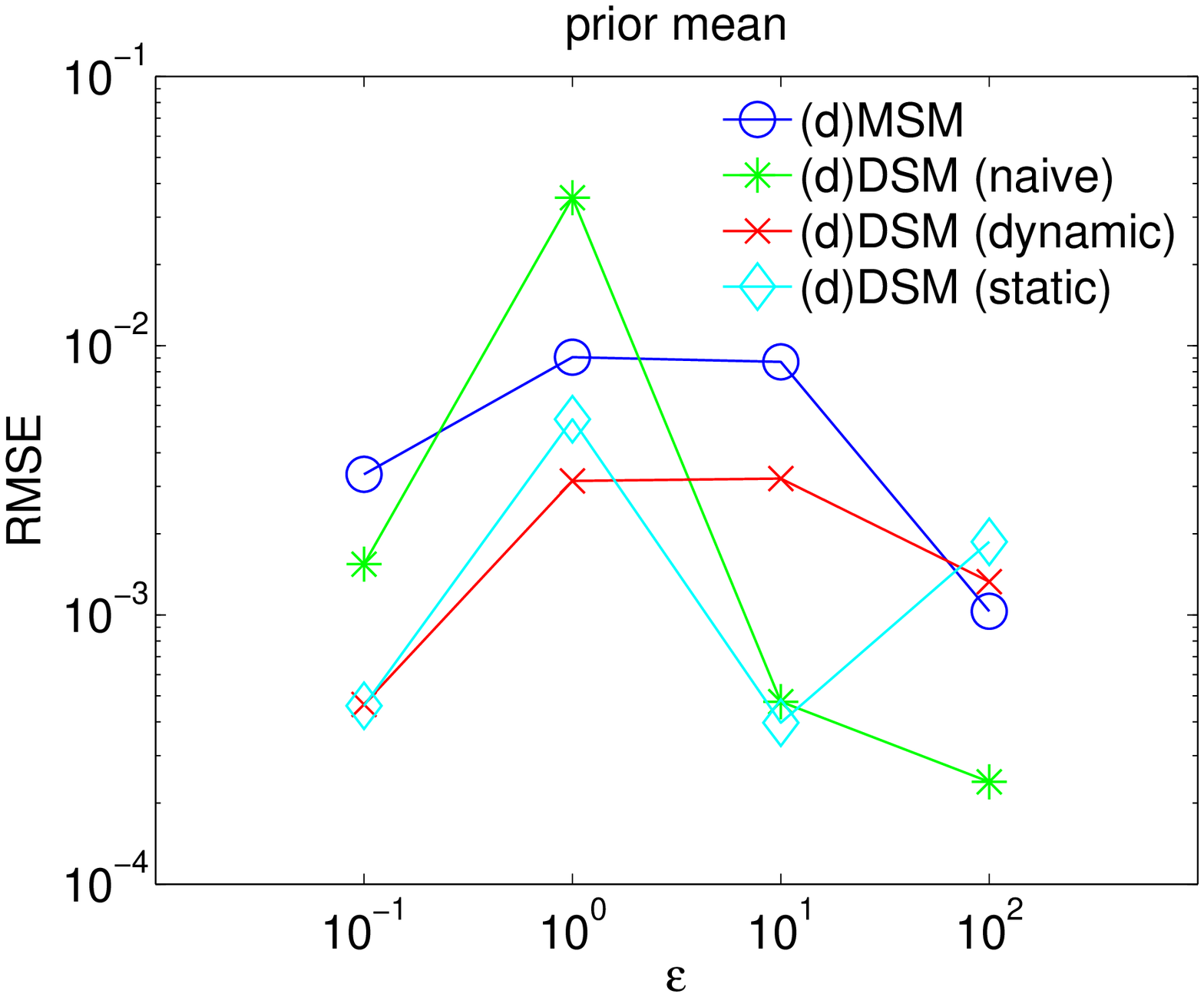} } 
\quad
\subfigure
{\includegraphics[width=0.46\textwidth]{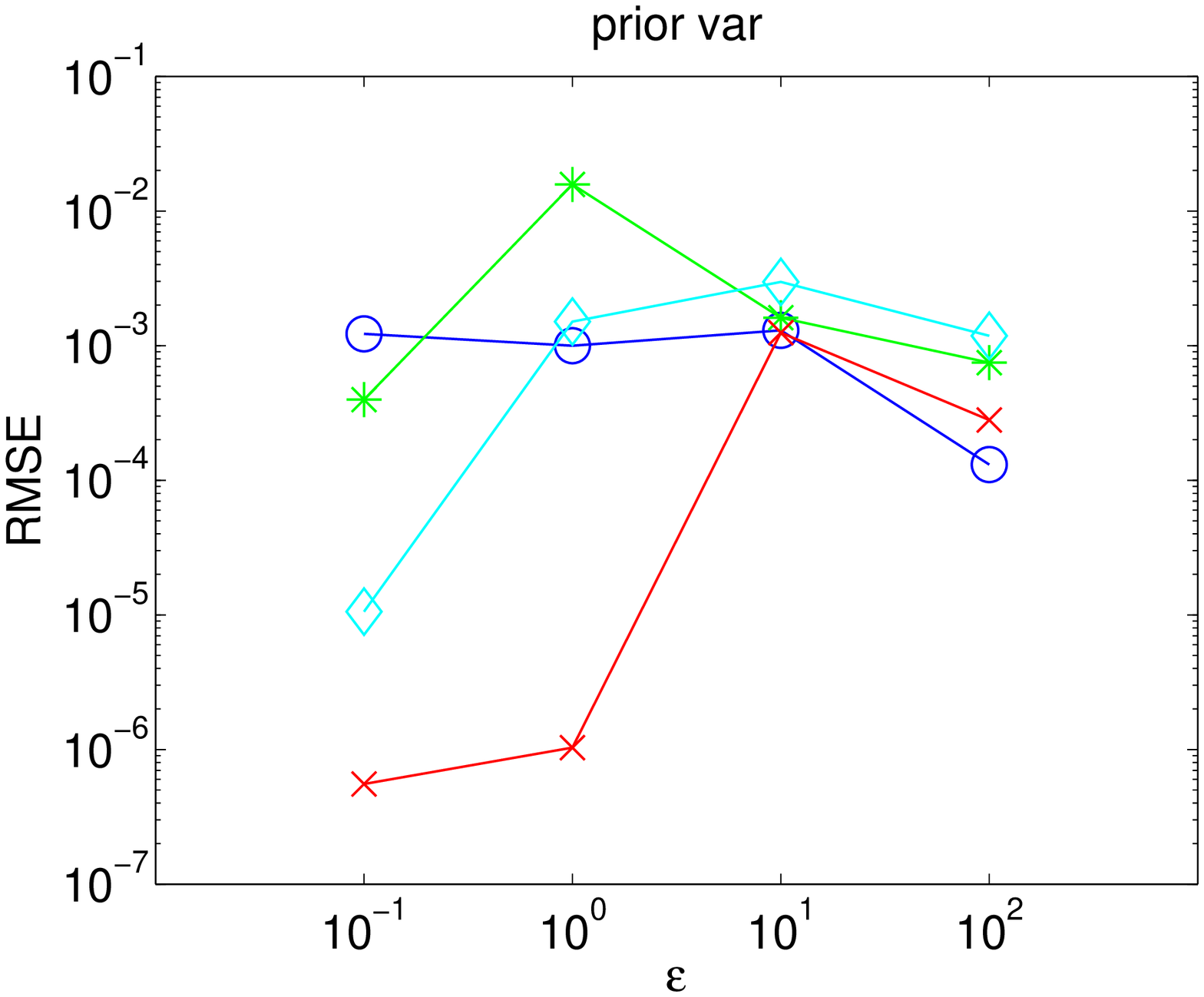} } 
\\
\vspace{-0.1in}
\subfigure
{\includegraphics[width=0.46\textwidth]{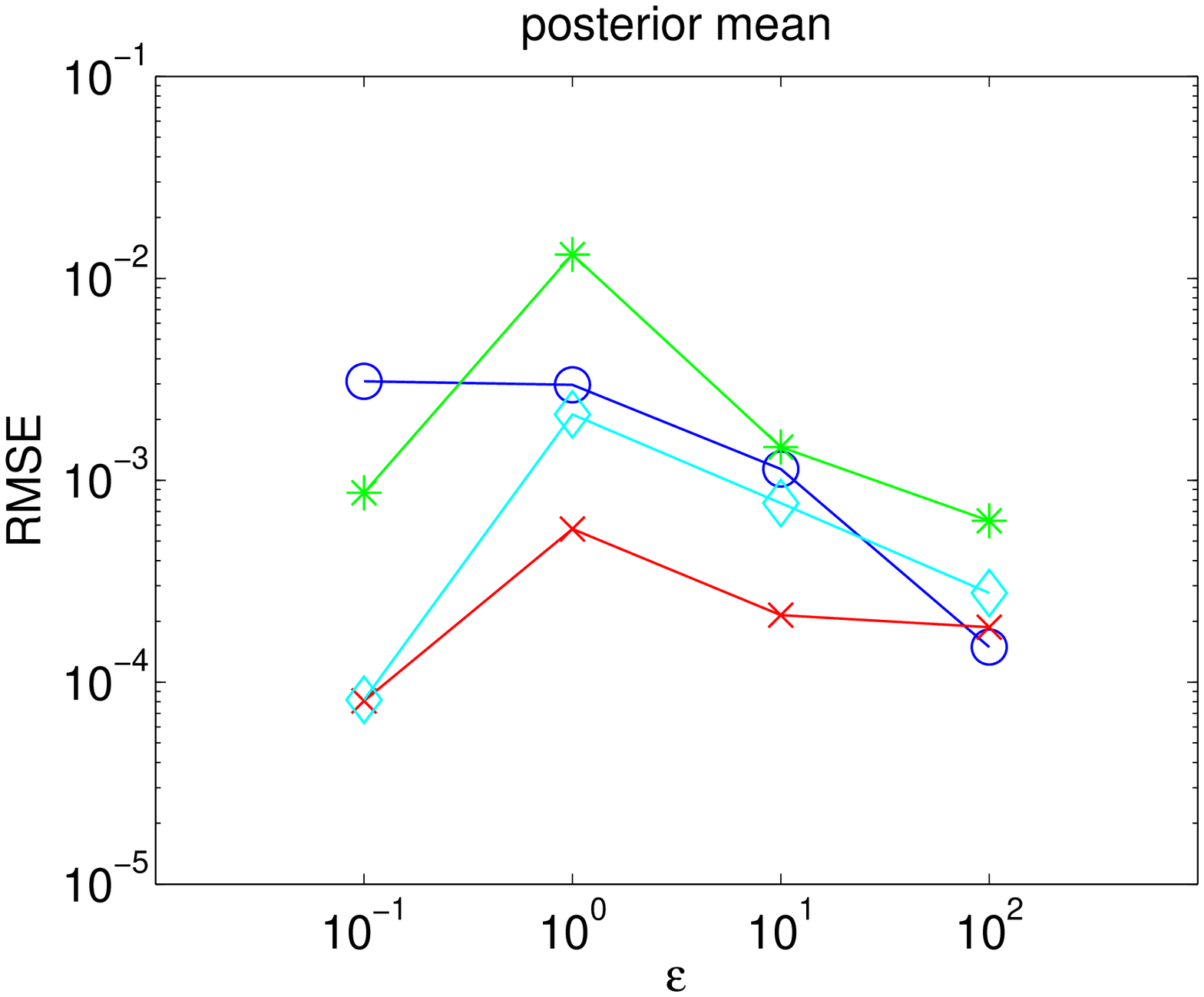} } 
\quad
\subfigure
{\includegraphics[width=0.46\textwidth]{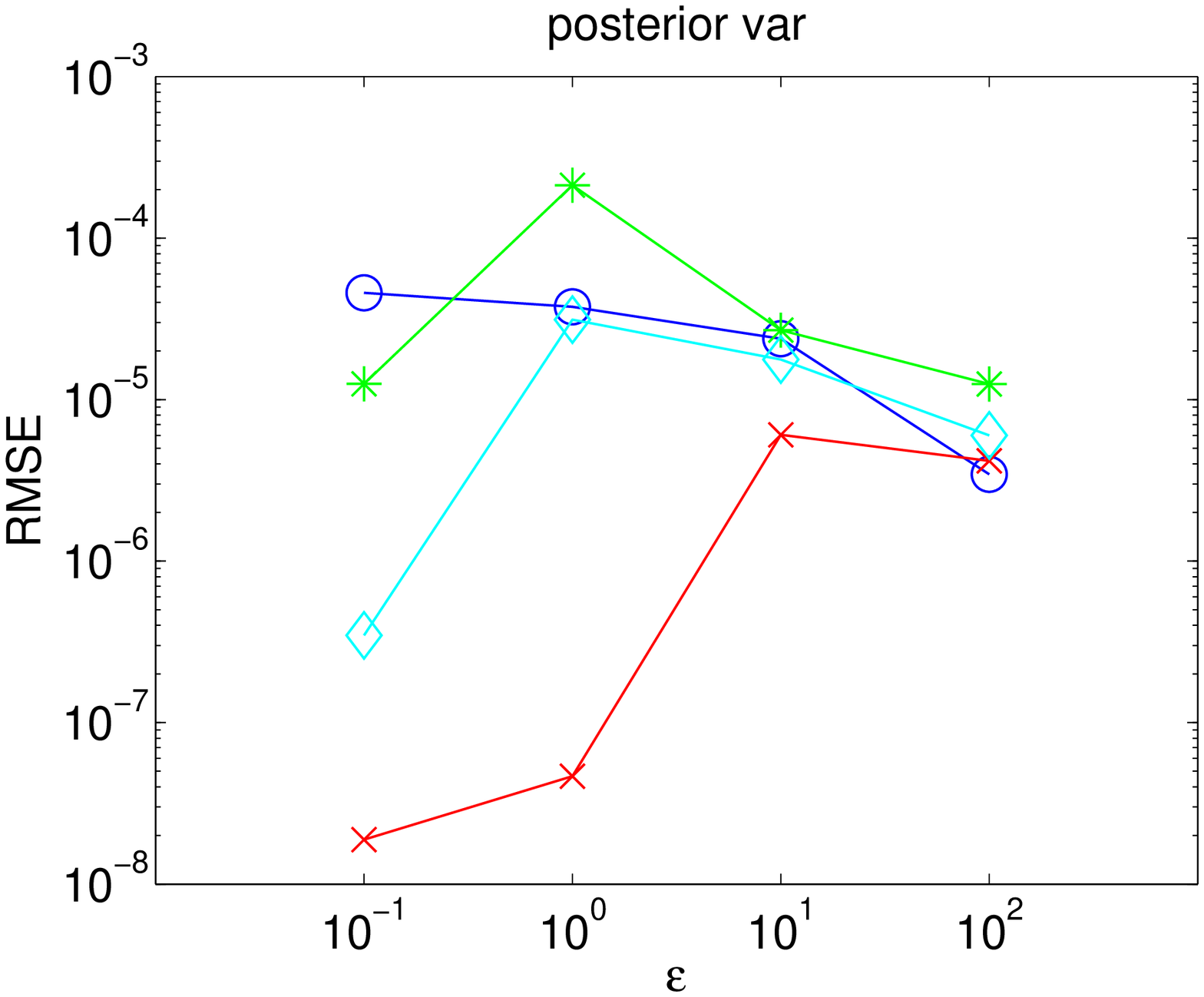} } 
\\
\vspace{-0.1in}
\caption{
The root mean square errors between the references from SSM filters and the approximations from MSM and DSM
filters.
} 
\label{fig:rmse} 
\end{figure}

\begin{figure}
\vspace{-0.2in}
\centerline{
\subfigure
{\includegraphics[width=0.46\textwidth]{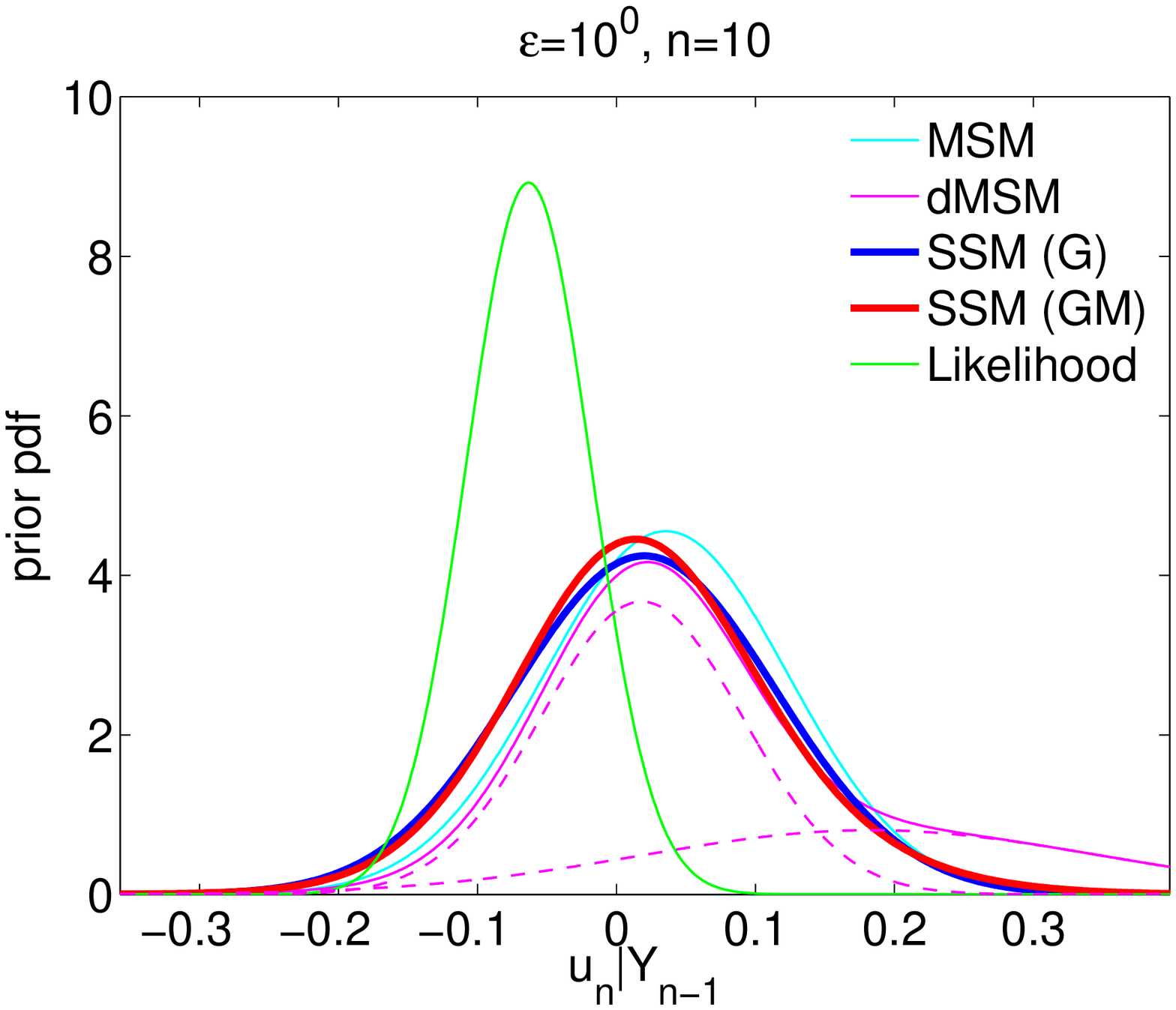} } 
\subfigure
{\includegraphics[width=0.46\textwidth]{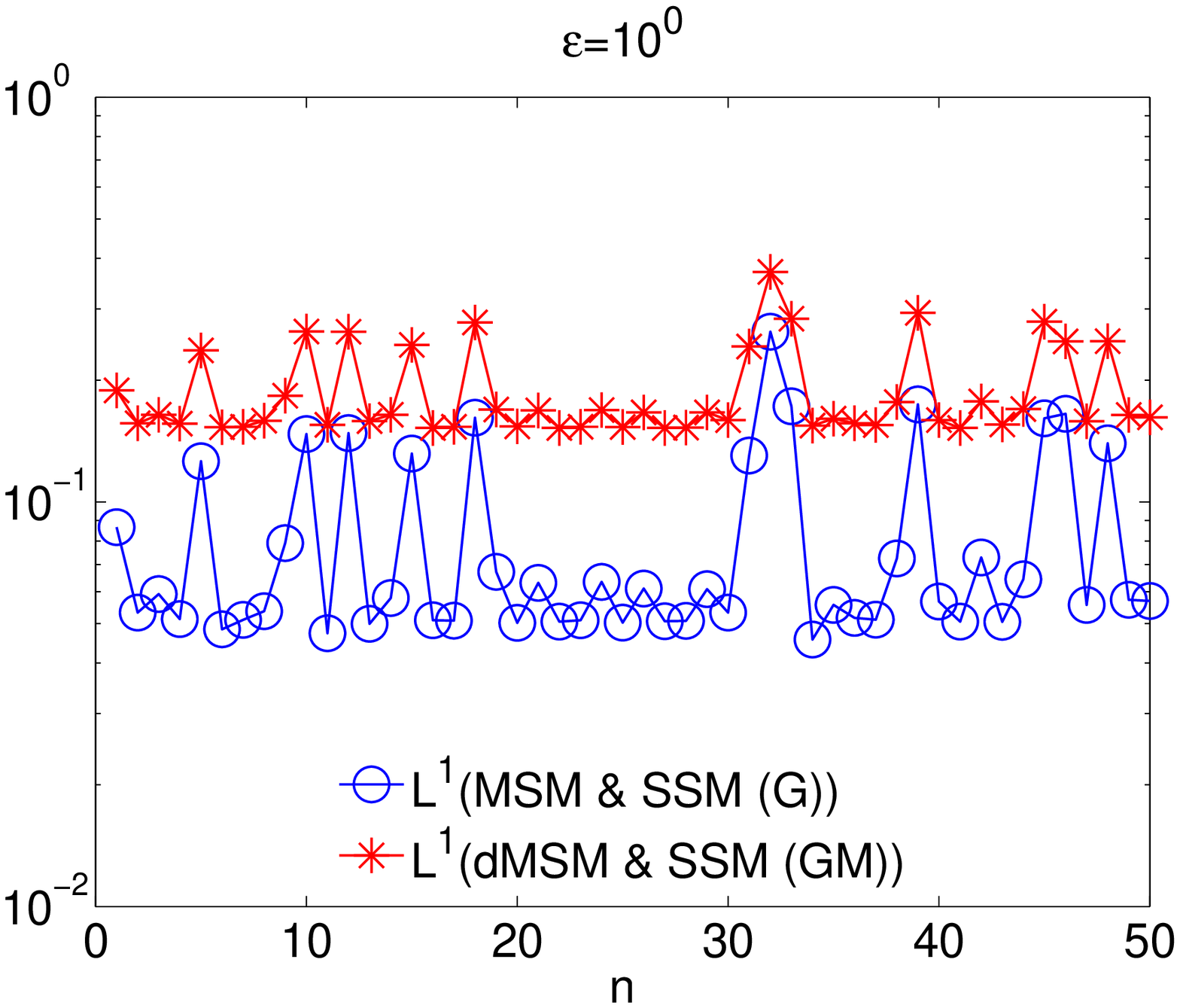} } 
}
\vspace{-0.2in}
\caption{
The Gaussian and Gaussian mixture approximations of SSM
(labeled SSM(G) and SSM(GM) respectively)
together with
MSM and dMSM when $\epsilon = 10^{0}$.
The dashed lines represent two Gaussian kernels consisting of dMSM (left).
} 
\label{fig:gsae1} 

\vspace{-0.1in}
\centerline{
\subfigure
{\includegraphics[width=0.46\textwidth]{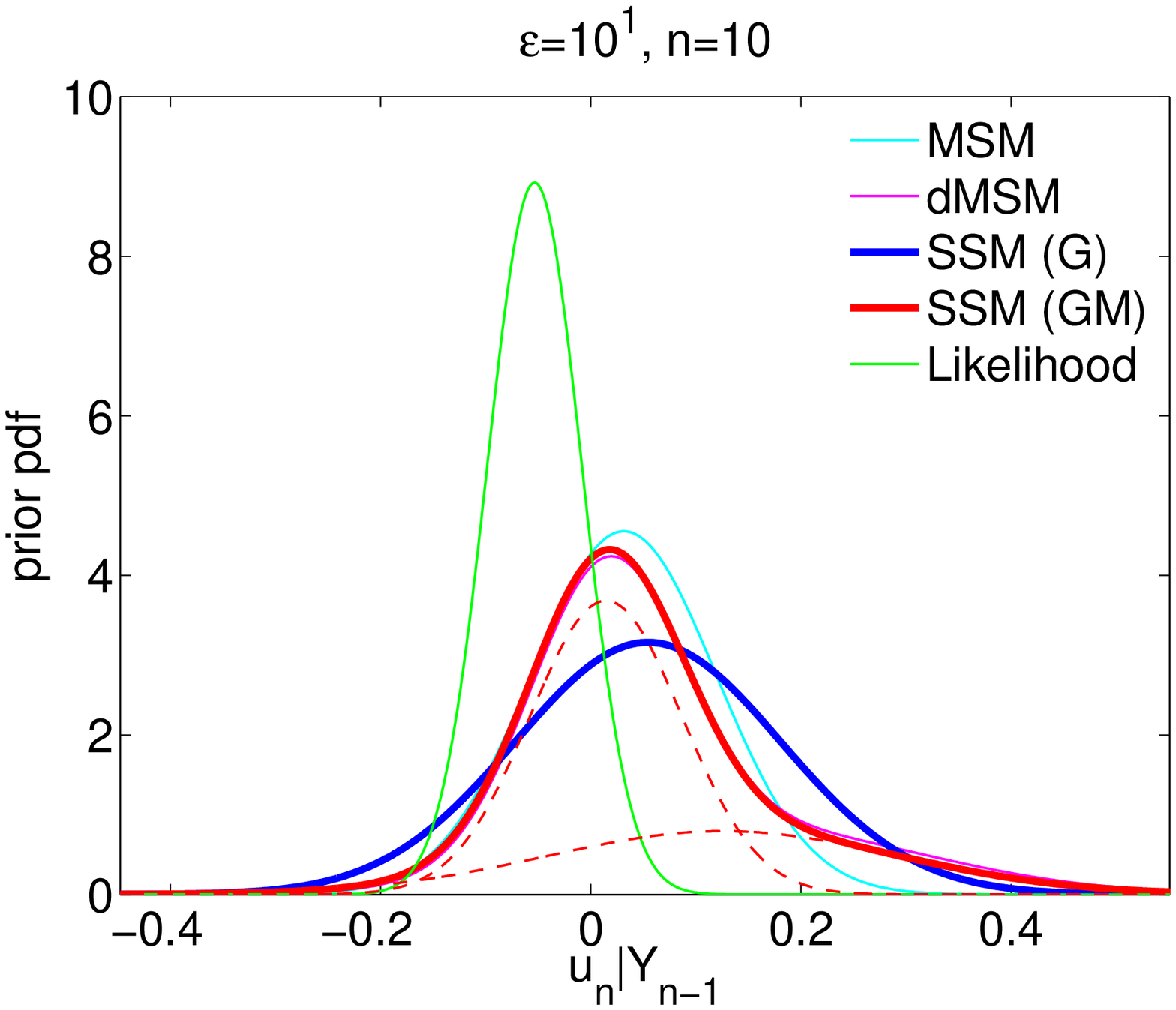} } 
\subfigure
{\includegraphics[width=0.46\textwidth]{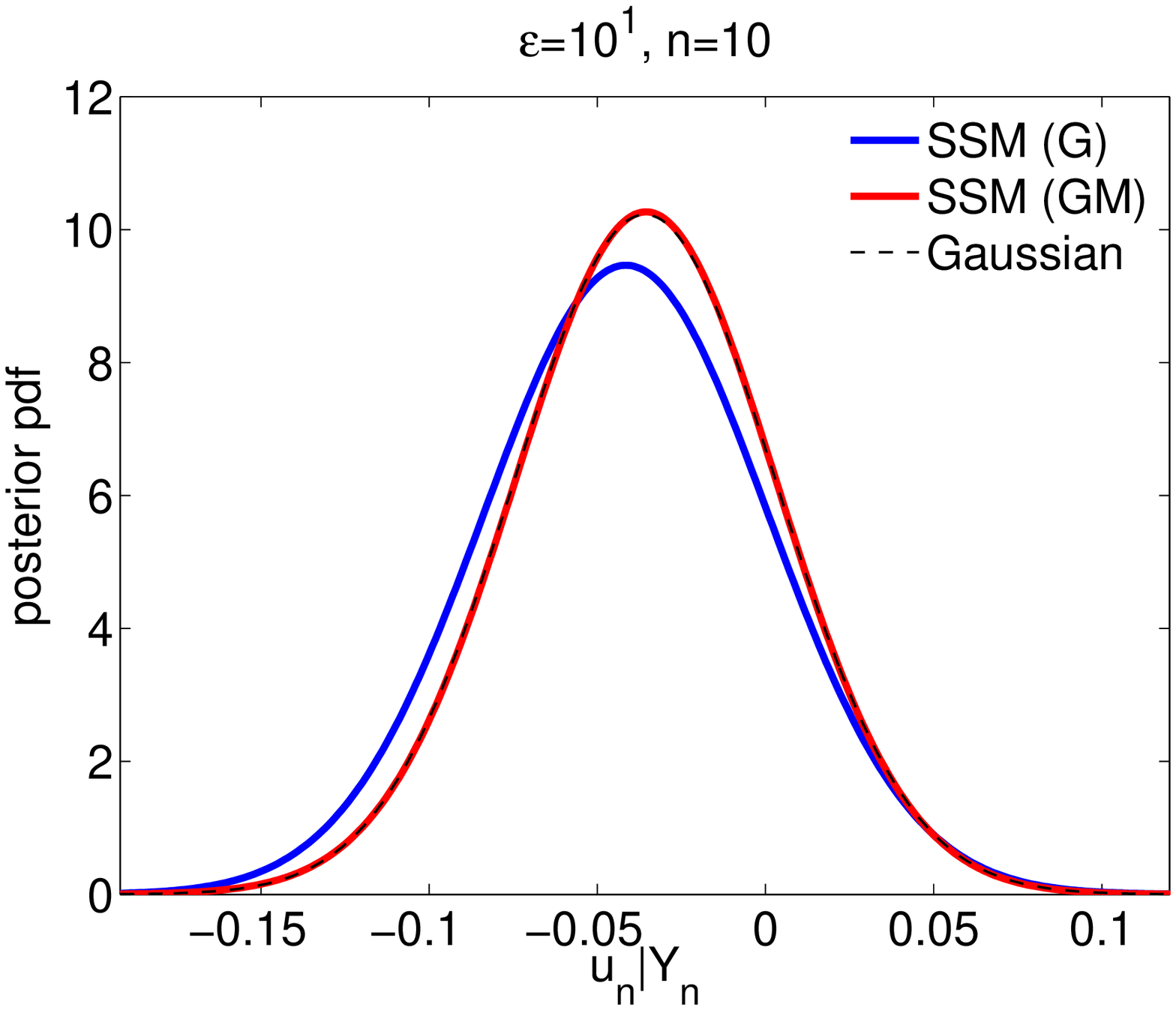} } 
}
\centerline{
\subfigure
{\includegraphics[width=0.46\textwidth]{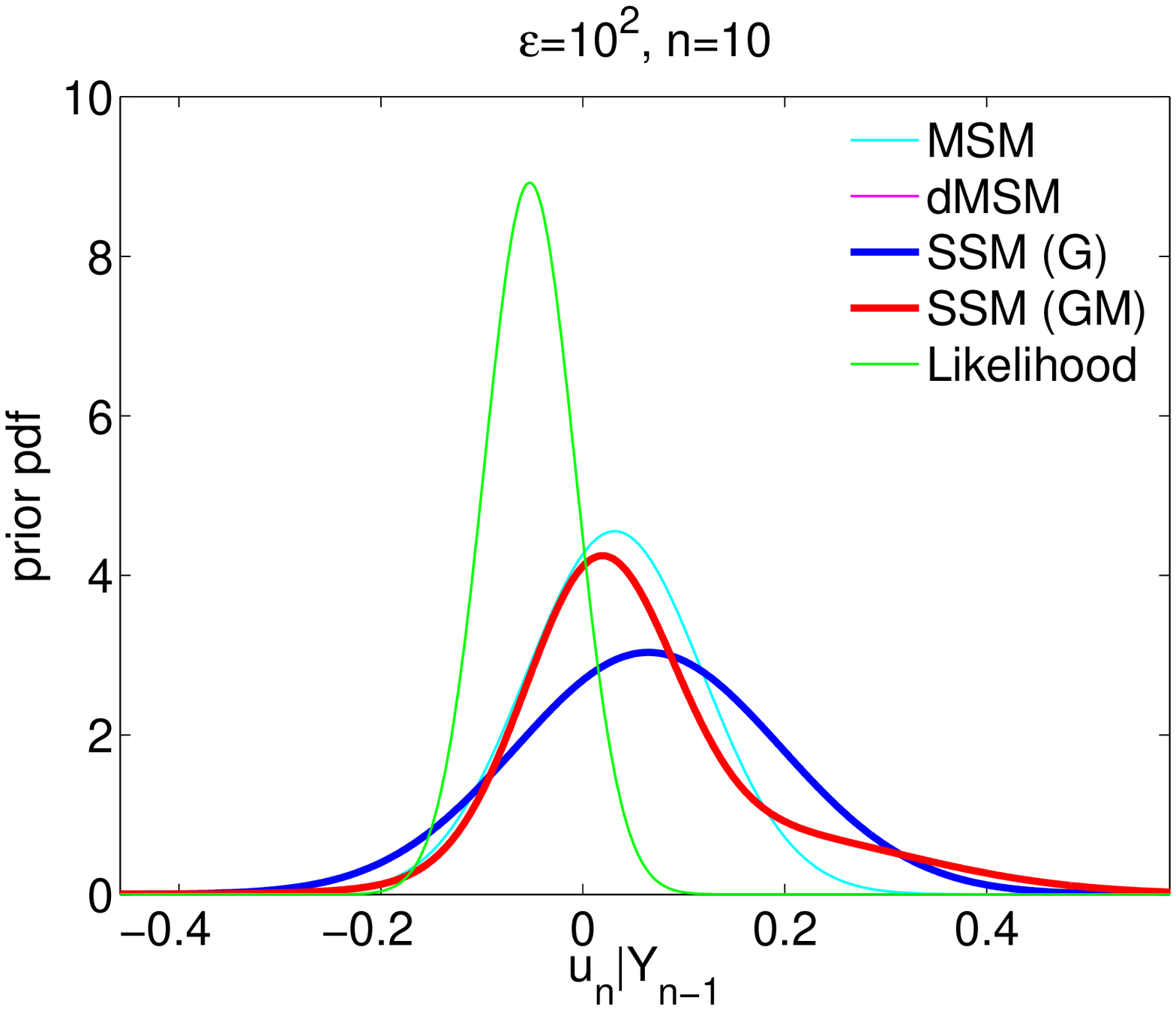} } 
\subfigure
{\includegraphics[width=0.46\textwidth]{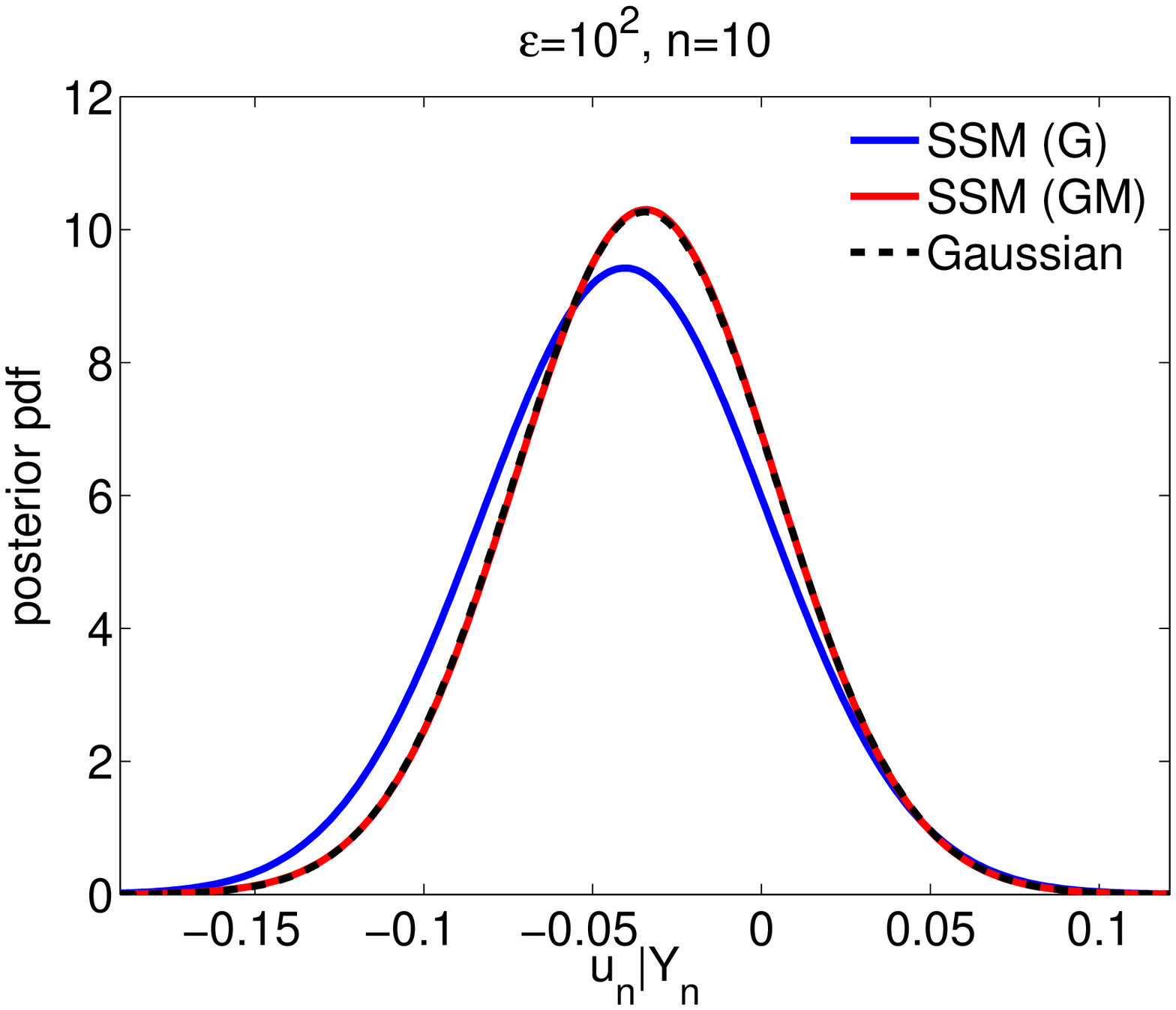} } 
}
\vspace{-0.2in}
\caption{
The Gaussian and Gaussian mixture approximations of SSM prior (left) 
and Gaussian approximations of SSM posterior (right) 
when $\epsilon = 10^{1}$ (top)
and 
$\epsilon = 10^{2}$ (bottom).
The dashed lines are two Gaussian kernels 
of SSM approximation (top-left)
and Gaussian approximation of Gaussian mixture SSM (right).
} 
\label{fig:gsae10100}

\vspace{-0.1in}
\centerline{
\subfigure
{\includegraphics[width=0.46\textwidth]{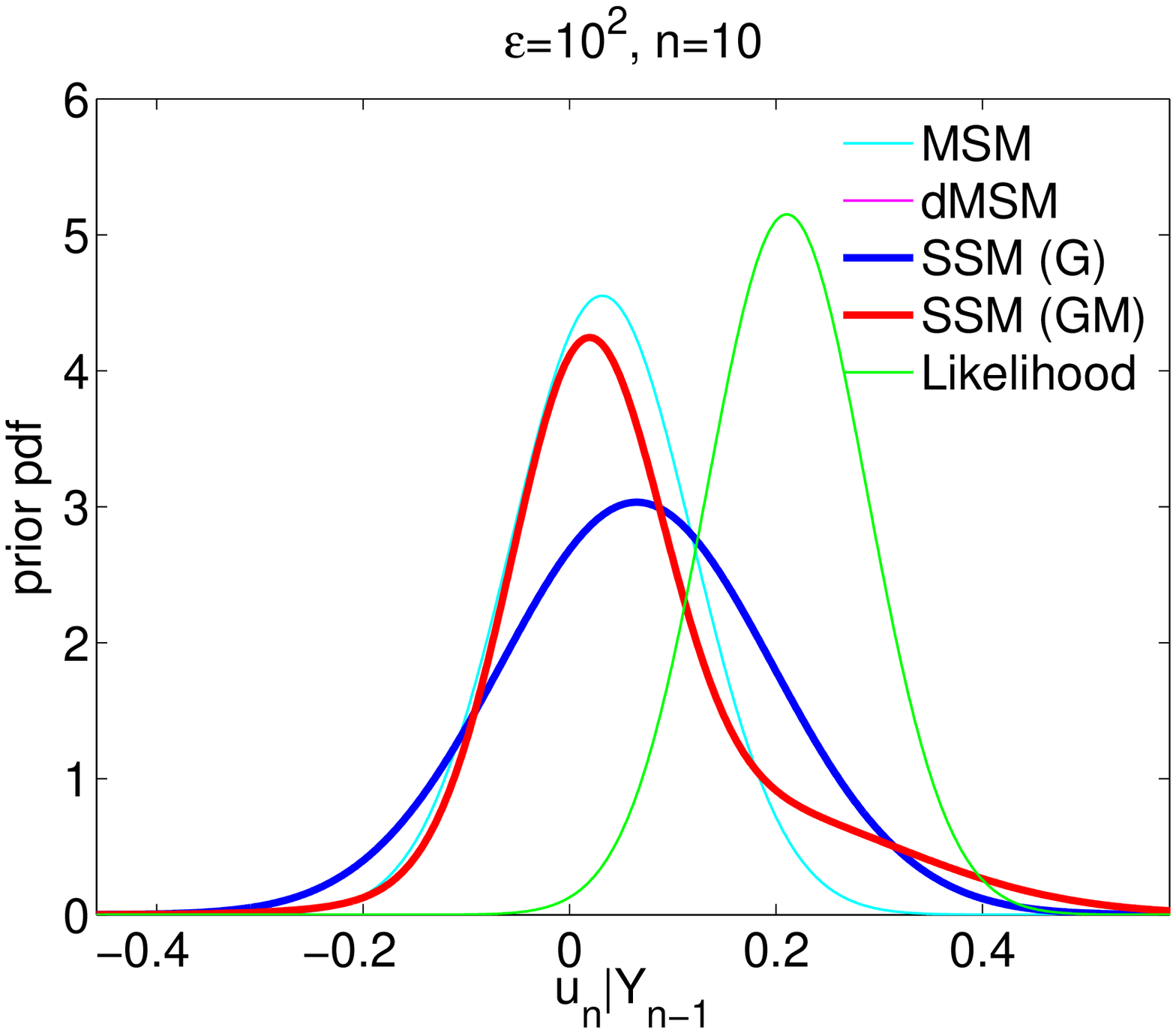} } 
\subfigure
{\includegraphics[width=0.46\textwidth]{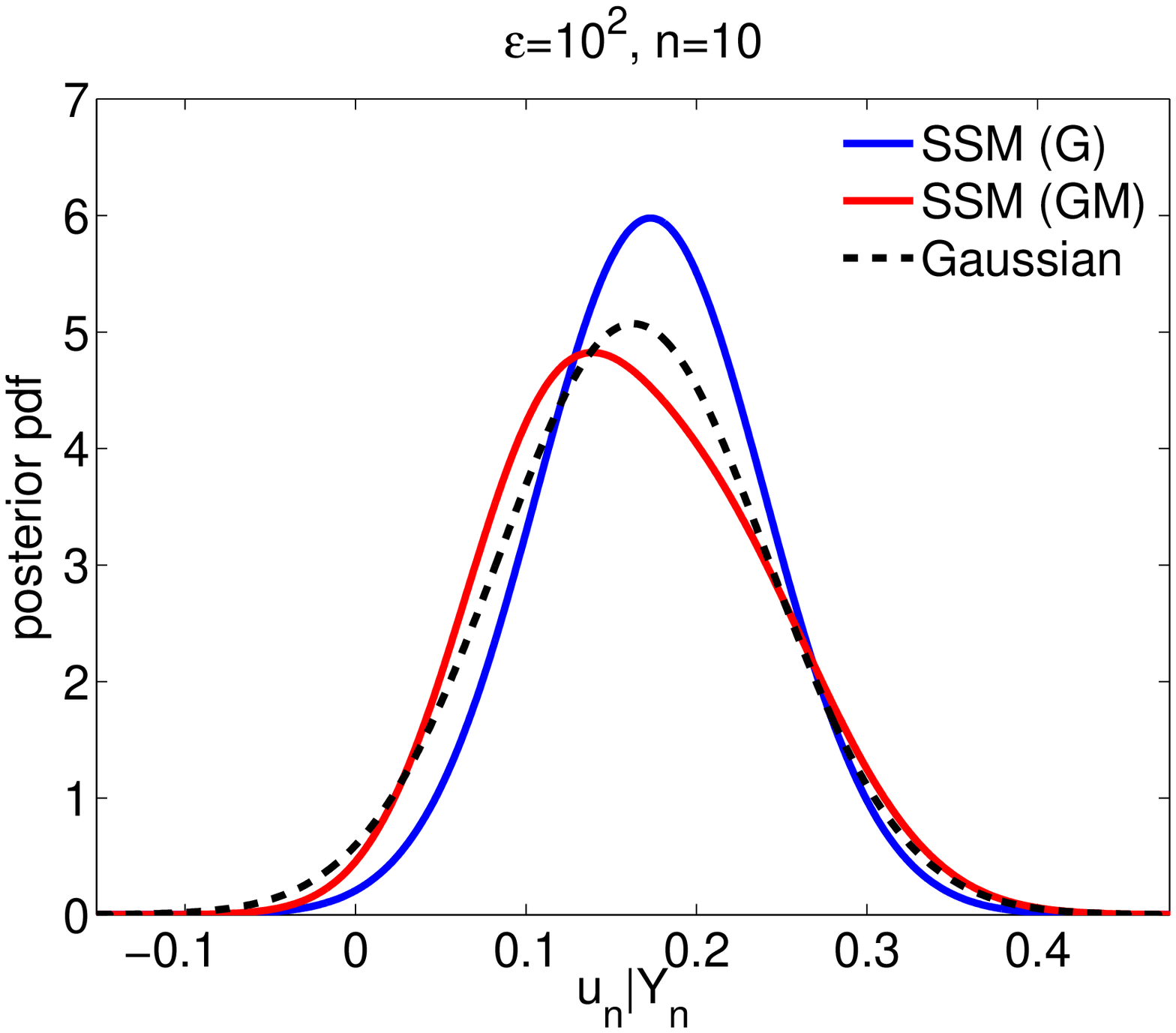} } 
}
\vspace{-0.2in}
\caption{
The Gaussian and Gaussian sum approximations of SSM prior (left) 
and Gaussian approximations of SSM posterior (right) when $\epsilon = 10^{2}$
and
$R=0.75E$.
} 
\label{fig:gsa2e100} 
\end{figure}

\section{Conclusions}
\label{sec:conclusions}
In this paper we have
employed simplified models for the estimation of a partially
observed turbulent signal. 
Our test bed, the switching stochastic model (SSM),
is a stochastic differential equation 
driven by a sign-alternating two-state Markov process.
The system is either forced or dissipated depending on the sign of the driving signal, 
and as a consequence exhibits intermittent turbulent bursting. It is
a cheap surrogate for turbulent signal generation, allowing rapid
prototyping of a variety of approximate filters -- filters with
model error.  Two approximate models (MSM, DSM) for SSM
have been constructed 
via simplification of the switching process underlying the turbulent
bursting,
leading to a Gaussian description for the filtering solution.
We study 
the moment generating function (MGF) with respect to the time integral of the switching process 
to reveal that
these two models precisely mimic the SSM behavior
when the switching frequency is relatively high.
In addition to these two models,
based on the same argument,
we also build 
two models (dMSM, dDSM)
whose regime of validity is rare switching for the driving signal.
We associate these two models with Gaussian sum filters.

We first use 
the ergodicity of the switching processes
to prove MSM (dMSM) is the high (low) switching frequency limit
of SSM and DSM (dDSM).  Besides verifying the consistency 
of the proposed approximate models,
the convergence results give rise to 
an analytic determination of DSM (dDSM) parameters 
when the time-scales of driving input and 
system output 
are well separated.
We achieve this 
from the comparison between
asymptotics of MGFs
in each of two opposing parameter regimes,
because their matching implies
the lower order moments of the corresponding DSM (dDSM)
are very close to those of SSM.
The result again gives rise to
a determination of DSM (dDSM) parameters 
when two time-scales are weakly separated.
In this case, 
we numerically find a 
minimizer of
the sum-of-squares error function between 
the mean and variance of
SSM and DSM (dDSM)
for which
the previous analytic solutions is used as the initial candidate.
In our numerical simulations,
the filtering results utilizing DSM (dDSM) with 
the parameters
tuned according to our suggestions
show
significant improvements 
in accuracy
in all the parameter regimes that we examined.
Furthermore, 
when the time-scale separation is weak,
the cost of performing 
the minimizations 
can be alleviated 
by averaging the parameters 
calculated only
for a number of observation time steps,
while maintaining the accuracy
of the filtering solution
to a considerable extent.

We have used the tools from
three different research areas:
limit theorems, asymptotic analysis
and computational optimization
to complete the whole scenario.
These methods are not separate but carefully 
chained
together 
through a solution cascade to 
provide a significant step in the analysis
and development of filters utilizing 
approximate models suggested from
a rigorous analysis of the underlying system.
As the ultimate goal of filtering with model error
is to estimate the system state and associated uncertainties 
of real-world turbulence, at tractable cost,
our
future work will include the development of these
algorithms, and their benchmarking, in the case
where the true signal is not generated by SSM, but rather
by a real turbulence model.

\section*{Acknowledgement}
Both authors are supported by the 
ERC-AMSTAT grant No.~226488. AMS is also
supported by EPSRC and ONR.

\begin{appendix}
\numberwithin{equation}{section}

\section{Switching Stochastic Model}
\label{app:ssm}
Here we analyze SSM.
Subsection~\ref{subsec:selecton} 
is with regard to the computation of MGFs of integral process of driving signal,
and subsection~\ref{subsec:amgfip}
to their asymptotic behaviors.
We develop SSM filters in 
subsection~\ref{subsec:ssmfilter}.
Note 
$\epsilon$ 
in $\lambda_{+}/\epsilon$
and $\lambda_{-}/\epsilon$
is dropped out
for the notational economy.

\subsection{Moment generating function (MGF) of integral process}
\label{subsec:selecton}
Here we aim to 
analytically compute 
\begin{equation}
  \begin{split}
\label{eq:ssmmgfjt20}
& \left\langle e^{\alpha (\Gamma_T-\Gamma_t )} \right\rangle \\
& \left\langle e^{\alpha (\Gamma_T-\Gamma_t )} \vert \gamma_0=\gamma_{\pm} \right\rangle
  \end{split}
\end{equation}
where $\Gamma_t = \int^t_0 {\gamma}(s) ds $.

We first point out that
it suffices 
provide the formula for
\begin{equation}
  \begin{split}
\label{eq:condssmmgf}
\left\langle e^{\alpha \Gamma_t} \vert \gamma_0 =\gamma_{+} \right\rangle.
  \end{split}
\end{equation}
Once it is done,
that of
$\left\langle e^{\alpha \Gamma_t} \vert \gamma_0 =\gamma_{-} \right\rangle$
is immediate from the exchange
$(\lambda_{+},\gamma_{+})
\leftrightarrow
(\lambda_{-},\gamma_{-})$.
Because
$\rho_\gamma(t)\equiv ( \mathbb{P}( \gamma_t =\gamma_{+} ), \mathbb{P}( \gamma_t =\gamma_{-} ) )^t$
solves
$d{\rho}_\gamma(t)/{dt}=L^t\rho_\gamma(t)$,
where
upper $t$ denotes transpose and
\begin{equation*}
L \equiv
\left( \begin{array}{cc} -\lambda_{+} &\lambda_{+}
\\ \lambda_{-} & -\lambda_{-} \end{array} \right),
\end{equation*}
it satisfies
\begin{equation}
  \begin{split}
\label{eq:probgamma}
\left( \begin{array}{c} 
\mathbb{P}( \gamma_t =\gamma_{+} )
\\
\mathbb{P}( \gamma_t =\gamma_{-} )
\end{array} \right) 
&=
\frac{1}{
\lambda_{+} +\lambda_{-}
}
 \left( \begin{array}{cc} 
\lambda_{-} +\lambda_{+}e^{-(\lambda_{-} +\lambda_{+})t} & \;  
\lambda_{-} -\lambda_{-} e^{-(\lambda_{-} +\lambda_{+})t}\\ 
\lambda_{+} -\lambda_{+} e^{-(\lambda_{-} +\lambda_{+})t}& \;
\lambda_{+} +\lambda_{-}e^{-(\lambda_{-} +\lambda_{+})t}
\end{array} \right) 
\left( \begin{array}{c} 
\mathbb{P}( \gamma_0 =\gamma_{+} )
\\
\mathbb{P}( \gamma_0 =\gamma_{-} )
\end{array} \right).
  \end{split}
\end{equation}
Then from substituting
(\ref{eq:probgamma})
into
\begin{equation}
  \begin{split}
\label{eq:ssmmgfjt}
\left\langle e^{\alpha (\Gamma_T-\Gamma_t )} \right\rangle 
&=
\mathbb{P}( \gamma_t =\gamma_{+} )
\left\langle e^{\alpha (\Gamma_T-\Gamma_t )} \vert \gamma_t =\gamma_{+} \right\rangle 
+
\mathbb{P}( \gamma_t =\gamma_{-} )
\left\langle e^{\alpha (\Gamma_T-\Gamma_t )} \vert \gamma_t =\gamma_{-} \right\rangle \\
&=
\mathbb{P}( \gamma_0 =\gamma_{+} )
\left\langle e^{\alpha (\Gamma_T-\Gamma_t )} \vert \gamma_0 =\gamma_{+} \right\rangle 
+
\mathbb{P}( \gamma_0 =\gamma_{-} )
\left\langle e^{\alpha (\Gamma_T-\Gamma_t )} \vert \gamma_0 =\gamma_{-} \right\rangle
  \end{split}
\end{equation}
it follows that
\begin{equation}
  \begin{split}
\label{eq:probgammats}
 \left\langle e^{\alpha (\Gamma_T-\Gamma_t )} \vert \gamma_0=\gamma_{+} \right\rangle 
 =
\frac{1}{
\lambda_{+} +\lambda_{-}
}
& \Bigg( 
\left(\lambda_{-} +\lambda_{+}e^{-(\lambda_{-} +\lambda_{+})t} \right)
\left\langle e^{\alpha (\Gamma_T-\Gamma_t )} \vert \gamma_t =\gamma_{+} \right\rangle \\
& \quad +
\left(\lambda_{+} -\lambda_{+} e^{-(\lambda_{-} +\lambda_{+})t}\right)
\left\langle e^{\alpha (\Gamma_T-\Gamma_t )} \vert \gamma_t =\gamma_{-} \right\rangle 
\Bigg).
  \end{split}
\end{equation}
Making use of
$\left\langle e^{\alpha (\Gamma_T-\Gamma_t )} \vert \gamma_t =\gamma_{+} \right\rangle 
=\left\langle e^{\alpha \Gamma_{T-t}} \vert \gamma_0 =\gamma_{+} \right\rangle$,
Eqs.~(\ref{eq:ssmmgfjt20}) are expressed in terms of
(\ref{eq:condssmmgf}) through
(\ref{eq:ssmmgfjt}), (\ref{eq:probgammats}).

We next attempt to compute 
Eq.~(\ref{eq:condssmmgf}).
In what follows, we will show
that
(\ref{eq:condssmmgf})
is given by
(\ref{eq:ssmcharasyhomo})
for identical $\lambda_{+}=\lambda_{-}$,
and that
(\ref{eq:condssmmgf})
can be computed 
with the help of 
(\ref{eq:ssmcharasy}),
(\ref{eq:pdfNt}) 
for distinct $\lambda_{+} \neq \lambda_{-}$.
When $\gamma(0)=\gamma_{+}$, let
$\gamma_k$ denote the value of $\gamma(s)$ after undergoing $k$ transitions, i.e.,
$\gamma_k=\gamma_{+}$ for even $k$ and $\gamma_k=\gamma_{-}$ for odd $k$.
Let $\tau_k \triangleq \tau^{\gamma_{+}}$ 
for even $k$ and
$\tau_k \triangleq \tau^{\gamma_{-}}$ for odd $k$.
Let
$T_{n}=\sum_{k=0}^{n-1} \tau_k$
and
let 
$N_t= \max\lbrace n\in \mathbb{N} : T_n \leq t\rbrace$
denote the number of transitions of $\gamma(s)$ in the interval $s\in (0, t]$.
From $\tau_k = T_{k+1}-T_k$, we have
\begin{equation*}
\Gamma_t = \int^t_0 {\gamma}(s) ds = \sum_{k=0}^{N_t-1} \gamma_k \tau_k + \gamma_{N_t}(t-T_{N_t})
= \sum_{k=0}^{N_t-1} (\gamma_k-\gamma_{N_t}) \tau_k + \gamma_{N_t}t.
\end{equation*}
Note $\tau \sim \exp(r)$ then
$\mathbb{E}(e^{\alpha \tau})=\int_0^\infty e^{\alpha t} re^{-rt} dt = (1-{\alpha \over r})^{-1}$ for $\alpha< r$.
Since
$\lbrace \tau_k \rbrace_{k=0}^{n-1}$ are mutually independent
and $\tau_k$ is distributed by exponential distribution,
Eq.~(\ref{eq:condssmmgf}) allows for
a formal expansion 
\begin{equation}
  \begin{split}
\label{eq:ssmcharasy}
\left\langle e^{\alpha \Gamma_t} \vert \gamma_0 =\gamma_{+} \right\rangle 
& = \sum_{n=0}^\infty 
\mathbb{P}(N_t=n\vert \gamma_0 =\gamma_{+} )
\left\langle e^{\alpha \left(\sum_{k=0}^{n-1} (\gamma_k-\gamma_{n})\tau_k +\gamma_{n}t\right)}
\vert \gamma_0 =\gamma_{+}
\right\rangle \\
& = \sum_{\text{n:even}} 
\mathbb{P}(N_t=n\vert \gamma_0 =\gamma_{+} )
e^{\alpha \gamma_{+}t} 
\left( 1-\frac{\alpha (\gamma_{-}-\gamma_{+})}{\lambda_{-}}\right)^{-\frac{n}{2}} \\
& \quad + \sum_{\text{n:odd}} 
\mathbb{P}(N_t=n\vert \gamma_0 =\gamma_{+} )
 e^{\alpha \gamma_{-}t} 
\left( 1-\frac{\alpha (\gamma_{+}-\gamma_{-})}{\lambda_{+}}\right)^{-\frac{n+1}{2}}
  \end{split}
\end{equation}
for
$ {\lambda_{-} \over \gamma_{-}-\gamma_{+} } < \alpha < {\lambda_{+} \over \gamma_{+}-\gamma_{-} } $.

To compute
$\mathbb{P}(N_t=n\vert \gamma_0 =\gamma_{+} )$,
notice
\begin{equation*}
  \begin{split}
\mathbb{P}(N_t=n) 
& = \mathbb{P}(N_t \geq n)-\mathbb{P}(N_t \geq n+1) \\
& = \mathbb{P}(T_n \leq t)-\mathbb{P}(T_{n+1} \leq t)
  \end{split}
\end{equation*}
which is,
in other words,
\begin{equation}
  \begin{split}
\label{eq:pdfofNt}
\begin{cases}
\mathbb{P}(N_t=0) & =1-\int^t_0 f_1(s)ds\\
\mathbb{P}(N_t=n) & = \int^t_0 f_n(s)ds -\int^t_0 f_{n+1}(s)ds, \quad n\geq 1
\end{cases}
  \end{split}
\end{equation}
where $f_n$ denotes the probability density function of $T_n$.
Let
$f_{+}(t) = \lambda_{+}e^{-\lambda_{+}t}$
and
$f_{-}(t) = \lambda_{-}e^{-\lambda_{-}t}$,
and let $*$ denote the convolution,
then
\begin{equation*}
  \begin{split}
f_{m+n} (t)
& = (f_{+})^{*m}*(f_{-})^{*n} (t) \\
& = \frac{\lambda_{+}^m\lambda_{-}^n}{(m-1)!(n-1)!} \int^t_0 s^{m-1}e^{-\lambda_{+}s}
(t-s)^{n-1}e^{-\lambda_{-}(t-s)}ds\\
& = \frac{\lambda_{+}^m\lambda_{-}^n}{(m-1)!(n-1)!} 
e^{-\lambda_{-}t}t^{m+n-1} \int^1_0 s^{m-1}(1-s)^{n-1}e^{-(\lambda_{+}-\lambda_{-})ts}ds\\
  \end{split}
\end{equation*}
with $|m-n|=1$.

\subsubsection{$\lambda_{+} = \lambda_{-}$}
In case of $\lambda_{+} = \lambda_{-}$,
both are denoted by $\lambda$.
Note
the beta function is given by
\begin{equation*}
\int^1_0 s^{m-1}(1-s)^{n-1}ds =
\frac{(m-1)!(n-1)!}{(m+n-1)!}
\end{equation*}
and the use of
integration by parts
yields
\begin{equation*}
  \begin{split}
\int^t_0 {f}_n ds 
= \int^t_0 \lambda^n \frac{s^{n-1}}{(n-1)!}e^{-\lambda s}ds
=\frac{\lambda^n}{n!}t^n e^{-\lambda t} + \int^t_0 f_{n+1}ds.
  \end{split}
\end{equation*}
Therefore,
from (\ref{eq:pdfofNt}),
$N_t$ is distributed by
Poisson distribution with parameter ${\lambda} t$
\begin{equation*}
\mathbb{P}(N_t = n ) = e^{-{\lambda}t}\frac{({\lambda}t)^n}{n!}.
\end{equation*}
In this case,
$\mathbb{P}(N_t=n\vert \gamma_0 =\gamma_{+} ) = \mathbb{P}(N_t=n)$
and
Eq.~(\ref{eq:ssmcharasy})
is summed into
the closed form
\begin{equation}
  \begin{split}
\label{eq:ssmcharasyhomo}
\left\langle e^{\alpha \Gamma_t} \vert \gamma_0 =\gamma_{+} \right\rangle 
& =  e^{-\lambda t}\Bigg(
e^{\alpha \gamma_{+}t} \cosh \left( \lambda t \left( 1-\frac{\alpha
(\gamma_{-}-\gamma_{+})}{\lambda_{-}}\right)^{-\frac{1}{2}}\right)
\\
& 
\quad
+
e^{\alpha \gamma_{-}t} \sinh \left( \lambda t \left( 1-\frac{\alpha
(\gamma_{+}-\gamma_{-})}{\lambda_{+}}\right)^{-\frac{1}{2}}\right)
\left( 1-\frac{\alpha
(\gamma_{+}-\gamma_{-})}{\lambda_{+}}\right)^{-\frac{1}{2}}
\Bigg)
  \end{split}
\end{equation}
when $ \lambda$ represents  $\lambda_{+} = \lambda_{-}$.

\subsubsection{$\lambda_{+} \neq \lambda_{-}$}
Consider the case where $\lambda_{+}$ and $\lambda_{-}$ are different.
Let $B(m,n;\alpha) \equiv \int^1_0 s^{m-1}(1-s)^{n-1}e^{\alpha s}ds$
then
integration by parts yields
\begin{equation*}
B(m,n;\alpha)=\frac{1}{\alpha} \left[ (m-1)B(m-1,n;\alpha)-(n-1)B(m,n-1;\alpha)\right].
\end{equation*}
Using
$B(m-1,n;\alpha) = e^{\alpha} B(n,m-1;-\alpha)$
and 
\begin{equation*}
\int t^n e^{\alpha t}dt=\frac{t^n e^{\alpha t}}{\alpha}-\frac{n}{\alpha}\int t^{n-1}e^{\alpha t}dt
\end{equation*}
we recursively obtain 
$f_n$ and 
the probability distribution of $N_t$
\begin{equation}
  \begin{split}
\label{eq:pdfNt}
\begin{cases}
\mathbb{P}(N_t=0 \vert \gamma_0=\gamma_{+}) & = e^{-\lambda_{+}t}\\
\mathbb{P}(N_t=1\vert \gamma_0=\gamma_{+}) 
& = \frac{\lambda_{+} (e^{-\lambda_{+}t}-e^{-\lambda_{-}t})}{-(\lambda_{+}-\lambda_{-})}\\
\mathbb{P}(N_t=2\vert \gamma_0=\gamma_{+}) 
&= \frac{ 
\lambda_{-}\lambda_{+}\exp(- \lambda_{-}t) 
- \lambda_{-}\lambda_{+}\exp(- \lambda_{+}t) 
+ t \lambda_{-}\lambda_{+}(\lambda_{-}-\lambda_{+}) \exp(- \lambda_{+}t) 
}{ (\lambda_{-}-\lambda_{+})^2}  \\
& \cdots
\end{cases}
  \end{split}
\end{equation}
from (\ref{eq:pdfofNt}).
One can make use of
(\ref{eq:pdfNt}) 
to compute
a truncation of
(\ref{eq:ssmcharasy}).
We conclude the section with the following theorem.

\bigskip

\begin{Theorem}
\label{thm:unif}
The RHS of
Eq.~(\ref{eq:ssmcharasy}) uniformly converges.
\end{Theorem}
\begin{proof}
Let
$f^{+}_n(s) = (\lambda_{+})^n \frac{s^{n-1}}{(n-1)!}e^{-\lambda_{+} s}$
and
$f^{-}_n(s) = (\lambda_{-})^n \frac{s^{n-1}}{(n-1)!}e^{-\lambda_{-} s}$.
The inequalities
\begin{equation*}
\left( \frac{\lambda_{+}}{\lambda_{-}}\right)^m {f}^{-}_{m+n}(t)
\leq f_{m+n}(t) \leq 
\left( \frac{\lambda_{-}}{\lambda_{+}}\right)^n {f}^{+}_{m+n}(t).
\end{equation*}
hold
from
$1 \leq e^{-(\lambda_{+}-\lambda_{-})ts}\leq 
e^{-(\lambda_{+}-\lambda_{-})t}$
for $s \in [0,1]$.
Let $ f_{m+1+n} (t) = (f_{+})^{*(m+1)}*(f_{-})^{*n} (t) $ then
using
\begin{equation*}
  \begin{split}
\int^t_0 \lambda^n \frac{s^{n-1}}{(n-1)!}e^{-\lambda s}ds
\leq
\int^t_0 \lambda^n \frac{s^{n-1}}{(n-1)!}ds
=\frac{(\lambda t)^n}{n!}
  \end{split}
\end{equation*}
we obtain
\begin{equation*}
\begin{split}
\int^t_0 
\left( f_{m+n}(s)-f_{m+1+n}(s)  \right) ds
& \leq 
\int^t_0 
\left(
\left( \frac{\lambda_{-}}{\lambda_{+}}\right)^n {f}^{+}_{m+n}(s)
-
\left( \frac{\lambda_{+}}{\lambda_{-}}\right)^{m+1} {f}^{-}_{m+1+n}(s)
\right) ds \\
& \leq 
\left( \frac{\lambda_{-}}{\lambda_{+}}\right)^n 
\frac{(\lambda_{+}t)^{m+n}}{(m+n)!}
+
\left( \frac{\lambda_{+}}{\lambda_{-}}\right)^{m+1} 
\frac{(\lambda_{-}t)^{m+1+n}}{(m+1+n)!}.
\end{split}
\end{equation*}
Therefore when $n$ is even
\begin{equation*}
\mathbb{P}(N_t = n \vert \gamma_0=\gamma_{+}) 
\leq
\left( \frac{\lambda_{-}}{\lambda_{+}}\right)^{\frac{n}{2}}
\frac{(\lambda_{+}t)^{n}}{n!}
+
\left( \frac{\lambda_{+}}{\lambda_{-}}\right)^{\frac{n}{2}+1} 
\frac{(\lambda_{-}t)^{n+1}}{(n+1)!}
\end{equation*}
is satisfied
and a similar bound holds when $n$ is odd. 
The application of Weierstrass M-test
leads to
the uniform convergence
in view of Eq.~(\ref{eq:ssmcharasyhomo}).
\end{proof}

\subsection{Asymptotics for MGF of integral process}
\label{subsec:amgfip}
We here derive asymptotic formulae for MGFs 
(\ref{eq:ssmmgfjt20}), (\ref{eq:condssmmgf})
when
$\epsilon \ll 1$ 
(scale-separation regime)
and $\epsilon \gg 1$
(rare-event regime).

\subsubsection{Scale-separation regime}
\label{subsec:scaleseparation}
When 
$\lambda_{+}$ 
and
$\lambda_{-}$ 
are distinct,
and $\epsilon$ is small,
we
replace
$\lambda$ 
in 
Eq.~(\ref{eq:ssmcharasyhomo})
by
the harmonic mean
$\bar{\lambda} \equiv
\frac{2\lambda_{+}\lambda_{-}}{\lambda_{+}+\lambda_{-} }$
to derive an approximation.
This is because
the density function of $T_n$ is the convolution of 
$f_{+}=\lambda_{+}e^{-\lambda_{+}t}$
and
$f_{-}=\lambda_{-}e^{-\lambda_{-}t}$, $n/2$ times respectively,
and
from
\begin{equation*}
\begin{split}
\left( 1- \frac{i\alpha}{\lambda_{-}}\right)^{-\frac{n}{2}}
\left( 1- \frac{i\alpha}{\lambda_{+}}\right)^{-\frac{n}{2}}
\simeq
\left( 1- \frac{i\alpha}{\bar{\lambda}}\right)^{-{n}}
\end{split}
\end{equation*}
we see
$T_n \sim f_{-}^{* n/2}* f_{+}^{* n/2} \simeq f$
where $f$ is the gamma distribution with rate $\bar{\lambda}$.
Note $P(N_t = n ) \simeq e^{-\bar{\lambda}t}(\bar{\lambda}t)^n/n!$ 
from $P(N_t \geq n ) = P(T_n \leq t)$.
Therefore
we get
\begin{equation*}
  \begin{split}
\left\langle e^{\alpha \Gamma_t} \vert \gamma_0 =\gamma_{+}\right\rangle 
& \simeq  e^{-\bar{\lambda} t}\Bigg(
e^{\alpha \gamma_{+}t}
\cosh
\left( \bar{\lambda} t \left( 1-\frac{\alpha
(\gamma_{-}-\gamma_{+})}{\lambda_{-}}\right)^{-\frac{1}{2}}\right)\\
& \quad +
e^{\alpha \gamma_{-}t} 
\sinh
\left( \bar{\lambda} t \left( 1-\frac{\alpha
(\gamma_{+}-\gamma_{-})}{\lambda_{+}}\right)^{-\frac{1}{2}}\right)
\left( 1-\frac{\alpha
(\gamma_{+}-\gamma_{-})}{\lambda_{+}}\right)^{-\frac{1}{2}}
\Bigg)
\\
& \simeq \exp\left( \alpha \bar{\gamma}_\infty t +\alpha^2\frac{3}{8}\frac{
(\gamma_{-}-\gamma_{+})^2(\lambda_{-}^2+\lambda_{+}^2)}{\lambda_{-}\lambda_{+}
(\lambda_{+}+ \lambda_{-})}t
+
\frac{\alpha (\gamma_{+}-\gamma_{-})}{4\lambda_{+}}
\right) 
  \end{split}
\end{equation*}
and
the rescaling 
$\lambda_{\pm} \to {\lambda_{\pm}}/{\epsilon}$
yields
\begin{equation*}
  \begin{split}
\left\langle e^{\alpha \Gamma_t} \vert \gamma_0 =\gamma_{+}\right\rangle 
\simeq \exp\left( \alpha \bar{\gamma}_\infty t +\alpha^2\frac{3}{8}\frac{
(\gamma_{-}-\gamma_{+})^2(\lambda_{-}^2+\lambda_{+}^2)}{\lambda_{-}\lambda_{+}
(\lambda_{+}+ \lambda_{-})}t \epsilon
+
\frac{\alpha (\gamma_{+}-\gamma_{-})}{4\lambda_{+}}
\epsilon
+\mathcal{O}(\epsilon^2)
\right) 
  \end{split}
\end{equation*}
where
$(1+x)^k = 1+kx+(k(k-1)/2!)x^2+\cdots$ and 
$\frac{1}{2}e^{x+\alpha \epsilon}+\frac{1}{2}e^{x+\beta \epsilon}(1+\gamma \epsilon)
\simeq e^{x+ ((\alpha+\beta+\gamma)/2)\epsilon}$ is used.

Further we use 
$c e^{A+\alpha \epsilon} + (1-c)e^{A+\beta \epsilon} \simeq e^{A+( c\alpha+(1-c) \beta )\epsilon}$ 
to obtain
\begin{equation*}
  \begin{split}
\left\langle e^{\alpha (\Gamma_T-\Gamma_t)} \right\rangle 
& \simeq \exp\Bigg( \alpha \bar{\gamma}_\infty (T-t) +\alpha^2\frac{3}{8}\frac{
(\gamma_{-}-\gamma_{+})^2(\lambda_{-}^2+\lambda_{+}^2)}{\lambda_{-}\lambda_{+}
(\lambda_{+}+ \lambda_{-})} (T-t) \epsilon\\
&  \qquad +
\alpha \left(
\mathbb{P}(\gamma_0=\gamma_{+})
\frac{(\gamma_{+}-\gamma_{-})}{4\lambda_{+}} 
+
\mathbb{P}(\gamma_0=\gamma_{-})
\frac{(\gamma_{-}-\gamma_{+})}{4\lambda_{-}} 
\right)
\epsilon  
+\mathcal{O}(\epsilon^2)
\Bigg) 
  \end{split}
\end{equation*}
when
$\epsilon < T-t$.

\subsubsection{Rare-event regime}
\label{subsec:ssmrer}
When $\epsilon$ is large,
we leave the first term in Eq.~(\ref{eq:ssmcharasy})
to obtain
\begin{equation}
  \begin{split}
\label{eq:rer}
\left\langle e^{\alpha \Gamma_T} \right\rangle 
& =
\mathbb{P}( \gamma_0 =\gamma_{+} )
\left\langle e^{\alpha \Gamma_T} \vert \gamma_0 =\gamma_{+} \right\rangle 
+
\mathbb{P}( \gamma_0 =\gamma_{-} )
\left\langle e^{\alpha \Gamma_T} \vert \gamma_0 =\gamma_{-} \right\rangle\\
& \simeq
\mathbb{P}( \gamma_0 =\gamma_{+} )
\exp\left( -\frac{\lambda_{+}}{\epsilon}t+\alpha \gamma_{+}t\right)
+
\mathbb{P}( \gamma_0 =\gamma_{-} )
\exp\left( -\frac{\lambda_{-}}{\epsilon}t+\alpha \gamma_{-}t\right).
  \end{split}
\end{equation}

\subsection{Filters for SSM}
\label{subsec:ssmfilter}
Here we make use of the calculations 
given in subsection~\ref{subsec:selecton}
to define
Gaussian filter and Gaussian sum filter for SSM.

\subsubsection{Gaussian filter}
\label{subsec:ssmgaussianfilter}
We design the filter as the assumed density filter.
Accordingly we assume $u_0$ to be Gaussian,
and further assume
the independence of 
$(u_0,\gamma_0)$ 
hence that of $(u_0,\Gamma_T )$.
Then, from Eq.~(\ref{eq:uqmom}), it satisfies
\begin{equation}
\begin{split}
\label{eq:ssmmomuq}
\left\langle u_T \right\rangle & 
= \left\langle e^{-\Gamma_T}\right\rangle 
\langle u_0 \rangle 
\\
\text{Var}(u_T)
& = \left\langle e^{-2\Gamma_T}\right\rangle 
\Big( \text{Var}(u_0) + \langle u_0 \rangle^2 \Big)
 - \left\langle e^{-\Gamma_T}\right\rangle^2
\langle u_0 \rangle^2 \\
& \quad
+ \sigma_u^2  \int^T_0  \left\langle e^{-2(\Gamma_T-\Gamma_t)} \right\rangle dt
\end{split}
\end{equation}
for SSM.
Either using closed form solution in case of identical $\lambda_{\pm}$
or 
a truncation of the series solution in case of distinct $\lambda_{\pm}$,
one can compute
\begin{equation*}
\left\langle e^{\alpha (\Gamma_T-\Gamma_t)} \right\rangle
\end{equation*}
where $\alpha=-1,-2$ and $t\in [0,T]$,
and
Eq.~(\ref{eq:ssmmomuq}). 
Together with
using Eq.~(\ref{eq:probgamma}) for the prediction of $\gamma_t$, 
the first two moments mapping for $(u_0,\gamma_0) \to (u_T,\gamma_T)$ 
has been achieved.
To complete the filter,
we apply Kalman data assimilation for $u_T$ and keep $\gamma_T$ unchanged
as this is consistent with Bayes' rule when $(u_T, \gamma_T)$ is independent
\cite{doucet2001sequential}.

\subsubsection{Gaussian sum filter}
\label{subsec:ssmgaussiansumfilter}
Let $u_0$ be Gaussian mixture and
the independence of $(u_0,\gamma_0)$ be assumed.
Using
\begin{equation*}
\begin{split}
\mathbb{P}({u}_T)
=\mathbb{P}({\gamma}_0=\gamma_{+})\mathbb{P}({u}_T\vert {\gamma}_0=\gamma_{+})
+\mathbb{P}({\gamma}_0=\gamma_{-})\mathbb{P}({u}_T\vert {\gamma}_0=\gamma_{-})
\end{split}
\end{equation*}
we approximate $u_T$ as Gaussian mixture
with doubled 
Gaussian kernels.
Similarly with Eq.~(\ref{eq:ssmmomuq}),
the mean and variance of $u_T\vert \gamma_0=\gamma_{\pm}$
are determined by
\begin{equation*}
\left\langle e^{\alpha (\Gamma_T-\Gamma_t)} \vert \gamma_0=\gamma_{\pm}\right\rangle
\end{equation*}
where $\alpha=-1,-2$ and $t\in [0,T]$.
Using prior calculations,
the conditioned mean and variance of each kernel are obtained.
Then, using Eq.~(\ref{eq:probgamma}) for the prediction of $\gamma_t$,
the algorithm of $(u_0,\gamma_0) \to (u_T,\gamma_T)$ 
is 
established.

To complete the filter,
we apply Kalman data assimilation for each Gaussian kernel of $u_T$
with care of weights,
while preserving the law of $\gamma_T$.
Because the 
latter
procedure
preserves
the number of Gaussians,
total $2^n$ weighted Gaussian kernels describe the posterior distribution
after $n$ inter-observation time steps provided $u_0$ is Gaussian.

\section{Diffusive Stochastic Models}
\label{app:dsm}

This section is concerned with
DSM and dDSM.
The moments mapping formulae of DSM
are derived
in subsection~\ref{sec:spekfmomentmapping}.
We
study the 
computation of 
MGFs of integral process in
subsection~\ref{subsec:dsmmgf}
and their asymptotic behaviors 
in subsection~\ref{subsec:asmgfip}.

\subsection{DSM moments mapping}
\label{sec:spekfmomentmapping}
We in this subsection provide the moments mapping of
\begin{equation*}
\begin{split}
\begin{cases}
du &= - \gamma  u\,dt+\sigma_u dB_u \\
d\gamma & = -d_\gamma(\gamma-\bar{\gamma}) \,dt+\sigma_\gamma dB_{\gamma}
\end{cases}
\end{split}
\end{equation*}
when $(u_0,\gamma_0)$ is Gaussian.
Let 
$\Gamma_\gamma(t)  \equiv \int^t_0 \gamma(s) ds$
then the path-wise solutions read
\begin{equation*}
\begin{split}
u_t & = e^{-\Gamma_\gamma(t)}u_0 + \sigma_u \int^t_0  e^{-(\Gamma_\gamma(t)-\Gamma_\gamma(s))}dB_u(s) 
\equiv {A}_t+{B}_t \\
\gamma_t & = \bar{\gamma}+(\gamma_0-\bar{\gamma})e^{-d_\gamma t} + \sigma_\gamma
\int^t_0 e^{-d_\gamma (t-s)} dB_\gamma(s). 
\end{split}
\end{equation*}
Define
\begin{equation*}
  \begin{split}
	b_\gamma(t) & 
	\equiv (1-e^{-d_\gamma t)})/ {d_\gamma}
	\\
\mathcal{B}_\gamma(t) & \equiv \sigma_\gamma \int^t_0 ds \int^s_0 e^{-d_\gamma(s-s')} dB_\gamma(s')
  \end{split}
\end{equation*}
then
$	\Gamma_\gamma(t) 
= \bar{\gamma}(t-b_\gamma(t))+b_\gamma(t) \gamma_0 + \mathcal{B}_\gamma(t) $
and 
we have
\begin{equation}
  \begin{split}
\label{eq:mommaping}
\langle u_t \rangle & = \langle A_t \rangle \\
\langle \gamma_t \rangle & 
= \bar{\gamma}+(\langle\gamma_0\rangle-\bar{\gamma})e^{-d_\gamma t}  \\
\text{Var}(u_t ) & = 
\langle u_t^2 \rangle - \langle u_t \rangle^2 =
\langle  A_t^2 \rangle +\langle B_t^2 \rangle - \langle A_t \rangle^2 \\
\text{Var}(\gamma_t ) & = e^{-2d_\gamma t} \text{Var}( \gamma_0 )+ 
\frac{\sigma_\gamma^2}{2d_\gamma}\left( 1-e^{-2d_\gamma t}\right) \\
\text{Cov}(u_t,\gamma_t ) & = 
\langle u_t \gamma_t\rangle - \langle u_t\rangle \langle
\gamma_t\rangle = \bar{\gamma}\left( 1- e^{-d_\gamma t}\right)\langle A_t \rangle + e^{-d_\gamma t}
\langle A_t \gamma_0 \rangle \\
& \qquad \qquad \qquad \qquad \qquad + \langle A_t \dot{\mathcal{B}}_\gamma(t) \rangle
- \langle A_t \rangle \langle \gamma_t \rangle \\
  \end{split}
\end{equation}
where upper dot denotes derivative.

Using
\begin{equation*}
  \begin{split}
	\langle \Gamma_\gamma(t)-\Gamma_\gamma(s) \rangle & = \left( b_\gamma(t)-b_\gamma(s) \right) \langle
	\gamma_0 \rangle 
+\bar{\gamma}((t-s)-(b_\gamma(t)-b_\gamma(s)))
 \\
	\text{Var}\left( \Gamma_\gamma(t)-\Gamma_\gamma(s) \right) & = \left( b_\gamma(t)-b_\gamma(s) \right)^2
	\text{Var}( \gamma_0 ) + \text{Var}\left( \mathcal{B}_\gamma(t)-\mathcal{B}_\gamma(s) \right) \\
	\langle \mathcal{B}_\gamma(t) \rangle & = 0 \\
	\text{Var}(\mathcal{B}_\gamma(t)) & = -\frac{\sigma_\gamma^2}{2d_\gamma^3}
	\left( 3-4e^{-d_\gamma t}+e^{-2d_\gamma t}-2d_\gamma t\right) \\
\text{Var}(\mathcal{B}_\gamma(t)-\mathcal{B}_\gamma(s)) & = -\frac{\sigma_\gamma^2}{d_\gamma^3}
\left( 1+d_\gamma(s-t)+e^{-d_\gamma(s+t)}\times\left( -1-e^{2d_\gamma s}+\cosh(d_\gamma(s-t))
\right) \right) \\
\left\langle e^{- \mathcal{B}_\gamma(t)} \dot{\mathcal{B}}_\gamma(t)\right\rangle
& = -\frac{1}{2} \partial_t \left( \text{Var}(\mathcal{B}_\gamma(t))  \right) 
\left\langle e^{- \mathcal{B}_\gamma(t)} \right\rangle
  \end{split}
\end{equation*}
and using
\begin{equation*}
  \begin{split}
\langle e^{z} \rangle & = e^{ \langle z \rangle + \frac{1}{2}\text{Var}(z) } \\
\langle e^{z}x \rangle & =  
e^{ \langle z \rangle + \frac{1}{2}\text{Var}(z) }
\big( \langle x\rangle + \text{Cov}(x,z ) \big) \\
\langle e^{z} xy \rangle & = 
e^{ \langle z \rangle + \frac{1}{2}\text{Var}(z) }
\Big( \text{Cov}(x,y) + 
\big( \langle x \rangle + \text{Cov}(x,z )  \big)
  \big( \langle y \rangle + \text{Cov}(y,z )  \big)
\Big)
  \end{split}
\end{equation*}
where $(x,y,z)$ is joint Gaussian,
we can compute
\begin{equation*}
  \begin{split}
\left\langle A_t  \right\rangle 
= \left\langle e^{- \Gamma_\gamma(t) }u_0 \right\rangle
& = e^{- \bar{\gamma}(t- b_\gamma(t))}
\left\langle e^{-  b_\gamma(t) \gamma_0 } u_0 \right\rangle
\left\langle e^{- \mathcal{B}_\gamma(t) }\right\rangle
\\
\left\langle A_t\gamma_0  \right\rangle 
= \left\langle e^{- \Gamma_T }u_0\gamma_0 \right\rangle
& = e^{-\bar{\gamma}(t- b_\gamma(t))}
\left\langle e^{-b_\gamma(t) \gamma_0} u_0\gamma_0
\right\rangle  
\left\langle e^{- \mathcal{B}_\gamma(t) }\right\rangle
\\
\left\langle A_t\dot{\mathcal{B}}_\gamma(t)  \right\rangle 
= \left\langle e^{- \Gamma_T }u_0\dot{\mathcal{B}}_\gamma(t) \right\rangle
& = e^{-\bar{\gamma}(t- b_\gamma(t))}
\left\langle e^{-b_\gamma(t) \gamma_0} u_0 \right\rangle  
\left\langle e^{-\mathcal{B}_\gamma(t)} \dot{\mathcal{B}}_\gamma(t)\right\rangle 
\\
\left\langle A_t^2  \right\rangle 
= \left\langle e^{- 2\Gamma_T }u_0^2 \right\rangle
& = e^{-2 \bar{\gamma}(t- b_\gamma(t))}
\left\langle e^{-2 b_\gamma(t) \gamma_0 } u_0^2 \right\rangle
\left\langle e^{- 2\mathcal{B}_\gamma(t) }\right\rangle
\\
\langle   B_t^2 \rangle 
& 
= \sigma_u^2 \int^t_0 \left \langle e^{- 2(\Gamma_\gamma(t)-\Gamma_\gamma(s))} \right\rangle ds 
\\
& = \sigma_u^2 \int^t_0 e^{-2 \langle \Gamma_\gamma(t)-\Gamma_\gamma(s) \rangle+2 
\text{Var}\left( \Gamma_\gamma(t)-\Gamma_\gamma(s) \right) } ds
  \end{split}
\end{equation*}
and thereby
Eq.~(\ref{eq:mommaping}).
Here numerical integrator like 
the trapezoidal rule can be employed 
for the computation of
$\langle B_t^2 \rangle$.
As a consequence,
the analytic moment-mapping $(u_0,\gamma_0 ) \to (u_t,\gamma_t)$ is obtained.

\subsection{Moment generating function of integral process}
\label{subsec:dsmmgf}
Recall
\begin{equation*}
\begin{split}
\textbf{(DSM)}\qquad
\left\{ 
\begin{array}{ll}
d\widehat{u} &= - \widehat{\gamma}  \widehat{u}\,dt+\sigma_u dB_u \\
d\widehat{\gamma} & = -\frac{\nu}{\epsilon}(\widehat{\gamma}-\mu) \,dt+
\frac{\sigma}{\sqrt{\epsilon}} dB_{\gamma} \\
\end{array} \right. 
\end{split}
\end{equation*}
and
$\widehat{\Gamma}_t = \int^t_0 \widehat{\gamma}(s) ds$
then, from the preceding subsection,
we have
\begin{equation*}
\begin{split}
 \langle \widehat{\Gamma}_t\rangle
& =\mu 
t
+\left\langle \widehat{\gamma}_0
-\mu \right\rangle b_\gamma(t)\\
\text{Var}(\widehat{\Gamma}_t)
& = 
\text{Var}(\widehat{\gamma}_0)b_\gamma(t)^2 +\text{Var}( \mathcal{B}_\gamma(t)) \\
	\langle \widehat{\Gamma}_t-\widehat{\Gamma}_s \rangle 
& = b_\gamma(t-s)  \langle \widehat{\gamma}_s \rangle 
+\mu((t-s)-b_\gamma(t-s))  \\
	\text{Var}( \widehat{\Gamma}_t-\widehat{\Gamma}_s ) & = (b_\gamma(t-s))^2
	\text{Var}( \widehat{\gamma}_s ) + \text{Var}\left( \mathcal{B}_\gamma(t-s) \right) \\
\langle \widehat{\gamma}_t \rangle & 
= \mu+(\langle\widehat{\gamma}_0\rangle-\mu)e^{-\nu t/\epsilon}  \\
\text{Var}(\widehat{\gamma}_t ) & = e^{-2\nu t/\epsilon} \text{Var}( \widehat{\gamma}_0 )+ 
\frac{\sigma^2}{2\nu}\left( 1-e^{-2\nu t/\epsilon}\right) \\
  \end{split}
\end{equation*}
where
\begin{equation*}
\begin{split}
	b_\gamma(t) & = \epsilon (1-e^{-\nu t/\epsilon)})/ {\nu}  \\
	\text{Var}(\mathcal{B}_\gamma(t)) &  = -\epsilon^2 \frac{\sigma^2}{2\nu^3}
	\left( 3-4e^{-\nu t/\epsilon}+e^{-2\nu t/\epsilon}-2\nu t/\epsilon\right).
\end{split}
\end{equation*}
Let $\widehat{\gamma}_0$ be Gaussian
then $\widehat{\Gamma}_t$
is Gaussian as well,
and the MGFs
\begin{equation}
\begin{split}
\label{eq:spmcharasy}
\left\langle e^{\alpha \widehat{\Gamma}_t}\right\rangle 
& = \exp\left( {\alpha \langle \widehat{\Gamma}_t\rangle +
\frac{\alpha^2}{2}\text{Var}(\widehat{\Gamma}_t)} \right)\\
\left\langle e^{\alpha (\widehat{\Gamma}_t-\widehat{\Gamma}_s)} \right\rangle 
& = \exp \left( {\alpha \langle \widehat{\Gamma}_t-\widehat{\Gamma}_s\rangle +
\frac{\alpha^2}{2}\text{Var}(\widehat{\Gamma}_t-\widehat{\Gamma}_s)} \right)
\end{split}
\end{equation}
can be computed.

\subsection{Asymptotics of MGFs of integral process}
\label{subsec:asmgfip}

\subsubsection{Scale-separation regime}
\label{subsec:dsmssr}
For small $\epsilon$,
from
substituting 
$b_\gamma(t)  = \frac{1}{d}\epsilon + \mathcal{O}(\epsilon^2)$
and
$\text{Var}(\mathcal{B}_\gamma(t)-\mathcal{B}_\gamma(s)) =\frac{\sigma^2}{d^2}(t-s)\epsilon 
+\mathcal{O}(\epsilon^2)$
into
(\ref{eq:spmcharasy}),
we obtain
\begin{equation*}
\begin{split}
\left\langle e^{\alpha \widehat{\Gamma}_t}\right\rangle 
& = \exp\left( \alpha \left( \mu t 
+ \langle \widehat{\gamma}_0-\mu\rangle
\frac{\epsilon}{\nu}\right)
+\alpha^2 \frac{\sigma^2}{2\nu^2}t\epsilon +\mathcal{O}(\epsilon^2) \right) 
\quad \epsilon < t\\
\left\langle e^{\alpha (\widehat{\Gamma}_t-\widehat{\Gamma}_s)} \right\rangle 
& = \exp\left( \alpha \mu(t-s) 
+\alpha^2 \frac{\sigma^2}{2\nu^2}(t-s)\epsilon + \mathcal{O}(\epsilon^2) \right) 
\quad \epsilon < s.
\end{split}
\end{equation*}

\subsubsection{Rare-event regime}
\label{subsec:dsmrer}
When $\epsilon$ is large,
we use $b_r(t)=t-\frac{1}{2}\nu t^2/\epsilon+\mathcal{O}(1/\epsilon^2)$
and
$\text{Var}(\mathcal{B}_\gamma(t))=\frac{\sigma^2}{3}t^3/\epsilon+\mathcal{O}(1/\epsilon^2)$
to obtain
\begin{equation*}
\begin{split}
\left\langle e^{\alpha \widehat{\Gamma}_t} \right\rangle 
& =\exp
\Bigg( \alpha \left( \mu 
t
+ \langle \widehat{\gamma}_0-\mu\rangle
\left(t-\frac{\nu}{2} \frac{t^2}{\epsilon}\right) \right)+\frac{\alpha^2}{2} 
\left(
\text{Var}(\widehat{\gamma}_0)
\left(t^2-\nu \frac{t^3}{\epsilon} \right)
+
\frac{\sigma^2}{3} \frac{t^3}{\epsilon}
\right) \\
& \qquad \qquad + \mathcal{O}\left(\frac{1}{\epsilon^2} \right)
\Bigg)
\end{split}
\end{equation*}
for DSM.
Therefore,
in case of dDSM, 
we have
\begin{equation}
\begin{split}
\label{eq:dsmdasym}
\left\langle e^{\alpha \widehat{\Gamma}'_t} \vert \widehat{\gamma}_0'=\gamma_{\pm} \right\rangle 
&
=\exp\left( 
\alpha \left( \mu_{\pm}
t
+ ( {\gamma}_{\pm}-\mu_{\pm})
\left(t-\frac{\nu_{\pm}}{2} \frac{ t^2}{\epsilon} \right) \right)+\frac{\alpha^2}{2} 
\frac{(\sigma_{\pm})^2}{3} \frac{t^3}{\epsilon}
+ \mathcal{O}\left(\frac{1}{\epsilon^2} \right)
\right)\\
& =\exp\left( 
\alpha 
\left(
{\gamma}_{\pm}t - \frac{1}{2\epsilon} ( {\gamma}_{\pm}-\mu_{\pm}) \nu_{\pm} t^2 \right) 
+\frac{\alpha^2}{2} \frac{(\sigma_{\pm})^2}{3\epsilon}t^3 
+ \mathcal{O}\left(\frac{1}{\epsilon^2} \right)
\right).
\end{split}
\end{equation}

\section{Proofs of Theorems}
\label{sec:equivalenceproof}

\subsection{Scale-Separation Limit}

\begin{proof}
[of Lemma~\ref{thm:uconv2}]
The convergence of the mean and variance
follows from 
Eq.~(\ref{eq:uqmom}) and
the bounded convergence theorem.

To show $L^2(\Omega;\mathbb{R})$ convergence,
from
Eq.~(\ref{eq:voc})
we obtain
\begin{equation*}
\begin{split}
u_T^\epsilon -\bar{u}_T
& = \left(e^{-\Gamma^\epsilon_T}u_0^\epsilon- e^{-\bar{\gamma}T}\bar{u}_0 \right)
+ \sigma_u \int^T_0  \left( e^{-(\Gamma^\epsilon_T-\Gamma^\epsilon_t)} - e^{-\bar{\gamma}(T-t)}   \right) dB_u(t) \\
\end{split}
\end{equation*}
and
\begin{equation*}
\begin{split}
 |u_T^\epsilon -\bar{u}_T|^2  
&  \leq 2   \left\lvert 
e^{-\Gamma^\epsilon_T}u_0^\epsilon- e^{-\bar{\gamma}T}\bar{u}_0 
\right\rvert^2
+2 \sigma_u^2  
\left\lvert \int^T_0  \left( e^{-(\Gamma^\epsilon_T-\Gamma^\epsilon_t)} - e^{-\bar{\gamma}(T-t)}   \right) dB_u(t)
\right\rvert^2.
\end{split}
\end{equation*}
Use
It\^o lemma to obtain
\begin{equation*}
\begin{split}
\left\langle |u_T^\epsilon -\bar{u}_T|^2 \right\rangle 
&  \leq 2  \left\langle \left\lvert 
e^{-\Gamma^\epsilon_T}u_0^\epsilon
- e^{-\bar{\gamma}T}\bar{u}_0 
\right\rvert^2
 \right\rangle 
+2 \sigma_u^2  
  \int^T_0  \left\langle
\left\lvert
e^{-(\Gamma^\epsilon_T-\Gamma^\epsilon_t)} - e^{-\bar{\gamma}(T-t)}   
\right\rvert^2
\right\rangle dt.
\end{split}
\end{equation*}
Here note the term
\begin{equation}
\begin{split}
\label{eq:integral}
\left\langle e^{-2(\Gamma^\epsilon_T-\Gamma^\epsilon_t)}\right\rangle  
-2  \left\langle e^{-(\Gamma^\epsilon_T-\Gamma^\epsilon_t)} \right\rangle
e^{-\bar{\gamma}(T-t)} + e^{-2\bar{\gamma}(T-t)}  
\end{split}
\end{equation}
converges to zero as $\epsilon \to 0$.
Then the bounded convergence theorem ensures
the integration of (\ref{eq:integral}) also 
converges to zero as $\epsilon \to 0$.
Therefore the convergence
\begin{equation*}
\begin{split}
\left\langle  |u_T^\epsilon - \bar{u}_T|^2 \right\rangle 
\to 0
\end{split}
\end{equation*}
as $\epsilon \to 0$
holds
for any $\gamma^\epsilon$.
\end{proof}

\bigskip

Now we state and prove
Lemma~\ref{thm:weakconv}
and Lemma~\ref{lem:mgf}
which
will be used to 
prove 
Lemma~\ref{thm:ssmmgf} and
Lemma~\ref{thm:dsmmgf}.

\bigskip

\begin{Lemma}
\label{thm:weakconv}
Let $\mathcal{Y}$ be a 
Markov chain 
or a diffusion process
associated with generator 
$\frac{1}{\epsilon}Q_0 $.
We assume $\mathcal{Y}$ is an ergodic process with invariant measure $\rho^\infty_\mathcal{Y}$
satisfying 
$\text{Null}({Q}_0)=\text{span}\{ \mathbf{1} \}$,
$\text{Null}({Q}_0^*)=\text{span}\{ \rho^\infty_\mathcal{Y} \}$.
Let $\mathcal{X}$ satisfy the ODE
\begin{equation*}
\begin{split}
\frac{d\mathcal{X}}{dt}=f(\mathcal{X},\mathcal{Y})
\end{split}
\end{equation*}
and
let the generator of the combined process $(\mathcal{X},\mathcal{Y})$
be of the form
\begin{equation*}
\begin{split}
Q = \frac{1}{\epsilon}Q_0+Q_1.
\end{split}
\end{equation*}
Let $\bar{\mathcal{X}}$ satisfy the ODE
\begin{equation}
\begin{split}
\label{eq:aode}
\frac{d\bar{\mathcal{X}}}{dt}= \bar{Q}_1(\bar{\mathcal{X}}) 
= \int  f(\bar{\mathcal{X}},\cdot) d\rho^\infty_\mathcal{Y} (\cdot)
\end{split}
\end{equation}
then,
for any $t>0$, $\mathcal{X}(t)$ 
converges weakly
or in distribution
to $\bar{\mathcal{X}}(t)$ as $\epsilon \to 0$
(recall
$X_\epsilon \rightharpoonup X$ is referred to converges weakly 
provided $\mathbb{E}(f(X_\epsilon))\to \mathbb{E}(f(X))$ for any bounded continuous function $f$).
\end{Lemma}

\begin{proof}
[of Lemma~\ref{thm:weakconv}]
The first step is to show that the averaged ODE is given by 
Eq.~(\ref{eq:aode}).
Let be $\Phi$ be a bounded continuous function
and
let 
\begin{equation*}
\begin{split}
v(x,y, t) &= \mathbb{E}( \Phi(\mathcal{X}_t, \mathcal{Y}_t) \vert \mathcal{X}_0 =x,\mathcal{Y}_0=y).
\end{split}
\end{equation*}
Then it satisfies the backward equation
\begin{equation}
\begin{split}
\label{eq:bwdeq}
\partial_t {v}(x,y, t) =Q {v}(x,y,t) =\left( \frac{1}{\epsilon} Q_0 +Q_1 \right){v}(x,y,t).
\end{split}
\end{equation}
We seek solution $v=v(x,y,t)$ in the form of the multi-scale expansion
\begin{equation*}
\begin{split}
v=v_0+\epsilon v_1 + \mathcal{O}(\epsilon^2).
\end{split}
\end{equation*}
From substituting the expansion and equating coefficients of equal powers of $\epsilon$ to zero,
we find
\begin{equation}
\begin{split}
\label{eq:multieq}
\mathcal{O}\left(\frac{1}{\epsilon}\right): & \qquad  Q_0v_0=0 \\
\mathcal{O}(1): & \qquad Q_0v_1 = -Q_1v_0 +\frac{dv_0}{dt}
\end{split}
\end{equation}
and
we see
$v_0$ is independent of $y$
due to $\text{null}(Q_0)=\mathbf{1}$.
The operator $Q_0$ is singular and,
for Eq.~(\ref{eq:multieq}) to have a solution,
the Fredholm alternative implies the solvability condition 
$$ -Q_1v_0+\frac{dv_0}{dt} \perp \text{Null}(Q_0^*).$$
For arbitrary $c(x)$, we find
\begin{equation*}
\begin{split}
\int \int  c(x) \left( \frac{dv_0}{dt}- Q_1 v_0  \right)\,dx d\rho_\infty(y)
=\int c(x) \left( \frac{dv_0}{dt}- \bar{Q}_1 v_0 \right) \,dx
= 0
\end{split}
\end{equation*}
implying that 
$$\frac{dv_0}{dt}- \bar{Q}_1 v_0=0  .$$

The second step is to show the weak convergence.
Substituting 
$$ v=v_0+\epsilon v_1 + r$$
into 
Eq.~(\ref{eq:bwdeq})
yields
\begin{equation*}
\begin{split}
\frac{dr}{dt} &= \left( \frac{1}{\epsilon}Q_0+Q_1 \right) r+ \epsilon q \\
q &= Q_1 v_1-\frac{dv_1}{dt}
\end{split}
\end{equation*}
and
\begin{equation*}
\begin{split}
r(t)=e^{Qt}r(0)+\epsilon \int^t_0 e^{Q(t-s)}q(s)ds
\end{split}
\end{equation*}
from the variation-of-constants.
From $v(t)=e^{Qt}v(0)$
we obtain
$|e^{Qt}|_\infty \leq 1$
because $\Phi$ is bounded.
We then have
\begin{equation*}
\begin{split}
|r(t)|_\infty 
& \leq \epsilon |e^{Qt}|_\infty |r(0)|_\infty +\epsilon \int^t_0 |e^{Q(t-s)}|_\infty |q(s)|_\infty ds \\
& \leq \epsilon |v_1(0)|_\infty +\epsilon \int^t_0 |q(s)|_\infty ds \\
& \leq \epsilon \left( |v_1(0)|_\infty +t \sup_{0 \leq s \leq t} |q(s)|_\infty \right)
\end{split}
\end{equation*}
and obtain
\begin{equation*}
\begin{split}
|v(t)-v_0(t)|_\infty \leq C(T)\epsilon
\end{split}
\end{equation*}
for $0 \leq t \leq T$.
\end{proof}

\bigskip

\begin{Lemma}
\label{lem:mgf}
Let 
$F_{X_\epsilon}(\cdot) \equiv \mathbb{P}(X_\epsilon \leq \cdot)$
and 
$F_{X}$
be the distribution function of $X_\epsilon$
and non-random variable $X$, respectively.
If
\begin{subequations}
\label{eq:covcond}
\begin{align}
\label{eq:covcond1}
&\lim_{ b\to \infty} e^{\alpha b}(F_{X_\epsilon}(b)-1) \to 0 \\
\label{eq:covcond2}
&\lim_{a\to -\infty} e^{\alpha a }F_{X_\epsilon}(a) \to 0\\
\label{eq:covcond3}
&\lim_{a\to -\infty, b\to \infty} 
\int^b_a \left( F_{X_\epsilon}(x)-F_{X}(x)\right) e^{\alpha x}dx \to 0
\end{align}
\end{subequations}
as $\epsilon \to 0$
then $\left\langle e^{\alpha X_{\epsilon}}\right\rangle \to e^{\alpha X}$ follows.
The convergence rate is given by
the lowest one in Eq.~(\ref{eq:covcond}).
\end{Lemma}

\bigskip

\begin{proof}
[of Lemma~\ref{lem:mgf}]
Use integration by parts to obtain
\begin{equation*}
\begin{split}
\left\langle e^{\alpha X_\epsilon} \right\rangle 
&= 
\lim_{a\to -\infty, b\to \infty} 
\int^b_a e^{\alpha x}dF_{X_\epsilon}(x) \\
& = \lim_{a\to -\infty, b\to \infty} 
e^{\alpha x}F_{X_\epsilon}(x) \vert^b_a-\int^b_a F_{X_\epsilon}(x) \alpha e^{\alpha x}dx \\
& =
\lim_{ b\to \infty} 
e^{\alpha b}(F_{X_\epsilon}(b)-1)-
\lim_{a\to -\infty} 
e^{\alpha a }F_{X_\epsilon}(a)
+e^{\alpha X}\\
& \quad - \lim_{a\to -\infty, b\to \infty} 
\int^b_a \left( F_{X_\epsilon}(x)-F_{X}(x)\right) \alpha e^{\alpha x}dx.
\end{split}
\end{equation*}
\end{proof}

\bigskip

\begin{proof}
[of Lemma~\ref{thm:ssmmgf}]
For a bounded continuous function $\Phi$,
let
\begin{equation*}
\begin{split}
v(x,y_i, t) &= \mathbb{E}( \Phi(\Gamma_t, \gamma_t) \vert \Gamma_0 =x,\gamma_0=y_i) \\
\end{split}
\end{equation*}
where
$y_1=\gamma_{+}$
and
$y_2=\gamma_{-}$.
It satisfies
the backward equation
\begin{equation*}
\begin{split}
\partial_t v(x,y_i, t) &= \sum_j L_{ij} v(x,y_j,t)+y_i \partial_x v(x,y_i,t) \\
\end{split}
\end{equation*}
or
\begin{equation*}
\begin{split}
\partial_t {v}(x, t) & =  Q {v}(x,t)\\
Q
&= \frac{1}{\epsilon}Q_0+Q_1 
=
\frac{1}{\epsilon}
 \left( \begin{array}{cc} 
-\lambda_{+} &\lambda_{+}  \\ \lambda_{-} & -\lambda_{-}
\end{array} \right) 
+
 \left( \begin{array}{cc} 
y_1 \partial_x & 0  \\ 0 & y_2 \partial_x
\end{array} \right) 
\end{split}
\end{equation*}
in vector notation.
The generator of $\gamma$ is then given by
\begin{equation*}
L=
\frac{1}{\epsilon}
\left( \begin{array}{cc} -\lambda_{+} &\lambda_{+}
\\ \lambda_{-} & -\lambda_{-} \end{array} \right)
\end{equation*}
and $\gamma$ is ergodic process
\cite{pavliotis2008multiscale}.

From Eq.~(\ref{eq:probgamma}),
the time invariant measure of 
$\gamma$ 
is
\begin{equation*}
\begin{split}
\rho_\gamma^\infty 
 = 
\frac{1}{\lambda_{-} +\lambda_{+}}
\left( \begin{array}{c} 
\lambda_{-}
\\
\lambda_{+}   
\end{array} \right) 
\end{split}
\end{equation*}
or
\begin{equation*}
\begin{split}
\rho_\gamma^\infty & \triangleq
\frac{\lambda_{-}}{\lambda_{-} +\lambda_{+}}\delta_{\gamma_{+}}
+\frac{\lambda_{+}}{\lambda_{-} +\lambda_{+}}\delta_{\gamma_{-}}
\end{split}
\end{equation*}
on $\mathbb{R}$.
An averaging of
\begin{equation*}
\begin{split}
\frac{d{\Gamma}}{dt} = {\gamma}
\end{split}
\end{equation*}
yields
\begin{equation*}
\begin{split}
\frac{d\bar{\Gamma}}{dt} = 
\int \gamma d\rho^\infty_\gamma 
= \frac{\lambda_{-}\gamma_{+}+\lambda_{+}\gamma_{-}}{\lambda_{-} +\lambda_{+}}
\equiv \gamma_\infty. 
\end{split}
\end{equation*}
Let 
\begin{equation*}
\begin{split}
v_0(x,t) &= \mathbb{E}( \phi(\bar{\Gamma}_t) \vert \bar{\Gamma}_0 =x)
\end{split}
\end{equation*}
where
$\phi(\cdot)=\Phi(\cdot,y)$
then
\begin{equation*}
\begin{split}
\partial_t v_0(x,t) ={\gamma}_\infty\partial_x v_0(x,t) = \bar{Q}_1 v_0(x,t)
\end{split}
\end{equation*}
and
Lemma~\ref{thm:weakconv} ensures 
$v(x,y_i,t)\to v_0(x,t)$ as $\epsilon \to 0$.
In this case
the weak convergence of 
$\Gamma_t$ to $\gamma_\infty t$ implies
$\Gamma_T-\Gamma_t \rightharpoonup {\gamma}_\infty(T-t) $
from
Slutsky's theorem stating that
if $X_\epsilon \rightharpoonup X$ and $Y_\epsilon \rightharpoonup Y$ as $\epsilon \to 0$,
where $Y$ is non-random, 
then $X_\epsilon+Y_\epsilon \rightharpoonup X+Y$ as $\epsilon \to 0$.

Let the distribution function of $\Gamma_T-\Gamma_t$ 
be denoted by $F_{\Gamma_T-\Gamma_t}(x) \equiv \mathbb{P}(\Gamma_T-\Gamma_t \leq x)$ then
\begin{equation*}
\begin{split}
F_{\Gamma_T-\Gamma_t}(x) =
\begin{cases}
& 0 \qquad \text{for} \quad x < \gamma_{-} (T-t)\\
& 1 \qquad \text{for} \quad x \geq \gamma_{+} (T-t)
\end{cases}
\end{split}
\end{equation*}
Taking $a<\gamma_{-} (T-t)$ and $b>\gamma_{+} (T-t)$,
Eqs.~(\ref{eq:covcond1}), (\ref{eq:covcond2}) are satisfied.
Note $\Gamma_T-\Gamma_t \rightharpoonup {\gamma}_\infty(T-t)$ 
is equivalent to $F_{\Gamma_T-\Gamma_t}(x) \to F_{{\gamma}_\infty(T-t)}(x)$ 
for every $x$ that is continuity point of
$F_{{\gamma}_\infty(T-t)}$,
given by
\begin{equation*}
\begin{split}
F_{{\gamma}_\infty(T-t)}(x)=
\begin{cases}
& 0 \qquad \text{for} \quad x < {\gamma}_\infty(T-t)\\
& 1 \qquad \text{for} \quad x \geq {\gamma}_\infty(T-t)
\end{cases}
\end{split}
\end{equation*}
from L\'{e}vy-Cram\'{e}r continuity theorem.
Then
Eq.~(\ref{eq:covcond3})
is satisfied from the bounded convergence theorem
and 
Lemma~\ref{lem:mgf} ensures the MGF convergence. 
The convergence rate of
$\left\langle e^{\alpha (\Gamma_T-\Gamma_t)}\right\rangle \to e^{\alpha \gamma_\infty(T-t)}$ 
is equal to the one in
\begin{equation*}
\begin{split}
\lim_{\epsilon \to 0} 
\int^{\gamma_{+}(T-t)}_{\gamma_{-}(T-t)} 
\left( F_{\Gamma_T-\Gamma_t}(x)-F_{\gamma_\infty(T-t)}(x)\right) e^{\alpha x}dx = 0.
\end{split}
\end{equation*}
\end{proof}

\bigskip

\begin{proof}
[of Lemma~\ref{thm:dsmmgf}]
The generator of the system
\begin{equation*}
\begin{split}
\left\{ 
\begin{array}{ll}
d\widehat{\Gamma} &= \widehat{\gamma} dt \\
d\widehat{\gamma} & = -\frac{1}{\epsilon}\nabla U(\widehat{\gamma})dt+
\frac{1}{\sqrt{\epsilon}} \beta( \widehat{\gamma})dB_{\gamma} \\
\end{array} \right. 
\end{split}
\end{equation*}
is given by
\begin{equation*}
\begin{split}
y \partial_x + \frac{1}{\epsilon}
\left( -\nabla U(y)\partial_y +\frac{1}{2}\beta(y)^2\partial_y^2 \right)
=Q_1+\frac{1}{\epsilon}Q_0.
\end{split}
\end{equation*}
If $\widehat{\gamma}$ is an ergodic process 
with invariant measure 
$\rho_{\widehat{\gamma}}^\infty$
then
Lemma~\ref{thm:weakconv} ensures 
$\widehat{\Gamma}(t) \rightharpoonup \overline{\widehat{\Gamma}}(t)$
solving
\begin{equation*}
\begin{split}
\frac{d\overline{\widehat{\Gamma}}}{dt} = 
\int \widehat{\gamma} d\rho_{\widehat{\gamma}}^\infty.
\end{split}
\end{equation*}

In case of DSM,
the generator reads
\begin{equation*}
\begin{split}
y \partial_x + \frac{1}{\epsilon}
\left( -\nu(y-\mu)\partial_y +\frac{\sigma^2}{2}\partial_y^2 \right)
=Q_1+\frac{1}{\epsilon}Q_0
\end{split}
\end{equation*}
and
the invariant measure for $\widehat{\gamma}$ is
\begin{equation*}
\begin{split}
\rho_{\widehat{\gamma}}^\infty & = 
\mathcal{N}\left(\mu, \frac{\sigma^2}{2\nu} \right)
\end{split}
\end{equation*}
because it solves
$Q_0^*\rho_{\widehat{\gamma}}^\infty=0$.
Therefore we obtain
$\widehat{\Gamma}_t\rightharpoonup\mu t$
and 
further
$\widehat{\Gamma}_T-\widehat{\Gamma}_t \rightharpoonup \mu(T-t) $
from 
Slutsky's theorem.

Since $\widehat{\gamma}$ is Gaussian process,
we use Chernoff bound $F_{\mathcal{N}(0,1)}(x) \leq \frac{1}{2}e^{-x^2/2}$
to meet Eqs.~(\ref{eq:covcond1}), (\ref{eq:covcond2}).
Note Eq.~(\ref{eq:covcond3}) is satisfied from bounded convergence theorem.
As a consequence, MGF convergence follows
and the analyis of convergence rate is the same with SSM case.
\end{proof}

\subsection{Rare-Event Limit}
\begin{proof}
[of Lemma~\ref{thm:uconv3}]
It follows from
\begin{equation*}
\begin{split}
\left\langle {u}^\epsilon_t |\gamma^\epsilon_0=\gamma_{\pm}\right\rangle 
& = \left\langle e^{-{\Gamma}^\epsilon_t}{u}^\epsilon_0 |\gamma^\epsilon_0=\gamma_{\pm}\right\rangle 
\\
\text{Var}(u^\epsilon_t|\gamma^\epsilon_0=\gamma_{\pm}) &=
\left\langle \left( e^{-{\Gamma}^\epsilon_t}u_0^\epsilon \right)^2 
|\gamma^\epsilon_0=\gamma_{\pm} \right\rangle
-\left\langle e^{-{\Gamma}^\epsilon_t}u^\epsilon_0 |\gamma^\epsilon_0=\gamma_{\pm}\right\rangle^2\\
& \quad + \sigma_u^2  \int^t_0  \left\langle e^{-2({\Gamma}^\epsilon_t-{\Gamma}^\epsilon_s)}
 |\gamma^\epsilon_0=\gamma_{\pm}\right\rangle ds.
\end{split}
\end{equation*}
\end{proof}

\begin{proof}
[of Lemma~\ref{lem:ssmd}]
We take $\epsilon\to \infty$ in
Eq.~(\ref{eq:ssmcharasy})
and use
Theorem~\ref{thm:unif}.
In view of 
(\ref{eq:rer}),
this yields the case of $t=0$.
We
invoke
(\ref{eq:probgammats}) to complete proof.

\end{proof}

\begin{proof}
[of Lemma~\ref{lem:dsmd}]
We take $\epsilon\to \infty$ to
Eq.~(\ref{eq:dsmdasym})
for the case of $t=0$.
Direct computation of
(\ref{eq:spmcharasy}) leads to the result.

\end{proof}

\end{appendix}

\bibliographystyle{siam.bst}
\bibliography{ssm}

\end{document}